\documentclass[a4paper,10pt]{article}

\usepackage{amsmath}
\usepackage{amsfonts}
\usepackage{amssymb}
\usepackage{mathrsfs}
\usepackage{amscd}
\usepackage{pb-diagram}
\usepackage{amsthm}
\usepackage{color}
\usepackage[all]{xy}
\usepackage{graphicx}
\usepackage{url}

\newcommand{\transposee}[1]{{\vphantom{#1}}^{\mathit t}{#1}} 

\author{Mathieu Molitor
\thanks{
Universidade Federal da Bahia, Instituto de Matem\'{a}tica, Av. Adhemar de Barros, S/N, Ondina, 40170-110 Salvador-BA, Brazil}\\
\small \url{pergame.mathieu@gmail.com}}

\title{Gaussian distributions, Jacobi group and Siegel-Jacobi space} 
\date{}

\begin{document}
\newtheorem{lemma}{Lemma}[section]
\newtheorem{definition}[lemma]{Definition}
\newtheorem{proposition}[lemma]{Proposition}
\newtheorem{corollary}[lemma]{Corollary}
\newtheorem{theorem}[lemma]{Theorem}
\newtheorem{remark}[lemma]{Remark}
\newtheorem{example}[lemma]{Example}
\bibliographystyle{alpha}

\maketitle 

\begin{abstract}
	Let $\mathcal{N}$ be the space of Gaussian distribution functions over $\mathbb{R}$, regarded as a 
	2-dimensional statistical manifold parameterized by the mean $\mu$ and the deviation $\sigma$. In this paper 
	we show that the tangent bundle of $\mathcal{N}$, endowed with its natural K\"{a}hler structure,
	is the Siegel-Jacobi space appearing in the context of Number Theory and Jacobi forms. Geometrical aspects of the Siegel-Jacobi 
	space are discussed in detail (completeness, curvature, group of holomorphic isometries, space of K\"{a}hler functions, 
	relationship to the Jacobi group), and are related to the quantum formalism in its geometrical form, i.e., 
	based on the K\"{a}hler structure of the complex projective space. 

	This paper is a continuation of our previous work \cite{Molitor-quantique,Molitor-hydrodynamical,Molitor-exponential}, 
	where we studied the quantum formalism from a geometric and 
	information-theoretical point of view. 
\end{abstract}

\tableofcontents

\section{Motivation: the quantum formalism}

	It was recently suggested 
	that the quantum formalism might be ``\textit{grounded on the K\"{a}hler geometry which naturally emerges from statistics}" 
	\cite{Molitor-exponential}. What motivates this claim comes from the following facts
	(see also \cite{Molitor-hydrodynamical,Molitor-quantique}). 
	
	There exists a large class of statistical manifolds, called \textit{exponential families} 
	(see Definition \ref{que difnkdgnkdfndk} and \ref{definition exp}), 
	whose tangent bundles possess automatically a K\"{a}hler structure of information-theoretical origin 
	(see Section \ref{section information geometry}). 
	For example, the space $\mathcal{B}(n)$ of binomial distributions $p(k)=\binom{n}{k}q^{k}(1-q)^{n-k}$ 
	defined over $\{0,...,n\}$ forms a 1-dimensional exponential family parameterized by $q\in (0,1)$. Therefore, its 
	tangent bundle is a K\"{a}hler manifold of real dimension 2, and one can show that it is locally 
	isomorphic to the natural K\"{a}hler structure of 
	the sphere $S^{2}$ multiplied by $n$. Another important example is the following. Take a finite set 
	$\Omega:=\{x_{1},...,x_{n}\}$ and consider the space 
	$\mathcal{P}_{n}^{\times}$ of nowhere vanishing probabilities $p\,:\,\Omega\rightarrow \mathbb{R}\,,$ $p>0\,,$ 
	$\sum_{k=1}^{n}\,p(x_{k})=1\,.$ This is a $(n{-}1)$-dimensional exponential family, and it can be shown 
	(see \cite{Molitor-quantique}) that $T\mathcal{P}_{n}^{\times}$ is locally isomorphic to the complex 
	projective space $\mathbb{P}(\mathbb{C}^{n})$ (see also \cite{Molitor-exponential} for a refinement of this statement 
	using the concept of ``K\"{a}hlerification"). 

	Many authors have stressed the importance of K\"{a}hler geometry in relation to the quantum formalism 
	\cite{Cirelli-Hamiltonian,Cirelli-Quantum,Heslot-Une,Heslot-Quantum,Kibble,Spera-geometric}. It is known that 
	a quantum system, with Hilbert space $\mathbb{C}^{n}$, can be entirely described 
	by means of the K\"{a}hler structure of $\mathbb{P}(\mathbb{C}^{n});$ this is the so-called 
	\textit{geometrical formulation of quantum mechanics} \cite{Ashtekar}. Therefore, by recovering the K\"{a}hler structure of 
	$\mathbb{P}(\mathbb{C}^{n})$ from a purely statistical object like $\mathcal{P}_{n}^{\times}$, one may legitimately 
	suspect that the quantum formalism has an information-theoretical origin, at least for finite-dimensional Hilbert spaces. 

	In \cite{Molitor-exponential}, we pursued this line of thought and observed that, in finite dimension, 
	all the ingredients of the geometrical formulation of quantum mechanics (quantum state space, 
	observables, probabilistic interpretation, etc.) can be expressed 
	in terms of the statistical structure of $\mathcal{P}_{n}^{\times}$ (+ completion arguments). 
	This is a crucial observation, for it allows to somewhat enlarge the geometrical formulation of quantum mechanics 
	and gives new geometrical insight. For example, we characterized the so-called 
	\textit{spin coherent states}\footnote{Spin coherent states are a particular case of what physicists call coherent states, 
	historically discovered in 1926 by Schr\"{o}dinger in relation to the quantum harmonic oscillator \cite{Schrodinger}, 
	and later on rediscovered by Glauber \cite{Glauber} who used them to explain coherence phenomena in quantum optics (for example 
	laser light can be thought of as an appropriate coherent state).
	Nowadays, the concept of coherent states has been generalized in various 
	directions, leading to many non-equivalent definitions (see for example \cite{Ali,Combescure,Klauder,Perelomov}).} in terms of 
	the \textit{Veronese embedding} $S^{2}\hookrightarrow \mathbb{P}(\mathbb{C}^{n+1})$, simply 
	by studying the derivative of the canonical injection $\mathcal{B}(n)\hookrightarrow \mathcal{P}_{n+1}^{\times}$ 
	(see \cite{Brody,Molitor-exponential}). 

	It is important to note that the above ``statistical-K\"{a}hler" geometry is not related to 
	quantum mechanics in the same way as symplectic manifolds are related to quantum 
	mechanics via a quantization scheme (e.g. geometric quantization \cite{Kostant,Souriau}). 
	In some sense, the above geometry  is ``quantum" right from the start due to its 
	statistical origin. Let us illustrate this point by the following result (see Corollary \ref{,cf,df,dkf,kf,ekf,ekf,ek}). 
	Let $\mathcal{E}$ be an exponential family (like $\mathcal{B}(n)$ or 
	$\mathcal{P}_{n}^{\times}$) defined over a measure space $(\Omega,dx)$, with canonical projection 
	$\pi\,:\,T\mathcal{E}\rightarrow \mathcal{E}$. Fix an arbitrary holomorphic isometry $\Phi$ of $T\mathcal{E}$. In this situation, 
	it can be shown that there exists a vector space $\mathcal{A}_{\mathcal{E}}$ of random variables 
	$X\,:\,\Omega\rightarrow \mathbb{R}$ such that: (1) $\textup{dim}(\mathcal{A}_{\mathcal{E}})=\textup{dim}(\mathcal{E})+1$, and (2) 
	functions of the form $T\mathcal{E}\rightarrow\mathbb{R},\,\,\,p\mapsto 
	\int_{\Omega}\,X(x)[(\pi\circ \Phi)(p)](x)dx$ are automatically \textit{K\"{a}hler functions}, that is, they preserve 
	the K\"{a}hler structure of $T\mathcal{E}$ (see Definition \ref{fdkjfekgjrkgj}). K\"{a}hler functions are important in relation to the 
	geometrical formulation of quantum mechanics, for they play the role of observables (see \cite{Ashtekar}). 
	The geometrical formalism of quantum mechanics analysed in \cite{Molitor-exponential} under the light of the above 
	``K\"{a}hler decomposition" led naturally to the following definition: the spectrum
	of a K\"{a}hler function $f\,:\,T\mathcal{E}\rightarrow \mathbb{R}$ 
	of the form $\int_{\Omega}\,X(x)[(\pi\circ \Phi)(p)](x)dx$ 
	is $\textup{Spec}(f):=\textup{Im}(X)$, where $\textup{Im}(X)$ is 
	the image of the random variable $X\in \mathcal{A}_{\mathcal{E}}$. Using this definition, we described the spin of a particle 
	passing through two consecutive Stern-Gerlach devices, \textit{without} using physicists' standard approach based on the unitary 
	representations of $\mathfrak{su}(2)$.

	It is on the basis of the above facts (together with others that are collected in 
	\cite{Molitor-hydrodynamical,Molitor-quantique,Molitor-exponential}), that we arrived at the conclusion 
	that the quantum formalism might have 
	an information-theoretical origin. Now there are two possibilities: 
	\begin{enumerate}
	\item the quantum formalism has indeed an information-theoretical origin. In this case, the formalism should be rewritten and 
	the role of the above statistical-K\"{a}hler geometry should be fully clarified. Recently, 
	many authors have tried to derive (or ``reconstruct") the quantum formalism from purely information-theoretical 
	principles \cite{Brukner,Clifton,Grimbaum,Chiribella,Goyal-2008,Goyal-2010,Masanes,Rovelli}. These attempts have their own merits
	and respective successes, but to our knowledge, no consensus has emerged yet.
	\item Quantum mechanics cannot be derived from information-theoretical principles. In this case, one should still 
	explain the relationship between the above definition of $\textup{Spec}(f)$, which is a priori independent of 
	representation theory, and the definition of the spectrum of an operator. 
	It may well be that there is some (obscure) geometrical content hidden behind the main results of functional 
	analysis that goes beyond the well-known 
	correspondence between the space of K\"{a}hler functions of the complex projective space and the space of Hermitian operators 
	(as described for example in \cite{Cirelli-Quantum}, or Lemma 7.6 in \cite{Molitor-exponential}). 
	\end{enumerate}
	In any case, it is necessary to investigate the matter further and to study more examples. 

	In this paper, we investigate an example which for obvious reasons should be particularly important, namely the family $\mathcal{N}$ 
	of Gaussian distribution functions 
	\begin{eqnarray}
		\frac{1}{\sqrt{2\pi}\sigma}\textup{exp}\Big\{-\frac{(x-\mu)^{2}}{2\,\sigma^{2}}\Big\}\,\,\,\,\,\,\,\,\,(x\in \mathbb{R})
	\end{eqnarray}
	defined over $\mathbb{R}$. Clearly, $\mathcal{N}$ is a 2-dimensional statistical manifold parameterized 
	by the mean $\mu\in \mathbb{R}$ and the deviation $\sigma>0$, and it is well-known that it is an exponential family 
	(see Definition \ref{definition exp} and \eqref{equation reecriture normal}). Therefore $T\mathcal{N}$ is 
	naturally a K\"{a}hler manifold of real dimension 4. The objective of this paper is to study the geometry of $T\mathcal{N}$, 
	having in mind quantum mechanics as discussed above. 
	We distinguish two aspects: the \textit{intrinsic geometry} of $T\mathcal{N}$, coming 
	from the fact that $T\mathcal{N}$ is a K\"{a}hler manifold by itself, and the 
	\textit{extrinsic geometry}, related to the fact that $T\mathcal{N}$ can be regarded as a submanifold of an 
	infinite-dimensional complex projective space 
	$\mathbb{P}(\mathcal{H})$. Of these two 
	approaches, it is extrinsic geometry which makes the connection between $T\mathcal{N}$ and 
	the quantum formalism most transparent.  
	
	Let us now describe our results regarding the geometry of $T\mathcal{N}$. \\
	\textbf{The intrinsic geometry.} As a K\"{a}hler manifold, $T\mathcal{N}$ is the \textit{Siegel-Jacobi space} $\mathbb{S}^{J}$ 
	(see Definition \ref{enekngkrgn} and Proposition \ref{dlfld,f}). 
	The Siegel-Jacobi space appears in the context of Number Theory, in relation to the so-called \textit{Jacobi forms} 
	(see \cite{Berndt98,Eichler}). As a complex manifold, it is the product $\mathbb{H}\times \mathbb{C}$, where $\mathbb{H}$ is the 
	Poincar\'{e} upper half-plane $\{\tau\in \mathbb{C}\,\vert\,\textup{Im}(\tau)>0\}$, and its K\"{a}hler metric is the 
	$\textit{K\"{a}hler-Berndt metric}$ $g_{KB}$ (see Definition \ref{cspwpwpdkpw}). Using the general properties of 
	\textit{Dombrowski's construction} (see Section \ref{section dom} and \ref{dlldgklrgf}), 
	we compute the curvature of $T\mathcal{N}$ and observe that 
	the scalar curvature is constant and negative, albeit not Einstein. The group of holomorphic isometries 
	of $T\mathcal{N}$ is computed in Section \ref{dfiejfiejennc}; it is the 
	\textit{affine symplectic group} $\textup{SL}(2,\mathbb{R})\ltimes \mathbb{R}^{2}$ (see Theorem \ref{theoremnnn}). 
	We also describe the whole group 
	of isometries using a result of Kulkarni which characterizes curvature-preserving maps between Riemannian manifolds of 
	dimension $\geq 4$ (see Theorem \ref{ef,kdgj,kdgf,k} and Proposition \ref{ksnfksgnkt}). A few geometrical 
	consequences are derived in Proposition \ref{df,dkfjekrfjkefjkrt}, the most notable being that $T\mathcal{N}$ 
	is a homogeneous K\"{a}hler manifold (a result which was already known for $\mathbb{S}^{J}$). In 
	Section \ref{que direndknkfgnkfg}, we study the space of K\"{a}hler functions on $\mathbb{S}^{J}$. As it turns out, 
	they are conveniently described by means of the \textit{Jacobi group} $G^{J}(\mathbb{R})$, the semi-direct product 
	$\textup{SL}(2,\mathbb{R})\ltimes \textup{Heis}(\mathbb{R})$, where $\textup{Heis}(\mathbb{R})$ is the Heisenberg group 
	of dimension 3 (see Section \ref{que dire?jffndknfkdfnx}). We show that the Jacobi group acts in a Hamiltonian way 
	on $\mathbb{S}^{J}$, and compute the corresponding momentum map $\textup{\textbf{J}}\,:\,\mathbb{S}^{J}\rightarrow 
	(\mathfrak{g}^{J})^{*}$ (here $\mathfrak{g}^{J}$ denotes the Lie algebra of $G^{J}(\mathbb{R})$). We then show that a 
	smooth function $f\,:\,\mathbb{S}^{J}\rightarrow \mathbb{R}$ is K\"{a}hler if and only if there exists $\xi\in \mathfrak{g}^{J}$ 
	such that $f(p)=\langle \textup{\textbf{J}}(p),\xi\rangle$ for all $p\in \mathbb{S}^{J}$, where $\langle\,,\,\rangle$ is the 
	natural pairing between $\mathfrak{g}^{J}$ and $(\mathfrak{g}^{J})^{*}$. From this we deduce that the space of K\"{a}hler 
	functions on $\mathbb{S}^{J}$ is a Poisson algebra of dimension 6, isomorphic in the Lie algebra sense to $\mathfrak{g}^{J}$. 
	We also use Kostant's Coadjoint Orbit Covering Theorem \cite{Kostant-orbits} to deduce that $\mathbb{S}^{J}$ is a coadjoint orbit 
	of $G^{J}(\mathbb{R})$ (see Proposition \ref{que direnkdnkgdgn}). Having 
	quantum mechanics in mind, we then study the spectral properties of the K\"{a}hler functions of $\mathbb{S}^{J}$ in the sense 
	discussed above and in \cite{Molitor-exponential}. The K\"{a}hler functions we consider on $\mathbb{S}^{J}$ are of the form 
	(see Proposition \eqref{ffkelgkrltikelrkflkr}) :
	\begin{eqnarray}\label{dgjrgjrltgk}
		f(p)=\int_{-\infty}^{\infty}\,(\alpha x^{2}+\beta x+\gamma) \big[(\pi\circ \Phi_{g^{-1}})(p)\big](x)dx,
	\end{eqnarray}
	where $p\in \mathbb{S}^{J}$, $\pi\,:\,\mathbb{S}^{J}\cong T\mathcal{N}\rightarrow \mathcal{N}$ is the canonical projection, 
	$\Phi_{g^{-1}}$ is a holomorphic isometry of $\mathbb{S}^{J}$, $dx$ is the Lebesgue measure over $\mathbb{R}$ 
	and where $\alpha x^{2}+\beta x+\gamma$ is a polynomial with real coefficients in the variable $x\in \mathbb{R}$. 
	We define the spectrum $\textup{Spec}(f)$ of a function of this type 
	as the image of the polynomial $\alpha x^{2}+\beta x+\gamma$. In Lemmas \ref{sdçkgvfrikpeopw} and 
	\ref{rdkgjdkfg}, we check that this definition is independent of the decomposition in \eqref{dgjrgjrltgk}. 
	Instances of spectra are given in Example \ref{Que dire? ,f,grkg,tkg,}. Finally, given a point $p\in 
	\mathbb{S}^{J}$ and a K\"{a}hler function $f$ as above, we define a probability measure $P_{f,p}$ on $\textup{Spec}(f)$ 
	as the probability distribution of the polynomial 
	$\alpha x^{2}+\beta x+\gamma$, regarded as a random variable with respect to the probability measure 
	$[(\pi\circ \Phi_{g^{-1}})(p)](x)dx$ (see Lemma \ref{fkjdsjksgjsd} and Definition \ref{eg,,kd,zld,ek}). 
	From a quantum mechanical point of view, 
	the quantity $P_{f,p}(A)$ is interpreted as 
	the probability that the observable $f$ yields upon measurement an ``eigenvalue" 
	$\lambda\in A\subseteq \textup{Spec}(f)$ while the system is in the state $p\in \mathbb{S}^{J}$.  \\
	\textbf{The extrinsic geometry.} Let $\mathcal{H}:=L^{2}(\mathbb{R})$ be the Hilbert space of 
	square-integrable functions $f\,:\,\mathbb{R}\rightarrow \mathbb{C}$ endowed with the Hermitian product 
	$\langle f,g \rangle:=\int_{\mathbb{R}}\,\bar{f}gdx,$ where $dx$ is the Lebesgue measure. Associated to $\mathcal{H}$ is 
	the complex projective space $\mathbb{P}(\mathcal{H})$ of complex lines in $\mathcal{H}$, endowed with its natural 
	K\"{a}hler structure (Fubini-Study symplectic form and metric). In Section \ref{section extrinsic}, we introduce a map 
	$\Psi\,:\,\mathbb{S}^{J}\rightarrow \mathcal{H}$ and its companion map 
	$T:=[\Psi]\,:\,\mathbb{S}^{J}\rightarrow \mathbb{P}(\mathcal{H})$ having the following properties. The map $T$ is 
	a smooth and symplectic immersion, \textit{but} it is not isometric nor holomorphic (see Proposition \ref{fejfkegjkegfe4jk}). 
	Moreover, it gives the following characterization (see Proposition \ref{ed;fndjfkdfejfn}): 
	a smooth function $f\,:\,\mathbb{S}^{J}\rightarrow \mathbb{R}$ is K\"{a}hler if and only if $f$ can be written as
	\begin{eqnarray}
		f(p)=\big\langle \Psi(p),H\Psi(p)\big\rangle,\,\,\,\,\,\,\,\,(p\in \mathbb{S}^{J})
	\end{eqnarray}
	where $H$ is a real linear combination of the following Hermitian operators acting on $C^{\infty}(\mathbb{R},\mathbb{C})$ : 
	\begin{eqnarray}
		-x^{2},\,\,\,-i\frac{\partial}{\partial x},\,\,\,-\frac{\partial^{2}}{\partial x^{2}},\,\,\,x,\,\,\,
		2i\Big(x\frac{\partial}{\partial x}+\frac{1}{2}I\Big),\,\,\,
		I
	\end{eqnarray}
	($I$ denotes the identity operator). The precise statement 
	involves a unitary representation of the Lie algebra $\mathfrak{g}^{J}$ which is essentially the infinitesimal 
	\textit{Schr\"{o}dinger-Weil representation} (see \cite{Berceanu08}). Finally in Section \ref{final sectionnnn}, 
	we discuss briefly the Schr\"{o}dinger equation
	\begin{eqnarray}\label{ekdf,ekgf,rkgn}
		i\dfrac{d\psi}{dt}=H\psi,\,\,\,\,\,\,\,\,\,(\psi\in L^{2}(\mathbb{R}))
	\end{eqnarray}
	where $H$ is a linear combination of the above Hermitian operators. More precisely, given a K\"{a}hler function $f$ on $\mathbb{S}^{J}$ 
	with Hamiltonian vector field $X_{f}$, we observe that if $\alpha\,:\,I\rightarrow \mathbb{S}^{J}$ is an integral curve of $X_{f}$, 
	then there exists a smooth map $\lambda\,:\,I\rightarrow \mathbb{C}{-}\{0\}$ such that 
	$\lambda(t)\Psi\big(\alpha(t)\big)$ satisfies the above Schr\"{o}dinger equation for an appropriate $H$ 
	(see Corollary \ref{sfndkfnkdfnkd}). From a physical point of view, the above operators are related to 
	the free quantum particle, the quantum harmonic oscillator and the forced quantum harmonic hoscillator 
	(see Remark \ref{dkefkdgjrkgjrk}).  
	
	Let us comment the above results. Clearly, the main observation of this paper is the connection 
	between the space of Gaussian distributions, the Siegel-Jacobi space and the Jacobi group. Using the 
	terminology introduced in \cite{Molitor-exponential}, one may say that 
	\textit{the K\"{a}hlerification of the space of Gaussian distributions is the Siegel-Jacobi space}. 
	
	As we already mentioned, the Siegel-Jacobi space and Jacobi group play an 
	important role in the context of Number Theory, in relation to Jacobi forms \cite{Eichler,Berndt98}. The latter are 
	a mixture of modular forms and elliptic functions that generalize classical functions like the \textit{Jacobi theta function} 
	and the Fourier coefficients of the \textit{Siegel modular forms} \cite{Zagier88}. Roughly, they are holomorphic functions $f$ on 
	$\mathbb{H}\times \mathbb{C}$ enjoying invariance properties that involve the Jacobi group $G^{J}(\mathbb{R})$, 
	together with ``good" Fourier expansions (see also Remark \ref{snzjfekdjfnn} for more details on the role of 
	the K\"{a}hler-Berndt metric). In the context of physics, the Jacobi group, also known as 
	the \textit{Schr\"{o}dinger} or \textit{Hagen group}, is the symmetry group of the one-dimensional 
	Schr\"{o}dinger equation of a free quantum particle \cite{Hagen,Niederer}. In the context of quantum optics, 
	the Jacobi group is related to the so-called \textit{squeezed coherent states} 
	\cite{Berceanu07,Berceanu08,Berceanu-squeezed,Berceanu111,Berceanu11,Berceanu14,Dodonov,Drummond,Sivakumar,Zhang}. 

	It is somehow surprising that with so little, the Gaussian distribution, one can arrive at important 
	objects like the Siegel-Jacobi space and the Jacobi group, and discuss a fair amount of their quantum properties without any 
	quantization scheme (especially in view of the intrinsic geometry). 
	This reassures us and lends credence to the idea that the above statistical-K\"{a}hler geometry 
	is one of the keys to understand the foundations of quantum physics.

	There are however two important questions which are not discussed in this paper: 
	(1) what is the origin of the map $T\,:\,\mathbb{S}^{J}\rightarrow \mathbb{P}(\mathcal{H})$, 
	and (2) what are its equivariance properties? 
	In \cite{Molitor-exponential} we observed that the Veronese embedding $S^{2}\hookrightarrow \mathbb{P}(\mathbb{C}^{n+1})$, which 
	is a finite-dimensional analogue\footnote{To see this, compare Section \ref{section extrinsic} 
	with \cite{Molitor-exponential}.} of $T$, is essentially the derivative of the inclusion map $\mathcal{B}(n)\hookrightarrow 
	\mathcal{P}_{n+1}^{\times}$ (neglecting completion issues, it is the derivative up to the actions of 
	two discrete groups). In the case 
	of $T$, such interpretation is not directly available for the following reason. Let $\mathcal{D}$ be 
	the space of smooth density probability functions over $\mathbb{R}$ with respect to the Lebesgue measure. The space $\mathcal{D}$ 
	can be thought of as an infinite-dimensional analogue of $\mathcal{P}_{n}^{\times},$ but contrary to the latter, its tangent 
	bundle $T\mathcal{D}$ does not have a canonical K\"{a}hler structure that could be ``compared" with that of 
	$\mathbb{P}(L^{2}(\mathbb{R}))$. Therefore, the derivative of the inclusion map 
	$\mathcal{N}\hookrightarrow \mathcal{D}$ cannot be interpreted directly as a map 
	$T\mathcal{N}\rightarrow \mathbb{P}(L^{2}(\mathbb{R}))$. To overcome these difficulties, 
	it is necessary to first get a clear idea of what should be the infinite dimensional generalization of 
	the statistical-K\"{a}hler geometry discussed above; the papers \cite{Friedrich,Khesin,Modin,Molitor-hydrodynamical} 
	might be a good starting point in this respect. Regarding the second question, we observe that 
	$T$ exhibits properties that are usually shared 
	by coherent states (compare for example Proposition \ref{ed;fndjfkdfejfn} and Corollary \ref{sfndkfnkdfnkd} 
	with \cite{Berceanu-Schlichi,Perelomov,Rawnsley,Spera-coherent}). Moreover, $T$ is an infinite-dimensional analogue 
	of the Veronese embedding, which is known to characterize spin coherent 
	states \cite{Brody,Molitor-exponential}. Therefore it is very likely that $T$ itself is a coherent state in the 
	sense of Perelomov \cite{Perelomov}. To prove this, one should 
	establish equivariance properties of the map $T$, probably by means of the Schr\"{o}dinger-Weil 
	representation \cite{Berndt98}. It is interesting to note, in this respect, that Yang 
	considered in \cite{Yang-Weill,Yang-Weil} a map $\mathbb{S}^{J}\rightarrow L^{2}(\mathbb{R})$ which is 
	very similar to $\Psi$, and which enjoys such equivariance properties. It would be very 
	interesting to relate Yang's work to the properties of $T$, and then make a comparison with the coherent-state approach of 
	Berceanu \cite{Berceanu07,Berceanu08,Berceanu-Schlichi,Berceanu111,Berceanu11,Berceanu14}.  \\

	For the convenience of the reader, the paper starts  
	with a rather detailed discussion on the relation between K\"{a}hler geometry and statistics (see 
	Section \ref{dkdnfkfnrkgfvnrk}). Some of these results are known 
	(Proposition \ref{corollary c'est kahler!}, Proposition \ref{lxlsjfojfgingd}, Corollary \ref{xql,cnnn}, 
	Proposition \ref{dkf,sd,ls;dlz}, Proposition \ref{proposition proprietes fam exp}, Corollary \ref{corollary encore que dire?}),  
	others are new (Propositions \ref{dsmnkldsgmsa}, \ref{c kdlfldfkdl}, \ref{ldklsdklsgvnekg}, \ref{fpelpçeglv}), 
	others still appear in different contexts and different guises
	(Propositions \ref{fzlfkcslkflef} and \ref{gfsgag}, Corollary \ref{pouetteeee}).
	We shall present the subject in a uniform way by using the concept of
	\textit{dually flat structure}, with which not all reader may be familiar\footnote{Let us mention
	that an alternative description of (some parts of) the material presented in Section \ref{dkdnfkfnrkgfvnrk} 
	can be found in the book of Shima \cite{Shima}, through the concept of \textit{Hessian manifold}.}.
	The intrinsic and extrinsic geometry of $T\mathcal{N}$ are discussed 
	in Section \ref{dnfkefnekfnekfn} and Section \ref{section extrinsic}, respectively.

\section{Dually flat structures and K\"{a}hler geometry}\label{dkdnfkfnrkgfvnrk}
\subsection{Dombrowski's construction}\label{section dom}
	Let $M$ be a manifold endowed with an affine connection $\nabla$. We denote by $\pi\,:\,TM\rightarrow M$ the canonical projection
	and by $K$ the connector associated to $\nabla.$ Recall that $K$ is the unique map $T(TM)\rightarrow TM$ satisfying 
	(see \cite{Dombrowski,Lang,Michor-topics})
	\begin{eqnarray}
		\nabla_{X}Y=KY_{*}X
	\end{eqnarray}
	for all vector fields $X,Y$ on $M$ (here $Y_{*}X$ denotes the derivative of $Y$ in the direction of $X$). 
	
	Given $u_{p}\in T_{p}M,$ the subspaces
	\begin{alignat}{1}
		\textup{Hor}(TM)_{u_{p}} \quad:=&\quad \big\{Z\in T_{u_{p}}(TM)\,\big\vert\,KZ=0\big\},\\
		\textup{Ver}(TM)_{u_{p}} \quad:=&\quad \big\{Z\in T_{u_{p}}(TM)\,\big\vert\,\pi_{*_{u_{p}}}Z=0\big\},
	\end{alignat}
	are respectively called the \textit{space of horizontal tangent vectors} and 
	the \textit{space of vertical tangent vectors} of $TM$ at $u_{p}$. They are both isomorphic to $T_{p}M$ in a natural way, 
	and led to the following decomposition:  
	\begin{eqnarray}
		T_{u_{p}}(TM)\cong \textup{Hor}(TM)_{u_{p}}\oplus \textup{Ver}(TM)_{u_{p}}\cong T_{p}M\oplus T_{p}M.
	\end{eqnarray}
	More generally, $\nabla$ determines an isomorphism of vector bundles over $M$ (see \cite{Dombrowski,Lang}):
	\begin{eqnarray}
		T(TM)\cong TM\oplus TM\oplus TM,
	\end{eqnarray}
	the isomorphism being
	\begin{eqnarray}\label{equation Dombrowski}
	T_{u_{p}}(TM)\ni A_{u_{p}} \mapsto
	\big(u_{p},\pi_{*_{u_{p}}}A_{u_{p}},K A_{u_{p}}\big).
	\end{eqnarray}
	If there is no danger of confusion, we shall thus regard 
	an element of $T_{u_{p}}(TM)$ as a triple $(u_{p},v_{p},w_{p})$, where $u_{p},v_{p},w_{p}\in T_{p}M.$
	The second component $v_{p}$ is usually referred to as the horizontal component (with respect 
	to $\nabla$) and $w_{p}$ the vertical component. 

	Let $h$ be a Riemannian metric on $M$. Together with $\nabla$, the couple $(h,\nabla)$ determines 
	an almost Hermitian structure on $TM$ via the following formulas:
	\begin{alignat}{5}\label{equation definition G, omega, etc.}
		g_{u_{p}}\big(\big(u_{p},v_{p},w_{p}\big),
			\big({u}_{p},\overline{v}_{p},
			\overline{w}_{p}\big)\big)\quad:=&\quad
			h_{p}\big(v_{p},\overline{v}_{p}\big)+
			h_{p}\big(w_{p},\overline{w}_{p}\big)\,,&
			(\textup{metric})&\nonumber\\
		\omega_{u_{p}}\big(\big(u_{p},v_{p},w_{p}\big),
			\big({u}_{p},\overline{v}_{p},
			\overline{w}_{p}\big)\big)\quad:=&\quad h_{p}\big(v_{p},\overline{w}_{p}\big)-
			h_{p}\big(w_{p},\overline{v}_{p}\big)\,,&
			(\textup{2-form})&\nonumber\\
		J_{u_{p}}\big(\big(u_{p},v_{p},w_{p}\big)\big)\quad:=&\quad
		\big(u_{p},-w_{p},v_{p}\big)\,,&
			(\textup{almost complex structure})&
	\end{alignat}
	where $u_{p},v_{p},w_{p},\overline{v}_{p},\overline{w}_{p}
	\in T_{p}M\,.$ Clearly, $J^{2}=-\textup{Id}$  and $g(J\,.\,,J\,.\,)=g(\,.\,,\,.\,),$ which means that 
	$(TM,g,J)$ is an almost Hermitian manifold, and one readily sees that 
	$g,J$ and $\omega$ are compatible, i.e., that $
	\omega=g\big(J\,.\,,\,.\,\big).$ The $2$-form $\omega$ is thus the fundamental 2-form of 
	the almost Hermitian manifold $(TM,g,J).$ This is \textit{Dombrowski's construction}.
\begin{remark}\label{fçekflrekr}
	By construction, the map $\pi\,:\,(TM,g)\rightarrow (M,h)$ is a Riemannian submersion. 
\end{remark}
\begin{remark}\label{sfksfjkd}
	Let $\gamma(t)$ be a smooth curve in $TM.$ Regarding $\gamma(t)$ as vector field $V(t)$ along 
	$c(t):=(\pi\circ \gamma)(t)$, one has $\pi_{*}\dot{\gamma}=\dot{c}$ 
	and $K\dot{\gamma}=\nabla_{\dot{c}}V,$ where $\dot{\gamma}$ and $\dot{c}$ are 
	the time derivatives of $\gamma$ and $c$ respectively, and 
	where $\nabla_{\dot{c}}V$ is the covariant derivative of $V(t)$ along $c(t).$ 
	From this, it follows by inspection of Dombrowski's construction that 
	\begin{eqnarray}
		g_{\gamma(t)}(\dot{\gamma},\dot{\gamma})=h_{c(t)}(\dot{c},\dot{c})+h_{c(t)}\big(\nabla_{\dot{c}}V,\nabla_{\dot{c}}V\big).
	\end{eqnarray}
\end{remark}
	We now review the analytical properties of Dombrowski's construction. Let $\nabla^{*}$ be the unique connection on $M$ satisfying
	\begin{eqnarray}\label{pouette pouette}
		X\big(h(Y,Z)\big) = h\big(\nabla_{X}Y,Z\big)+h\big(Y,\nabla^{*}_{X}Z\big),
	\end{eqnarray}
	for all vector fields $X,Y,Z$ on $M.$ In the statistical literature, $\nabla^{*}$ is called the 
	\textit{dual connection} of $\nabla$ with respect to $h$ (and vice versa), and the triple $(h,\nabla,\nabla^{*})$ is called 
	a \textit{dualistic structure} (see \cite{Amari-Nagaoka}). 
\begin{definition}
	The dualistic structure $(h,\nabla,\nabla^{*})$ is 
	\textit{dually flat} is both $\nabla$ and $\nabla^{*}$ are flat, meaning that their torsions and curvature tensors are zero. 
\end{definition}
	
	As the literature is not uniform, let us agree that the torsion $T$ and the curvature tensor $R$ of a connection 
	$\nabla$ are defined as
	\begin{eqnarray}
	T(X,Y)&:=&\nabla_{X}Y-\nabla_{Y}X-[X,Y]\,,\nonumber\\
	R(X,Y)Z&:=&\nabla_{X}\nabla_{Y}Z-\nabla_{Y}\nabla_{X}Z-\nabla_{[X,Y]}Z\,,
	\end{eqnarray}
	where $X,Y,Z$ are vector fields on $M.$ 
\begin{remark}\label{les curvaturessss}
	Let $R$ and $R^{*}$ be the curvature tensors of the dual connections $\nabla$ and $\nabla^{*}$ respectively. Then, 
	\begin{eqnarray}
		h\big(R(X,Y)Z,W\big)=-h\big(R^{*}(X,Y)Z,W\big)
	\end{eqnarray}
	for all vector fields $X,Y,Z,W$ on $M$ (see \cite{Amari-Nagaoka}). In particular, $R$ is identically zero if and only of $R^{*}$ 
	is identically zero. 
\end{remark}
	Recall that an almost Hermitian structure $(g,J,\omega)$ is K\"{a}hler when the following two analytical conditions are met: 
	(1) $J$ is integrable; (2) $d\omega=0.$ 
\begin{proposition}\label{corollary c'est kahler!}
	Let $(h,\nabla,\nabla^{*})$ be a dualistic structure on $M$ and 
	$(g,J,\omega)$ the almost Hermitian structure on $TM$ associated to $(h,\nabla)$ via 
	Dombrowski's construction. Then, 
	\begin{eqnarray}
		(TM,g,J,\omega)\,\,\,\textup{is K\"{a}hler}\,\,\,\,\,\Leftrightarrow\,\,\,\,\,
		(M,h,\nabla,\nabla^{*})\,\,\,\textup{is dually flat.}
	\end{eqnarray}
\end{proposition}
\begin{remark}\label{gdfkledkrler}
	Proposition \ref{corollary c'est kahler!} is an easy consequence of Remark \ref{les curvaturessss} together with 
	the following equivalence which is due to Dombrowski (see \cite{Dombrowski,Molitor-exponential}):
	\begin{eqnarray}\label{fesfdfse}
		J\,\,\,\textup{is integrable}\,\,\,\,\,\Leftrightarrow\,\,\,\,\,
		\nabla\,\,\,\textup{is flat}
	\end{eqnarray}
	(here $J$ is the almost complex structure associated to $(h,\nabla)$ via Dombrowski's construction). 
\end{remark}
\subsection{Local formulas}\label{felg,rkg,rkf}
	Let $(h,\nabla,\nabla^{*})$ be a dualistic structure on a manifold $M.$ We denote by $(g,J,\omega)$ the almost Hermitian structure of 
	$TM$ associated to $(h,\nabla)$ via Dombrowski's construction. 
	We also denote by $\pi\,:\,TM\rightarrow M$ the canonical projection 
	and by $K\,:\,T(TM)\rightarrow TM$ the connector associated to $\nabla$. 
	
	Let $x=(x_{1},...,x_{n})$ be  system of coordinates on $M.$ If $dx_{i}$ denotes the differential of 
	$x_{i}$ (regarded as a local function on $TM$), then $(x_{1}\circ \pi,...,x_{n}\circ \pi,dx_{1},...,dx_{n})$ 
	forms a local coordinate system on $TM.$ By repeating, we obtain coordinates on $T(TM),$ say $(a_{i},b_{i},c_{i},d_{i})$, $i=1,...,n,$ 
	where 
	\begin{eqnarray}
		a_{i}=x_{i}\circ \pi\circ \pi_{TM},\,\,\,b_{i}=(dx_{i})\circ \pi_{TM},\,\,\,\,c_{i}=d(x_{i}\circ \pi),\,\,\,
		d_{i}=d(dx_{i}),
	\end{eqnarray}
	and where $\pi_{TM}\,:\,T(TM)\rightarrow TM$ is the canonical projection. Observe that $d_{i}$ is not zero, for $dx_{i}$ is regarded 
	as a local function on $TM,$ not as a one form.

	Let $\Gamma_{ij}^{k}$ be the Christoffel symbols of $\nabla$ in the coordinates $(x_{1},...,x_{n}),$ i.e., 
	\begin{eqnarray}
		\nabla_{\partial_{i}}\partial_{j}=\sum_{k=1}^{n}\,\Gamma_{ij}^{k}\partial_{k},
	\end{eqnarray}
	where $\partial_{i}=\frac{\partial}{\partial x_{i}}.$ In the coordinates introduced above, 
	one can check that 
	\begin{alignat}{3}
		K(a,b,c,d)\,\,=&\,\,\big(a,d+\Gamma_{a}(b,c)\big),\\
		\pi_{*}(a,b,c,d)\,\,=&\,\,(a,c),
	\end{alignat}
	where $\Gamma_{a}$ is the bilinear map 
	$\mathbb{R}^{n}\times \mathbb{R}^{n}\rightarrow \mathbb{R}^{n}$ defined by 
	$\big(\Gamma_{a}(b,c)\big)_{k}=\sum_{i,j=1}^{n}\,\Gamma_{ij}^{k}(a)b_{j}c_{i},$ $k=1,...,n.$ 
	Observe that if $(x_{i})$ is an affine coordinate system\footnote{Recall that 
	a coordinate system is \textit{affine} with respect to a flat connection if all the Christoffel symbols vanish. In this case, 
	the system of coordinates is called an \textit{affine coordinate system}. If a connection is flat, then there exists 
	an affine coordinate system around each point (see for example \cite{Shima}).} 
	with respect to $\nabla$, then $K$ reduces to the projection 
	$(a,b,c,d)\mapsto (a,d).$  

	Let us fix a coordinate system $(y_{i})$ on $M$, defined on the same neighborhood as $(x_{i})$. 
\begin{definition}
	The couple $((x_{i}),(y_{i}))$ is a \textit{pair of dual coordinate systems} 
	if :
	\begin{description}
	\item[$(i)$] $(x_{i})$ (resp. $(y_{i})$) is an affine coordinate system with respect to $\nabla$ (resp.\ $\nabla^{*}$), 
	\item[$(ii)$]  \label{csfkolgjod} 
	$h\big(\tfrac{\partial}{\partial x_{i}},\tfrac{\partial}{\partial y_{j}}\big)= \delta_{ij}$ (Kronecker symbol)
	for all $i,j\in \{1,...,n\}$. 
	\end{description}
	The system of coordinates $(y_{i})$ is called the dual coordinate 
	system of $(x_{i})$, and vice versa.  
\end{definition}
\begin{remark}
		If $(x_{i})$ is an affine coordinate system with respect to $\nabla$, then one can find 
		a coordinate system $(y_{i})$ dual to $(x_{i})$, i.e.\
		 such that $(y_{i})$ is affine with respect to $\nabla^{*}$ and such that 
		 $h\big(\tfrac{\partial}{\partial x_{i}},\tfrac{\partial}{\partial y_{j}}\big)= \delta_{ij}$ (see \cite{Amari-Nagaoka,Shima}).
\end{remark}
\begin{remark}\label{ekkrgjekfjke}
	If $x=(x_{i})$ and $y=(y_{i})$ are dual to each other, then 
	the $n\times n$ matrices $h_{ij}:=h(\tfrac{\partial}{\partial x_{i}},\tfrac{\partial}{\partial x_{j}})$ 
	and $h^{ij}:=h(\tfrac{\partial}{\partial y_{i}},\tfrac{\partial}{\partial y_{j}})$ are inverse to each other,
	and the following relations hold : $\tfrac{\partial x_{i}}{\partial y_{j}}=h^{ij}$ and 
	$\tfrac{\partial y_{j}}{\partial x_{i}}=h_{ij}$ 
	(see \cite{Amari-Nagaoka}).
\end{remark}

	Throughout this paper, we shall write $(x_{1},...,x_{n},\dot{x}_{1},...,\dot{x}_{n})=(x_{i},\dot{x_{i}})$ 
	instead of $(x_{i}\circ \pi,dx_{i})$ for simplicity. We shall also use the ``hybrid" coordinate system $(y_{1},...,y_{n},
	\dot{x}_{1},...,\dot{x}_{n})=(y_{i},\dot{x}_{i})$. Thus by definition, 
	\begin{eqnarray}
		\left \lbrace
		\begin{array}{ccc}
			(x,\dot{x})(v)&:=&
			(x_{1}(p),...,x_{n}(p),a_{1},...,a_{n}),\\
			(y,\dot{x})(v)&:=&
			(y_{1}(p),...,y_{n}(p),a_{1},...,a_{n}), 
		\end{array}
		\right.\,\,\,\,\,\,\textup{where}\,\,\,\,\,\,\,v=a_{1}\frac{\partial}{\partial x_{1}}\bigg\vert_{p}+...
		+a_{n}\frac{\partial}{\partial x_{n}}\bigg\vert_{p}\in T_{p}M. 
	\end{eqnarray}
\begin{proposition}\label{fzlfkcslkflef}\label{proposition les formes locales}
	Let $(h,\nabla,\nabla^{*})$ be a dually flat structure on a manifold $M$ and let $(g,J,\omega)$ be the K\"{a}hler structure on 
	$TM$ associated to $(h,\nabla)$ via Dombrowski's construction. Let also $(x_{i})$ and $(y_{i})$ be two coordinate systems on 
	$M$ dual to each other. Then locally,
	\begin{description}
	\item[$(i)$] in the coordinates $(x_{i},\dot{x}_{i}),$
	\begin{eqnarray}\label{gsgagkwomfkenfm}
		g=\begin{bmatrix}
			h_{ij} & 0\\
			0 & h_{ij}
		\end{bmatrix},\,\,\,\,\,\,
		J=\begin{bmatrix}
			0 & -I \\
			I &  0
		\end{bmatrix},\,\,\,\,\,\,
		\omega=\begin{bmatrix}
			0 & h_{ij} \\
			-h_{ij}& 0
		\end{bmatrix},
	\end{eqnarray}
	where $h_{ij}=h(\tfrac{\partial}{\partial x_{i}},\tfrac{\partial}{\partial x_{j}})$, $i,j\in \{1,...,n\},$ 	 
	\item[$(ii)$] in the coordinates $(y_{i},\dot{x}_{i})$,
		\begin{eqnarray}\label{dlf,lefklef}
			g=\begin{bmatrix}
					h^{ij}& 0\\
					0&h_{ij}
				\end{bmatrix},\,\,\,\,\,\,
			J=
					\begin{bmatrix}
					0&-h_{ij}\\
					h^{ij}&0
				\end{bmatrix},\,\,\,\,\,
			\omega=
					\begin{bmatrix}
					0&I\\
					-I&0
	\end{bmatrix},
	\end{eqnarray}
	where $h^{ij}:=h(\tfrac{\partial}{\partial y_{i}},\tfrac{\partial}{\partial y_{j}}),$ $i,j\in\{1,...,n\}.$
	\end{description}
\end{proposition}
\begin{proof}[Proof of Proposition \ref{fzlfkcslkflef}]
	$(i)$ Follows from Dombrowski's construction (see \eqref{equation definition G, omega, etc.}) taking 
	into account: (1) the explicit form of the isomorphism $T(TM)\rightarrow TM\oplus TM\oplus TM$ given in 
	\eqref{equation Dombrowski}; (2) the formulas $K(a,b,c,d)=(a,d)$ and $\pi_{*}(a,b,c,d)=(a,c)$.\\
	$(ii)$ One has $(x,\dot{x})\circ (y,\dot{x})^{-1}=(x\circ y^{-1},\dot{x})$ and $\tfrac{\partial x_{i}}{\partial y_{j}}=h^{ij}$ 
	(see Remark \ref{ekkrgjekfjke}). Thus, the differential of 
	$(x,\dot{x})\circ(y,\dot{x})^{-1}$ is given by 
	\begin{eqnarray}
		\big[(x,\dot{x})\circ(y,\dot{x})^{-1}\big]_{*}=
				\begin{bmatrix}
					h^{ij}& 0\\
					0& I
				\end{bmatrix}.
	\end{eqnarray}
	From this together with the formula $h_{ij}h^{ij}=I$, 
	one sees that the matrix representation of $g$ in the coordinates $(y,\dot{x})$ is:  
	\begin{eqnarray}
		\transposee{\begin{bmatrix}
					h^{ij}& 0\\
					0& I
				\end{bmatrix}}
			\begin{bmatrix}
					h_{ij}& 0\\
					0&h_{ij}
				\end{bmatrix}
			\begin{bmatrix}
					h^{ij}& 0\\
					0& I
				\end{bmatrix}
			&=&\begin{bmatrix}
					h^{ij}& 0\\
					0& I
				\end{bmatrix}
			\begin{bmatrix}
					h_{ij}h^{ij}& 0\\
					0&h_{ij}
				\end{bmatrix}=\begin{bmatrix}
					h^{ij}& 0\\
					0&h_{ij}
				\end{bmatrix}
	\end{eqnarray}
	(the superscript ``$t$'' means that we take the transpose of the corresponding matrix). The matrix representations 
	of $J$ and $g$ are obtained similarly. The proposition follows. 
\end{proof}
	By inspection of \eqref{gsgagkwomfkenfm} and \eqref{dlf,lefklef}, one sees that:
	\begin{description}
	\item[$\bullet$]
	If $\nabla$ is flat (which means that $J$ is integrable, see Remark \ref{gdfkledkrler}), 
	and if $(x_{i})$ is an affine coordinate system with respect to $\nabla$, then 
	\begin{eqnarray}\label{ldmzdmzmzmm}
		(z_{1},...,z_{n}):=(x_{1}+i\dot{x}_{1},...,x_{n}+i\dot{x}_{n})
	\end{eqnarray}
	are holomorphic coordinates on the complex manifold $(TM,J)$. To see this, compare \eqref{gsgagkwomfkenfm} with, for example, the 
	first chapter in \cite{Moroianu}.  
	\item[$\bullet$] If $(x_{i})$ and $(y_{i})$ are dual to each other, than $(y_{i},\dot{x}_{i})$ are symplectic coordinates on $TM$, 
	that is, $(y,\dot{x})$ it is a Darboux chart for the symplectic manifold $(TM,\omega).$
\end{description}
\begin{remark}	
	In the context of toric K\"{a}hler geometry, Abreu established formulas similar to 
	\eqref{gsgagkwomfkenfm} and \eqref{dlf,lefklef} in order to get symplectic coordinates on 
	toric manifolds (see \cite{Abreu}). Abreu doesn't use 
	the language of dually flat manifolds; instead, he focuses on the so-called \textit{Guillemin potential} and its associated Hessian 
	metric, in a spirit close to \cite{Shima}. 
\end{remark}
\subsection{Ricci curvature}\label{dlldgklrgf}
	Let $N$ be a K\"{a}hler manifold with K\"{a}hler metric $g.$ We denote by $\textup{Ric}$ the Ricci tensor of $g$:
	\begin{eqnarray}
		{\textup{Ric}}(X,Y):=\textup{Trace}\big\{Z\mapsto {R}(Z,X)Y\big\},
	\end{eqnarray}
	where $X,Y,Z$ are vector fields on $N$, and where $R$ is the curvature tensor of $g$. 
	
	On the complexified tangent bundle 
	$TN^{\mathbb{C}}=TN\otimes_{\mathbb{R}}\mathbb{C}$, we extend $\mathbb{C}$-linearly every tensor, using 
	the superscript $``\mathbb{C}"$ to distinguish the corresponding extensions ($g^{\mathbb{C}}$, $\textup{Ric}^{\mathbb{C}}$, etc.). 
	
	Regarding local computations and indices, Greek indices $\alpha,\beta,\gamma$ shall run over $1,...,n$ 
	while capital letters $A,B,C,...$ shall run over $1,...,n,\bar{1},...,\bar{n}$. Let $(z_{1},...,z_{n})$ be a system 
	of complex coordinates on $N$. If $x_{\alpha}$ and $y_{\alpha}$ are respectively the real part and the imaginary part of 
	$z_{\alpha}$ (i.e.\ $z_{\alpha}=x_{\alpha}+iy_{\alpha}$), then fiberwise, the vectors 
	\begin{eqnarray}
		 \dfrac{\partial}{\partial z_{\alpha}}:=\dfrac{1}{2}\Big\{\dfrac{\partial}{\partial x_{\alpha}}-
			i\dfrac{\partial}{\partial y_{\alpha}}\Big\},\,\,\,\,\,\,\,\,\,\,\,\,
		\dfrac{\partial}{\partial \bar{z}_{\alpha}}:=\dfrac{1}{2}\Big\{\dfrac{\partial}{\partial x_{\alpha}}+
		i\dfrac{\partial}{\partial y_{\alpha}}\Big\},
	\end{eqnarray}
	form a basis for $TN^{\mathbb{C}}$. Let $\textup{Ric}^{\mathbb{C}}_{AB}$ be 
	the components of $\textup{Ric}^{\mathbb{C}}$ in this basis, i.e., 
	\begin{eqnarray}
		\textup{Ric}^{\mathbb{C}}_{AB}:=\textup{Ric}^{\mathbb{C}}(Z_{A},Z_{B}),\,\,\,\,\,\,\,\textup{where}\,\,\,\,\,\,\,
		Z_{\alpha}:=\frac{\partial}{\partial z_{\alpha}}\,\,\,\,\,\,\,\textup{and}\,\,\,\,\,\,\,Z_{\bar{\alpha}}=
	\frac{\partial}{\partial \bar{z}_{\alpha}}.
	\end{eqnarray}
	As it is well-known, these 
	components are elegantly expressed via the following formulas (see \cite{Kobayashi-Nomizu-2,Moroianu}) : 
	\begin{eqnarray}\label{ezknkstjqkj}
		\textup{Ric}^{\mathbb{C}}_{\alpha\beta}=\textup{Ric}^{\mathbb{C}}_{\bar{\alpha}\bar{\beta}}=0,\,\,\,\,\,\,\,\,\,\,\,\,
		\textup{Ric}^{\mathbb{C}}_{\bar{\alpha}\beta}=\overline{\textup{Ric}^{\mathbb{C}}_{\alpha\bar{\beta}}},\,\,\,\,\,\,\,\,\,\,\,\,
		\textup{Ric}^{\mathbb{C}}_{\alpha\bar{\beta}}=-\dfrac{\partial^{2}\ln d}{\partial z_{\alpha} \partial z_{\bar{\beta}}},
	\end{eqnarray}
	where $d$ is the determinant of the matrix $g^{\mathbb{C}}_{\alpha\bar{\beta}}=g^{\mathbb{C}}(Z_{\alpha},Z_{\bar{\beta}})$.

	We now specialize to the case $N=TM$, assuming that $g$ is the K\"{a}hler metric
	associated to a dually flat structure $(h,\nabla,\nabla^{*})$ on a $M$ via Dombrowski's construction. 
	
	Fix an affine coordinate 
	system $(x_{1},...,x_{n})$ with respect to $\nabla$, and denote by $(x_{\alpha},\dot{x}_{\alpha})$ the corresponding 
	coordinates on $TM$, as defined in the previous section. If $z_{\alpha}:=x_{\alpha}+i\dot{x}_{\alpha}$, then 
	$(z_{1},...,z_{n})$ is a system of complex coordinates on $TM,$ and one can apply \eqref{ezknkstjqkj}. One obtains
	\begin{eqnarray}\label{ejdkgjfdkjdf}
		g^{\mathbb{C}}_{\alpha\bar{\beta}}=\frac{1}{2}\,h_{\alpha\beta}\circ \pi\,\,\,\,\,\,\,\textup{and}\,\,\,\,\,\,\,
		\textup{Ric}^{\mathbb{C}}_{\alpha\bar{\beta}}=
		-\frac{1}{4}\Big(\dfrac{\partial^{2}\ln\,d}{\partial x_{\alpha}\partial x_{\beta}}\Big)\circ \pi,
	\end{eqnarray}
	where $d$ is the determinant of the matrix 
	$h_{\alpha\beta}=h(\frac{\partial}{\partial x_{\alpha}},\frac{\partial}{\partial x_{\beta}}).$ The second formula in 
	\eqref{ejdkgjfdkjdf} is the local expression for the Ricci tensor in the basis $\{Z_{\alpha},Z_{\bar{\alpha}}\}.$ Returning 
	to the coordinates $(x,\dot{x})$, a direct calculation using 
	\begin{eqnarray}
	\frac{\partial}{\partial x_{\alpha}}=\frac{\partial}{\partial z_{\alpha}}+
			\frac{\partial}{\partial \bar{z}_{\alpha}}\,\,\,\,\,\,\,\,\,\,\,\,\,\,\textup{and}\,\,\,\,\,\,\,\,\,\,\,\,\,\, 
			\frac{\partial}{\partial \dot{x}_{\alpha}}=i\Big(\frac{\partial}{\partial z_{\alpha}}-
		\frac{\partial}{\partial \bar{z}_{\alpha}}\Big)
	\end{eqnarray} 
	shows the following result. 	
\begin{proposition}\label{gfsgag}
	Let $(h,\nabla,\nabla^{*})$ be a dually flat structure on a manifold $M$ and 
	$g$ the K\"{a}hler metric on $TM$ associated to $(h,\nabla)$ via 
	Dombrowski's construction. If $x=(x_{1},...,x_{n})$ is an affine coordinate system on $M,$ then in the coordinates $(x,\dot{x})$, 
	the matrix representation of the Ricci tensor of $g$ is
	\begin{eqnarray}
		\textup{Ric}(x,\dot{x})=
		\begin{bmatrix}
			\beta_{\alpha\beta}(x)  & 0\\
			0 & \beta_{\alpha\beta}(x)
		\end{bmatrix},\,\,\,\,\,\,\,\,\textup{where}\,\,\,\,
		\beta_{\alpha\beta}=-\frac{1}{2}\frac{\partial^{2} \ln d}{\partial x_{\alpha}\partial x_{\beta}},
	\end{eqnarray}
	and where $d$ is the determinant of the matrix 
	$h_{\alpha\beta}=h(\frac{\partial}{\partial x_{\alpha}},\frac{\partial}{\partial x_{\beta}})$. 
\end{proposition}
	Recall that the scalar curvature is by definition the trace of the Ricci tensor. 
\begin{corollary}\label{pouetteeee}
	In the coordinates $(x,\dot{x})$, the scalar curvature of $g$ is given by 
	\begin{eqnarray}\label{fdwmfgkd}
		\textup{Scal}(x,\dot{x})=-\sum_{\alpha,\beta=1}^{n}\,h^{\alpha\beta}(x)
		\frac{\partial^{2} \ln d}{\partial x_{\alpha}\partial x_{\beta}}(x),
	\end{eqnarray}
	where $d$ is the determinant of the matrix 
	$h_{\alpha\beta}$, and where $h^{\alpha\beta}$ are the coefficients of the inverse matrix of $h_{\alpha\beta}.$
\end{corollary}
\begin{remark}
	Observe that the scalar curvature on $TM$ can be written $\textup{Scal}=S\circ \pi,$ 
	where $S\,:\,M\rightarrow \mathbb{R}$ is a globally defined 
	function whose local expression is given by the right hand side 
	of \eqref{fdwmfgkd} (see also \cite{Shima}). 
\end{remark}
\subsection{Completeness}
	Let $(h,\nabla,\nabla^{*})$ be a dually flat structure on a manifold $M.$ We denote by 
	$g$ the Riemannian metric on $TM$ associated to $(h,\nabla)$ via Dombrowski's construction. 
	The corresponding Riemannian distances on $M$ and $TM$ are respectively denoted by $d$ and $\mathbb{\rho}$.
\begin{proposition}\label{dsmnkldsgmsa}
	In this situation, we have: 
	\begin{eqnarray}
		(TM,\rho)\,\,\,\,\textup{is complete}\,\,\,\,\,\,\,\,\,\Leftrightarrow\,\,\,\,\,\,\,\,\,(M,d)\,\,\,\,\textup{is complete}.
	\end{eqnarray}
\end{proposition}
	The rest of this section is devoted to the proof of Proposition \ref{dsmnkldsgmsa}.
\begin{lemma}
	If $(TM,\rho)$ is complete, then $(M,d)$ is complete.
\end{lemma}
\begin{proof}
	This is a direct consequence of the fact that $\pi\,:\,(TM,g)\rightarrow (M,h)$ is a Riemannian submersion 
	(take horizontal geodesics in $TM$ and project them on $M$). 
\end{proof}
	From now on we assume $(M,d)$ complete. Let us fix a Cauchy sequence $(v_{n})_{n\in \mathbb{N}}$ in $(TM,\rho).$ 
	Since $\pi$ is a Riemannian submersion, 
	\begin{eqnarray}
		d(\pi(u),\pi(v))\leq \rho(u,v)\,\,\,\,\,\textup{for all}\,\, u,v\in TM. 
	\end{eqnarray}	
	In particular, if 
	$p_{n}:=\pi(v_{n}),$ then $(p_{n})_{n\in \mathbb{N}}$ is a Cauchy sequence in $(M,d)$, and there exists 
	$p\in M$ such that $p_{n}\rightarrow p$ when $n\rightarrow \infty.$ 
	Take an affine coordinate system $x\,:\,U \rightarrow \mathbb{R}^{n}$ around $p$.
	We denote by $h_{eu}$ the Euclidean metric pulled-back on $U$ via the coordinate system 
	$x\,:\,U\rightarrow \mathbb{R}^{n}.$ 
	By restricting $U$ if necessary, we can assume that there exists $C>0$ such that (by local compactness) :
	\begin{eqnarray}\label{fedksjdkzedj}
		(h_{eu})_{q}(u,u)\leq C\,h_{q}(u,u)\,\,\,\,\,\,\,\textup{for all}\,\,\,q\in U\,\,\,\,\textup{and all}\,\,\,\,u\in T_{q}M.
	\end{eqnarray}
	We also choose $\varepsilon>0$ and $N\in \mathbb{N}$ such that:
	\begin{eqnarray}
		&& B(p,3\,\varepsilon):=\{q\in M\,\vert\,d(q,p)<3\,\varepsilon\} \subseteq U,\nonumber\\
		&& B(p,3\,\varepsilon)\,\,\,\,\textup{is a normal ball},\nonumber\\
		&&\begin{split}
		n,m\geq N\,\,\,\,\,\,\,&\Rightarrow& \,\,\,\,\,\,\,\rho(v_{n},v_{m})<\varepsilon,\nonumber\\
		n\geq N\,\,\,\,\,\,\,&\Rightarrow& \,\,\,\,\,\,\, v_{n}\in \pi^{-1}\big(B(p,\varepsilon)\big). 
		\end{split}
	\end{eqnarray}
\begin{lemma}\label{feffdlfkldkl}
	Let $\gamma(t)$ be a piecewise smooth curve in $TM$ joining $v_{n}$ and $v_{m}$ ($n,m\geq N$). 
	If the length $l(\gamma)$ of $\gamma$ is less than $2\,\varepsilon,$ 
	then $c(t):=(\pi\circ \gamma)(t)$ lies in $B(p,3\,\varepsilon)$ for all $t$. In particular, 
	$\gamma(t)\in \pi^{-1}(U)$ for all $t$.
\end{lemma}
\begin{proof}
	By hypothesis, $l(\gamma)< 2\,\varepsilon$, and since $\pi$ is a Riemannian submersion, $l(c)\leq l(\gamma)$.
	Thus, $l(c)<2\,\varepsilon.$ Therefore, $c(t)$ is a curve in $M$ whose extremities $p_{n}$ and $p_{m}$ lie in 
	$B(p,\varepsilon)$ and such that $l(c)<2\,\varepsilon.$ 
	Since $B(p,3\,\varepsilon)$ is a normal ball, this implies $c(t)\in B(p,3\,\varepsilon)$ for all $t$ 
	(otherwise we would have $l(c)\geq 2\,\varepsilon$ by application of the Gauss Lemma). The lemma follows. 	
\end{proof}
	Let $\gamma(t)$ be a curve in $TM$ as in Lemma \ref{feffdlfkldkl}, with $l(\gamma)<2\,\varepsilon$. 
	Since $\gamma(t)\in \pi^{-1}(U)$ for all $t$, 
	one can represent $\gamma$ in the coordinates $(x_{i},\dot{x}_{i}):$
	\begin{eqnarray}
		\tilde{\gamma}(t):=(x_{i},\dot{x}_{i})(\gamma(t))=\big(c_{1}(t),...,c_{n}(t),V_{1}(t),...,V_{n}(t)\big).
	\end{eqnarray}
	If $\gamma(t)$ is regarded as a vector field $V(t)$ along the curve $c(t)=(\pi\circ \gamma)(t)$, then 
	$(c_{1}(t),...,c_{n}(t))$ and $(V_{1}(t),...,V_{n}(t))$ are just the local expressions for $c(t)$ and $V(t)$ in the coordinates 
	$(x_{i}).$ Observe also that the local expression for the covariant derivative $\nabla_{\dot{c}}V$ is exactly 
	$(\dot{V}_{1},...,\dot{V}_{n})$ since $(x_{i})$ are affine coordinates.

	Similarly, we denote by $\tilde{v}_{n}$ the 
	local representation of $v_{n}$ in the coordinates $(x_{i},\dot{x}_{i})$ ($n\geq N$). This defines a sequence 
	$(\tilde{v}_{n})_{n\in \mathbb{N}}$ in $W\subseteq \mathbb{R}^{2n}$, where
	\begin{eqnarray}
		W:=x(\overline{B(p,\varepsilon)})\times \mathbb{R}^{n},
	\end{eqnarray}
	and where $\overline{B(p,\varepsilon)}$ is the closure of $B(p,\varepsilon)$ in $M.$ 
\begin{lemma}
	$(\tilde{v}_{n})_{n\in \mathbb{N}}$ is a Cauchy sequence in $W$ with respect to the Euclidean distance. 
\end{lemma}
\begin{proof}
	Let $\gamma(t)$ be a curve in $TM$ joining $v_{n}$ and $v_{m}$ ($n,m\geq N$), whose length is less than $2\,\varepsilon$. 
	If $\gamma$ is smooth at $t$, then 
	\begin{eqnarray}\label{fgfgqgdsq}
		\Big\|\frac{d\tilde{\gamma}}{dt}\Big\|^{2}=\sum_{i=1}^{n}\,\Big(\vert \dot{c}_{i}(t)\vert^{2}
		+\vert \dot{V}_{i}\vert^{2}\Big)
		=(h_{eu})_{c(t)}(\dot{c},\dot{c})+(h_{eu})_{c(t)}\big(\nabla_{\dot{c}}V,\nabla_{\dot{c}}V\big).
	\end{eqnarray}
	Let $l_{eu}(\tilde{\gamma})$ be the length of $\tilde{\gamma}$ with respect to the Euclidean metric and 
	$l(\gamma)$ the length of $\gamma$ with respect to $g$. 
	Taking into account \eqref{fedksjdkzedj}, \eqref{fgfgqgdsq} as well as Remark \ref{sfksfjkd}, we see that 
	$l_{eu}(\tilde{\gamma}) \leq \sqrt{C}\, l(\gamma),$ from which we get
	\begin{eqnarray}\label{fgjkrgjkrsld}
		\|\tilde{v}_{n}-\tilde{v}_{m}\|=\|\tilde{\gamma}(0)-\tilde{\gamma}(1)\| \leq l_{eu}(\tilde{\gamma}) \leq \sqrt{C}\, l(\gamma).
	\end{eqnarray}
	Hence $\|\tilde{v}_{n}-\tilde{v}_{m}\|\leq \sqrt{C}\, l(\gamma)$ for all curves $\gamma$ joining 
	$v_{n}$ and $v_{m}$ with $l(\gamma)<2\varepsilon.$ In particular, using a sequence $(\gamma_{k})_{k\in \mathbb{N}}$ 
	of curves joining $v_{n}$ and $v_{m}$ and such that $l(\gamma_{k})\rightarrow \rho(v_{n},v_{m}),$ we deduce that 
	\begin{eqnarray}
		\|\tilde{v}_{n}-\tilde{v}_{m}\|\leq \sqrt{C}\,\rho(v_{n},v_{m}).
	\end{eqnarray}
	Since $(v_{n})_{n\in \mathbb{N}}$ is a Cauchy sequence in $(TM,\rho)$, we conclude that 
	$(\tilde{v}_{n})_{n\in \mathbb{N}}$ is a Cauchy sequence in $W.$ The lemma follows. 
\end{proof}
	Since $W$ is complete (it is a closed subspace of the Euclidean space $\mathbb{R}^{2n}$), 
	$(\tilde{v}_{n})_{n\in\mathbb{N}}$ converges in $W$, and consequently, $(v_{n})_{n\in \mathbb{N}}$ converges 
	in $\pi^{-1}(U)\subseteq TM.$ This achieves the proof of Proposition \ref{dsmnkldsgmsa}. 
\begin{remark}
	The above proof is inspired by a paper of Ebin where the following similar result is shown (see \cite{Ebin}). 
	Let $M$ be a Hilbert manifold endowed with a Riemannian metric 
	$h$ and Levi-Civita connection $\nabla$, not necessarily flat. Let also $g$ be the Riemannian metric on $TM$ associated to $(h,\nabla)$ 
	via Dombrowski's construction. In this situation, if $M$ is complete, then $TM$ is complete.  
\end{remark}

\subsection{K\"{a}hler functions}\label{Kalhlerpouettepouette}

	Let $N$ be a K\"{a}hler manifold with K\"{a}hler structure $(g,J,\omega).$
\begin{definition}\label{fdkjfekgjrkgj}
	A smooth function $f\,:\,N\rightarrow \mathbb{R}$ is called a \textit{K\"{a}hler function} if it satisfies
	\begin{eqnarray}
		\mathscr{L}_{X_{f}}g=0, 
	\end{eqnarray}
	where $X_{f}$ is the Hamiltonian vector field associated to $f$ 
	(i.e.\ $\omega(X_{f},\,.\,)=df(.)$) and where $\mathscr{L}_{X_{f}}$ is the Lie derivative in the direction of $X_{f}.$
\end{definition}
	 Following \cite{Cirelli-Quantum}, we shall denote by 
	$\mathscr{K}(N)$ the space of K\"{a}hler functions on $N\,.$ When $N$ has a finite number of connected components, then 
	$\mathscr{K}(N)$ is a finite dimensional\footnote{The fact that $\mathscr{K}(N)$ is finite dimensional 
	comes from the following result: if $(M,h)$ is a connected Riemannian manifold, 
	then its space of Killing vector fields $\mathfrak{i}(M):=
	\{X\in \mathfrak{X}(M)\,\big\vert\,\mathscr{L}_{X}h=0\}$ 
	is finite dimensional (see for example \cite{Jost}). } Lie algebra for the Poisson 
	bracket $\{f,g\}:=\omega(X_{f},X_{g})\,.$

	Given a smooth function $f\,:\,N\rightarrow \mathbb{R},$ we denote by $\textup{Hess}(f)$ the Riemannian Hessian 
	of $f$ with respect to $g$. If $D$ denotes the Levi-Civita connection with respect to $g$, then by definition
	\begin{eqnarray}
		\textup{Hess}(f)(u,v)=g\big(D_{u}\textup{grad}(f),v\big),
	\end{eqnarray}
	where $u,v\in TM$, and where 
	$\textup{grad}(f)$ is the Riemannian gradient of $f$ with respect to $g$, i.e.\ $g(\textup{grad}(f),\,.\,)=df(.)$. It can be 
	shown that $\textup{Hess}(f)$ is a symmetric tensor (see \cite{Do-Carmo}).
\begin{proposition}[\cite{Cirelli-Quantum}]\label{lxlsjfojfgingd}
	A smooth function $f\,:\,N\rightarrow \mathbb{R}$ is K\"{a}hler if and only if
	\begin{eqnarray}
		\textup{Hess}(f)(JX,JY)=\textup{Hess}(f)(X,Y)
	\end{eqnarray}
	for all vector fields $X,Y$ on $N.$
\end{proposition}
	We now specialize to the case $N=TM$, assuming that $g$ is the K\"{a}hler metric
	associated to a dually flat structure $(h,\nabla,\nabla^{*})$ on a $M$ via Dombrowski's construction. We denote 
	by $\pi\,:\,TM\rightarrow M$ the canonical projection. 

	Let $x\,:\,U\rightarrow \mathbb{R}^{n}$ be an affine coordinate system on $M$ with associated coordinates 
	$(x_{i},\dot{x}_{i})$ on $\pi^{-1}(U)\subseteq TM.$ For $i\in \{1,...,n\}$, set
	\begin{eqnarray}
		\xi_{i}:=\frac{\partial}{\partial x_{i}},\,\,\,\,\,\,\,\,\,\,\,\,
		\xi_{\overline{i}}:=\frac{\partial}{\partial \dot{x}_{i}},
	\end{eqnarray}
	and denote by $(\Gamma^{g})_{AB}^{C}$ the Christoffel symbols of $g$ in the basis 
	$\big\{\xi_{1},...,\xi_{n},\xi_{\overline{1}},...,\xi_{\overline{n}}\big\}$,
	\begin{eqnarray}
		D_{\xi_{A}}\xi_{B}=\sum_{C}\,(\Gamma^{g})_{AB}^{C}\,\xi_{C},
	\end{eqnarray}
	where $A,B,C\in \{1,...,n,\overline{1},...,\overline{n}\}$. We also denote by 
	$(\Gamma^{h})_{ij}^{k}$ the Christoffel symbols of $h$ in the coordinates $(x_{1},...,x_{n})$. 
\begin{lemma}\label{dzld,ld,lzd}
	For $i,j,k\in\{1,...,n\}$, we have:
	\begin{alignat}{5}
		 (\Gamma^{g})_{ij}^{k}\,\, =&\,\,(\Gamma^{h})_{ij}^{k}\circ \pi, 
		 	\quad & \quad 	(\Gamma^{g})_{\overline{i}j}^{k} \,\, =&\,\,0, 
		 	\quad &	\quad	(\Gamma^{g})_{\overline{i}\overline{j}}^{k} \,\,=&\,\,-(\Gamma^{h})_{ij}^{k}\circ \pi,\nonumber\\
		 (\Gamma^{g})_{ij}^{\overline{k}} \,\,=&\,\,0, 
			\quad &	\quad	(\Gamma^{g})_{\overline{i}j}^{\overline{k}} \,\, =&\,\,(\Gamma^{h})_{ij}^{k}\circ \pi,
			 \quad& \quad	(\Gamma^{g})_{\overline{i}\overline{j}}^{\overline{k}}\,\,=&\,\, 0.
	\end{alignat}
\end{lemma}
\begin{proof}
	By a direct calculation. 
\end{proof}
\begin{remark}
	Similar formulas can be obtained in relation to the curvature. 
	Indeed, if $R^{g}$ and $R^{h}$ are the curvature tensors of $g$ and $h$ respectively, 
	then one can show that 
	\begin{eqnarray}
		(R^{g})_{\bar{i}\bar{j}k}^{a}=\Big(-(R^{h})_{ijk}^{a}+\frac{\partial (\Gamma^{h})_{ik}^{a}}{\partial x_{j}}-
		\frac{\partial (\Gamma^{h})_{jk}^{a}}{\partial x_{i}}\Big)\circ \pi,
	\end{eqnarray}
	and similar for $(R^{g})_{ijk}^{a},$ $(R^{g})_{\bar{i}jk}^{\bar{a}},$ etc. 
	In particular, one can prove the Ricci curvature formula given in Proposition \ref{gfsgag} without using the classical 
	formulas \eqref{ezknkstjqkj}\footnote{
	To do this, one has to establish the following two formulas:
	\begin{eqnarray}
		\dfrac{1}{2}\Big\{\dfrac{\partial\,(\Gamma^{h})_{jk}^{a}}{\partial x_{i}}-
			\dfrac{\partial\,(\Gamma^{h})_{ik}^{a}}{\partial x_{j}}\Big\}=\sum_{b=1}^{n}\,
			\Big\{(\Gamma^{h})_{ik}^{b}(\Gamma^{h})_{jb}^{a}-(\Gamma^{h})_{jk}^{b}(\Gamma^{h})_{ib}^{a}\Big\}
			\,\,\,\,\,\,\textup{and}\,\,\,\,\,\,
			\sum_{k=1}^{n}\,(\Gamma^{h})_{ik}^{k}=\frac{1}{2}\dfrac{\partial \ln d}{\partial x_{i}},
	\end{eqnarray}
	where $d$ is the determinant of the matrix 
	$h_{ij}=h(\frac{\partial}{\partial x_{i}},\frac{\partial}{\partial x_{j}}).$
	We caution that these two formulas are only valid for affine coordinate systems. 
	For similar computations, see \cite{Shima}.}. 
\end{remark}
\begin{lemma}\label{djfskdgkefds}
	Let $f\,:\,TM\rightarrow \mathbb{R}$ be a smooth function. Then, on $\pi^{-1}(U),$
	\begin{alignat}{2}
		\textup{Hess}(f)_{ij}	 =& \dfrac{\partial^{2}f}{\partial x_{i}\partial x_{j}} 
		-\sum_{b=1}^{n}\,(\Gamma_{ij}^{h})^{b}\circ \pi\,\dfrac{\partial f}{\partial x_{b}},\quad &\quad
		\textup{Hess}(f)_{i\bar{j}}  =&  \dfrac{\partial^{2}f}{\partial x_{i}\partial \dot{x}_{j}} 
		-\sum_{b=1}^{n}\,(\Gamma_{ij}^{h})^{b}\circ \pi\,\dfrac{\partial f}{\partial \dot{x}_{b}},\\
		\textup{Hess}(f)_{\bar{i}\bar{j}}
	 =&  \dfrac{\partial^{2}f}{\partial \dot{x}_{i}\partial \dot{x}_{j}} 
		+\sum_{b=1}^{n}\,(\Gamma_{ij}^{h})^{b}\circ \pi\,\dfrac{\partial f}{\partial x_{b}}.
	\end{alignat}
\end{lemma}
\begin{proof}
	By a direct calculation using Lemma \ref{dzld,ld,lzd} and the definition of $\textup{Hess}(f).$
\end{proof}
\begin{proposition}\label{c kdlfldfkdl}
	Let $(h,\nabla,\nabla^{*})$ be a dually flat structure on a manifold $M$ and let $(g,J,\omega)$ be the K\"{a}hler structure 
	on $TM$ associated to $(h,\nabla)$ via Dombrowski's construction. Let $f\,:\,TM \rightarrow \mathbb{R}$ be a smooth function. 
	Given an affine coordinate system $x\,:\,U\rightarrow \mathbb{R}^{n}$ on $M$, 
	we have the following equivalence: $f$ is K\"{a}hler on $\pi^{-1}(U)$ if and only if 
	\begin{eqnarray}\label{efkekfjkdf}
		\left \lbrace 
		\begin{array}{ccc}\label{dewngkkrtjkkre}
			\dfrac{\partial^{2}f}{\partial x_{i}\partial x_{j}}-
			\dfrac{\partial^{2}f}{\partial \dot{x}_{i}\partial \dot{x}_{j}} &=& 
			2\displaystyle\sum_{b=1}^{n}(\Gamma_{ij}^{h})^{b}\circ \pi\,\dfrac{\partial f}{\partial x_{b}},\\
			\dfrac{\partial^{2}f}{\partial x_{i}\partial \dot{x}_{j}}+
			\dfrac{\partial^{2}f}{\partial x_{j}\partial \dot{x}_{i}} &=& 
			2\displaystyle\sum_{b=1}^{n}(\Gamma_{ij}^{h})^{b}\circ \pi\,\dfrac{\partial f}{\partial \dot{x}_{b}},
		\end{array}
		\right.
	\end{eqnarray}
	for all $i,j=1,...,n.$
\end{proposition}
\begin{proof}
	According to Proposition \ref{lxlsjfojfgingd}, $f$ is K\"{a}hler if and only if $\textup{Hess}(f)(JX,JY)=\textup{Hess}(f)(X,Y)$ 
	for all vector fields $X,Y.$ If $\textup{Hess}(f)=\bigl[\begin{smallmatrix}
			A & B\\
			{}^{t}\!B & C \end{smallmatrix}\bigl]$ is the matrix representation of $\textup{Hess}(f)$ 
			in the coordinates $(x_{i},\dot{x}_{i})$, 
			then this condition reads 
	\begin{eqnarray}
		\transposee{\begin{bmatrix}
			0 & -I\\
			I & 0
		\end{bmatrix}}
		\begin{bmatrix}
			A & B\\
			{}^{t}\!B & C
		\end{bmatrix}
		\begin{bmatrix}
			0 & -I\\
			I & 0
		\end{bmatrix}=\begin{bmatrix}
			A & B\\
			{}^{t}\!B & C
		\end{bmatrix}
		\quad \Leftrightarrow \quad 
		\begin{bmatrix}
			C & -{}^{t}\!B\\
			-B & A
		\end{bmatrix}
		=\begin{bmatrix}
			A & B\\
			{}^{t}\!B & C
		\end{bmatrix},
	\end{eqnarray}
	that is, $A=C$ and ${}^{t}\!B=-B.$ Writing explicitly these equations using Lemma \ref{djfskdgkefds} exactly yields 
	\eqref{dewngkkrtjkkre}. The proposition follows. 
\end{proof}
	Using the complex coordinates $(z_{1},...,z_{n}):=(x_{1}+i\dot{x}_{1},...,x_{n}+i\dot{x}_{n})$, one can rewrite 
	Proposition \ref{c kdlfldfkdl} more compactly, as follows. 
\begin{proposition}\label{ldklsdklsgvnekg}
	In the same situation as above, 
	\begin{eqnarray}
		 f\,\,\,\textup{is K\"{a}hler on}\,\,\, \pi^{-1}(U)\,\,\,\,\,\,\,\,\Leftrightarrow\,\,\,\,\,\,\,\,
		\dfrac{\partial^{2}f}{\partial z_{i}\partial z_{j}}=
		\sum_{b=1}^{n}\,(\Gamma^{h})_{ij}^{b}\circ \pi\,
		\dfrac{\partial f}{\partial z_{b}}\,\,\,\,\,\,\textup{for all}\,\,\,i,j=1,...,n
	\end{eqnarray}
	(here $\frac{\partial}{\partial z_{k}}:=\frac{1}{2}\big\{\frac{\partial}{\partial x_{k}}-
	i\frac{\partial}{\partial \dot{x}_{k}}\big\}$).
\end{proposition}
	Recall that a vector field $X$ on a manifold $M$ is 
	$\nabla$\textup{-parallel} with respect to a connection $\nabla$ if 
	$\nabla_{Y}X=0$ for all vector fields $Y$ on $M\,.$
\begin{corollary}[\cite{Molitor-exponential}]\label{xql,cnnn}
	Let $(h,\nabla,\nabla^{*})$ be a dually flat structure on a manifold $M$ and let $(g,J,\omega)$ be the K\"{a}hler structure 
	on $TM$ associated to $(h,\nabla)$ via Dombrowski's construction. Let $f\,:\,M\rightarrow \mathbb{R}$ be a smooth function. Then, 
	\begin{eqnarray}
		f\circ \pi\,\,\,\,\,\textup{is K\"{a}hler}\,\,\,\,\,\,\,\, \Leftrightarrow\,\,\,\,\,\,\,\,
		\textup{grad}(f)\,\,\,\,\textup{is}\,\,\nabla\textup{-parallel},
	\end{eqnarray}
	where $\textup{grad}(f)$ is the Riemannian gradient of $f$ with respect to $h$.
\end{corollary}
\begin{proof}
	We use Einstein summation convention and the notation 
	$\partial_{i}=\frac{\partial}{\partial x_{i}}$, $h_{ij}=h(\partial_{i},\partial_{j})$. The coefficients of the 
	inverse matrix of $h_{ij}$ are denoted by $h^{ij}$. Let $f\,:\,M\rightarrow \mathbb{R}$ be a smooth function. 
	In the affine coordinate system $(x_{1},...,x_{n})$, we have: 
	\begin{eqnarray}
		h\big(\nabla_{\partial_{i}}\textup{grad}(f),\partial_{j}\big) &=& 
			h\big(\nabla_{\partial_{i}}\big(h^{ab}\partial_{b}(f)\partial_{a},
		\partial_{j}\big)\big)=h\big(\partial_{i}(h^{ab})\partial_{b}(f)\partial_{a}+
		h^{ab}\partial_{i}\partial_{b}(f)\partial_{a},\partial_{j}\big)\nonumber\\
		&=& h^{ab}h_{aj}\partial_{i}\partial_{b}(f)+\partial_{i}(h^{ab})\partial_{b}(f)h_{aj}
			=\partial_{i}\partial_{j}(f)+\partial_{i}\big(
			h^{ab}h_{aj}\partial_{b}(f)\big)\nonumber\\
		&&{}\,\,-h^{ab}\partial_{i}(h_{aj})\partial_{b}(f)-h^{ab}h_{aj}\partial_{i}\partial_{b}(f)\nonumber\\
		&=& \partial_{i}\partial_{j}(f)+\partial_{i}\partial_{j}(f)-\partial_{i}\partial_{j}(f) 
			-2\frac{1}{2}h^{ab}\partial_{a}(h_{ij})\partial_{b}(f)\nonumber\\
		&=& \partial_{i}\partial_{j}(f)-2(\Gamma^{h})_{ij}^{b}\partial_{b}(f),
	\end{eqnarray}
	where we have used the formula $(\Gamma^{h})_{ij}^{b}=\frac{1}{2}h^{ab}\partial_{a}(h_{ij})$ which comes from the fact 
	that $h$ is a Hessian metric (see \cite{Amari-Nagaoka,Shima}). 
	From this, it is clear that $\textup{grad}(f)$ is $\nabla$-parallel if and only if locally 
	$\partial_{i}\partial_{j}(f)-2(\Gamma^{h})_{ij}^{b}\partial_{b}(f)=0$ for all $i,j=1,...,n.$ But these are exactly the equations 
	characterizing locally a K\"{a}hler function of the form $f\circ \pi$ (compare with Proposition \ref{c kdlfldfkdl}). 
\end{proof}
	Let $(x_{i})$ (resp. $(y_{i})$) be an affine coordinate system with respect to a flat connection $\nabla$ 
	(resp. $\nabla^{*}$) on a Riemannian manifold $(M,h)$. Assume that $(h,\nabla,\nabla^{*})$ 
	is dually flat and that $(x_{i})$ and $(y_{i})$ are 
	dual to each other (in particular $TM$ is a K\"{a}hler manifold for the K\"{a}hler structure associated to 
	$(h,\nabla)$ via Dombrowski's construction). Taking into account Remark \ref{ekkrgjekfjke}, 
	it is not difficult to see that $\textup{grad}(y_{i})=\frac{\partial}{\partial x_{i}}$, and 
	since $\frac{\partial}{\partial x_{i}}$ is obviously $\nabla$-parallel, we deduce the following result. 
\begin{proposition}\label{fpelpçeglv}
	In this situation, the function $y_{i}\circ \pi\,:\,\pi^{-1}(U)\rightarrow \mathbb{R}$ is K\"{a}hler for all $i=1,...,n$. 
\end{proposition}

\subsection{Application: Information Geometry}\label{section information geometry}\label{section expo families}
\begin{definition}\label{que difnkdgnkdfndk}
	A \textit{statistical manifold} (or \textit{statistical model}), is a couple $(S,j),$ where 
	$S$ is a manifold and where $j$ is an injective map from $S$ to the space of 
	all probability density functions $p$ defined on a fixed measure space $(\Omega,dx)$ :
	\begin{eqnarray}
		j\,:\,S\hookrightarrow\Big\{p\,:\,\Omega\rightarrow\mathbb{R}\,\Big\vert\,
		p\,\,\textup{is measurable,}\,\,\,p\geq 0\,\,\,\textup{and}\,\,\,\int_{\Omega}\,
		p(x)\,dx=1\Big\}\,.
	\end{eqnarray}
\end{definition}
	If $\xi=(\xi_{1},...,\xi_{n})$ is a coordinate system on a statistical manifold $S$,
	then we shall indistinctly write $p(x;\xi)$ or $p_{\xi}(x)$ for the probability density 
	function determined by $\xi$.
	
	Given a ``reasonable" statistical manifold $S\,,$ it is possible to define a metric $h_{F}$ 
	and a family of connections $\nabla^{(\alpha)}$ on $S$ ($\alpha\in\mathbb{R}$) in the following 
	way:  for a chart $\xi=(\xi_{1},...,\xi_{n})$ 
	of $S\,,$ define  
	\begin{alignat}{2}
		(h_{F})_{\xi}\big(\partial_{i},\partial_{j})\,\,\,\,:=&\,\,\,\,
			\mathbb{E}_{p_{\xi}}(\partial_{i}\textup{ln}\,(p_{\xi})\cdot
			\partial_{j}\textup{ln}\,(p_{\xi})\big)\,,\\
		\Gamma_{ij,k}^{(\alpha)}(\xi)\,\,\,\,:=&\,\,\,\, \mathbb{E}_{p_{\xi}}\Big[\Big(\partial_{i}\partial_{j}
			\textup{ln}\,(p_{\xi})+\dfrac{1-\alpha}{2}\partial_{i}\textup{ln}\,(p_{\xi})\cdot\partial_{j}
			\textup{ln}\,(p_{\xi})\Big)\,\partial_{k}\textup{ln}\,(p_{\xi})\Big]\,,
	\end{alignat}
	where $\mathbb{E}_{p_{\xi}}$ denotes the mean, or expectation, with respect to the probability 
	$p_{\xi}\,dx\,,$ and where $\partial_{i}$ is a shorthand for $\frac{\partial}{\partial \xi_{i}}\,.$ 
	It can be shown that if the above expressions are defined and smooth for every chart of 
	$S$ (this is not always the case), then $h_{F}$ is a well defined metric on $S$ called the 
	\textit{Fisher metric}, and that the $\Gamma_{ij,k}^{(\alpha)}$'s define a
	connection $\nabla^{(\alpha)}$ via the formula 
	$\Gamma_{ij,k}^{(\alpha)}(\xi)=
	(h_{F})_{\xi}\big(\nabla^{(\alpha)}_{\partial_{i}}\partial_{i},\partial_{k}\big)$ which is 
	called the $\alpha$-\textit{connection}. 
	
	Among the $\alpha$-connections, the $(\pm1)$-connections are particularly important; 
	the 1-connection is usually referred to as the \textit{exponential connection}, also 
	denoted $\nabla^{(e)}\,,$ while the $(-1)$-connection is referred to as 
	the \textit{mixture connection}, denoted $\nabla^{(m)}\,.$

	In this paper, we will only consider statistical manifolds $S$ for which the Fisher 
	metric and $\alpha$-connections are well defined. 
\begin{proposition}[\cite{Amari-Nagaoka}]\label{dkf,sd,ls;dlz}
	Let $S$ be a statistical manifold. Then $(h_{F},\nabla^{(\alpha)},\nabla^{(-\alpha)})$ is a dualistic 
	structure on $S$. In particular, $\nabla^{(-\alpha)}$ is the dual connection of $\nabla^{(\alpha)}.$
\end{proposition}

	We now introduce an important class of statistical manifolds. 
\begin{definition}\label{definition exp}
	An exponential family $\mathcal{E}$ on a measure space $(\Omega,dx)$ is a set of probability 
	density functions $p(x;\theta)$ of the form 
	\begin{eqnarray}\label{equation definiton exp}
		p(x;\theta)=\textup{exp}\,\bigg\{C(x)+\sum_{i=1}^{n}\,\theta_{i}F_{i}(x)-\psi(\theta)\bigg\}\,,
	\end{eqnarray}
	where $C,F_{1},...,F_{n}$ are measurable functions on $\Omega\,,$ $\theta=(\theta_{1},...,\theta_{n})$ 
	is a vector varying in an open subset 
	$\Theta$ of $\mathbb{R}^{n}$ and where 
	$\psi$ is a function defined on $\Theta\,.$
\end{definition}
	In the above definition, it is assumed that the family 
	$\{1,F_{1},...,F_{n}\}$ is linearly independent, so that the map $p(x,\theta)\mapsto\theta\in\Theta$ 
	becomes a bijection, hence defining a global chart of $\mathcal{E}\,.$ 
	The parameters $\theta_{1},...,\theta_{n}$ are called the 
	\textit{natural} or \textit{canonical} \textit{parameters} of the exponential family $\mathcal{E}\,.$

	Besides the natural parameters $\theta_{1},...,\theta_{n}\,,$ an exponential family $\mathcal{E}$ 
	possesses another particularly important parametrization which is given by the 
	\textit{expectation} or \textit{dual parameters} $\eta_{1},...,\eta_{n}\,:$ 
	\begin{eqnarray}\label{fedjfdkgjrkgtr}
		\eta_{i}(p_{\theta}):=\mathbb{E}_{p_{\theta}}(F_{i})=\int_{\Omega}\,F_{i}(x)\,p_{\theta}(x)\,dx.	
	\end{eqnarray} 
	It is not difficult, assuming $\psi$ to be smooth, to show that 
	$\eta_{i}(p_{\theta})=\partial_{\theta_{i}}\psi\,.$ The map 
	$\eta=(\eta_{1},...,\eta_{n})$ is thus a global chart of $\mathcal{E}$ provided that
	$(\partial_{\theta_{1}}\psi,...,\partial_{\theta_{n}}\psi)\,:
	\,\Theta\rightarrow \mathbb{R}^{n}$ is a diffeomorphism onto its image, 
	condition that we will always assume.
\begin{proposition}[\cite{Amari-Nagaoka}]\label{proposition proprietes fam exp}
	Let $\mathcal{E}$ be an exponential family such as in 
	\eqref{equation definiton exp}. Then $(\mathcal{E},h_{F},\nabla^{(e)},\nabla^{(m)})$
	is dually flat and $\theta=(\theta_{1},...,\theta_{n})$ is an affine coordinate system 
	with respect to $\nabla^{(e)}$ while $\eta=(\eta_{1},...,\eta_{n})$ 
	is an affine coordinate system with respect to $\nabla^{(m)}\,.$ 
	Moreover, the following relation holds :
	\begin{eqnarray}\label{equation etrange, et curieuse}
		h_{F}\Big(\frac{\partial}{\partial \theta_{i}},\frac{\partial}{\partial \eta_{j}}\Big)=\delta_{ij},
	\end{eqnarray}
	that is, $\theta$ and $\eta$ are mutually dual coordinate systems.
\end{proposition}
	
\begin{corollary}\label{corollary encore que dire?}
	The tangent bundle $T\mathcal{E}$ of an exponential family $\mathcal{E}$ is a 
	K\"{a}hler manifold for the K\"{a}hler structure $(g,J,\omega)$ associated to 
	$(h_{F},\nabla^{(e)})$ via Dombrowski's construction.
\end{corollary}
	In the sequel, by the K\"{a}hler structure of $T\mathcal{E}\,,$ we shall implicitly refer to 
	the K\"{a}hler structure of $T\mathcal{E}$ described in Corollary \ref{corollary encore que dire?}.

\begin{corollary}[\cite{Molitor-exponential}]\label{,cf,df,dkf,kf,ekf,ekf,ek}
	Let $\mathcal{E}$ be an exponential family defined over a measure space $(\Omega,dx)$ (as in Definition \ref{definition exp}), 
	and let $\mathcal{A}_{\mathcal{E}}$ be the real vector space
	generated by the random variables $1,F_{1},...,F_{n}\,:\,\Omega\rightarrow \mathbb{R}$. In this situation, if 
	$\Phi\,:\,T\mathcal{E}\rightarrow T\mathcal{E}$ is a holomorphic isometry and if $X\in \mathcal{A}_{\mathcal{E}}$, 
	then the function 
	\begin{eqnarray}\label{ds:f,eifjkrf}
		T\mathcal{E}\rightarrow \mathbb{R},\,\,\,\,\,\,\,p\mapsto \int_{\Omega}\,X(x)[(\pi\circ \Phi)(p)](x)dx
	\end{eqnarray}
	is K\"{a}hler (here $\pi\,:\,T\mathcal{E}\rightarrow \mathcal{E}$ is the canonical projection).
\end{corollary}
\begin{proof}
	Assume that $X=\lambda_{0}+\lambda_{1}F_{1}+...+\lambda_{n}F_{n}$, $\lambda_{i}÷\in \mathbb{R}$. Clearly, the above 
	function is K\"{a}hler if and only if $T\mathcal{E}\ni p\mapsto \int_{\Omega}\,X(x)\pi(p)(x)dx$ is K\"{a}hler, which 
	is the case since it is a linear combination of K\"{a}hler functions. Indeed, taking into account the definition of 
	the expectation parameters $\eta_{i}$, one has
	\begin{eqnarray}
		\int_{\Omega}\,X(x)\pi(p)(x)dx=\lambda_{0}+\sum_{i=1}^{n}\,\lambda_{i}\int_{\Omega}\,F_{i}(x)\pi(p)(x)dx=
		\lambda_{0}+\sum_{i=1}^{n}\,(\eta_{i}\circ \pi)(p), 
	\end{eqnarray}
	and since $\theta$ and $\eta$ are affine coordinate systems dual to each other (see Proposition \ref{proposition proprietes fam exp}), 
	it follows from Proposition \ref{fpelpçeglv} that 
	$\eta_{i}\circ \pi$ is K\"{a}hler for all $i=1,...,n.$ The corollary follows. 
\end{proof}
\section{Gaussian distributions: intrinsic geometry}\label{dnfkefnekfnekfn}
	Let $\mathcal{N}$ be the set of all Gaussian distributions of mean $\mu$ and deviation $\sigma$ over $\mathbb{R}$, 
	that is, $\mathcal{N}$ is the set of all $p(x;\mu,\sigma),$ where 
	\begin{eqnarray}\label{dkjgikdgjkjgc}
		p(x;\mu,\sigma)=\dfrac{1}{\sqrt{2\pi}\sigma}\textup{exp}
		\Big\{-\dfrac{(x-\mu)^{2}}{2\,\sigma^{2}}\Big\}. 
	\end{eqnarray}
	It is a 2-dimensional statistical manifold parameterized by $\mu\in \mathbb{R}$ and $\sigma>0$, and since 
	$p(x;\mu,\sigma)=\textup{exp}\big\{F_{1}(x)\theta_{1}+F_{2}(x)\theta_{2}-\psi(\theta)\big\}$, where 
	\begin{eqnarray}\label{equation reecriture normal}
		\theta_{1}=\dfrac{\mu}{\sigma^{2}},
			\,\,\,\theta_{2}=-\dfrac{1}{2\sigma^{2}}\,,\,\,\,
			C(x)=0\,,\,\,\,F_{1}(x)=x\,,\,\,\,
			F_{2}(x)=x^{2}\,,\,\,\,\psi(\theta)=
			-\dfrac{(\theta_{1})^{2}}{4\theta_{2}}+\dfrac{1}{2}\textup{ln}\,
			\Big(-\dfrac{\pi}{\theta_{2}}\Big),
	\end{eqnarray}
	it is also an exponential family (see Definition \ref{definition exp}). 
	Observe that $\theta_{1}\in \mathbb{R}$ and $\theta_{2}<0,$ 
	and that the expectation parameters are (see \cite{Amari-Nagaoka}):
	\begin{eqnarray}\label{dsldgpppqqmdm,}
	\eta_{1}=\mu=-\frac{\theta_{1}}{2\theta_{2}},\,\,\,\,\,\,\,\,\,\,\,\,
	\eta_{2}=\mu^{2}+\sigma^{2}=\frac{(\theta_{1})^{2}-2\theta_{2}}{4(\theta_{2})^{2}}.
	\end{eqnarray}
		
	We denote by $h_{F},$ $\nabla^{(e)}$ and $\nabla^{(m)}$ the Fisher metric, exponential connection and mixture connection 
	on $\mathcal{N}$, respectively. According to Proposition \ref{proposition proprietes fam exp}, 
	$(h_{F},\nabla^{(e)},\nabla^{(m)})$ is a dually flat 
	structure, and consequently the almost Hermitian structure $(g,J,\omega)$ 
	on $T\mathcal{N}$ associated to $(h_{F},\nabla^{(e)})$ via Dombrowski's construction is K\"{a}hler. 

	In this section, we study the geometrical properties of $T\mathcal{N}$, regarded as a K\"{a}hler manifold. 
\subsection{Preliminaries: Siegel-Jacobi space and Jacobi group}\label{que dire?jffndknfkdfnx}

	Let $\textup{Heis}(\mathbb{R})$ and $\textup{SL}(2,\mathbb{R})$ denote respectively 
	the \textit{Heisenberg group} and the \textit{special linear group} of dimension 3. 
	Recall that $\textup{Heis}(\mathbb{R})$ can be identified with 
	$\mathbb{R}^{2}\times \mathbb{R}$ endowed with the multiplication :
	\begin{eqnarray}
		(X_{1},\kappa_{1})\cdot (X_{2},\kappa_{2}):=
		\big(X_{1}+X_{2},\kappa_{1}+\kappa_{2}+\Omega(X_{1},X_{2})\big),
	\end{eqnarray}
	where $\Omega$ is the symplectic form on $\mathbb{R}^{2}$ whose matrix representation in the canonical basis of $\mathbb{R}^{2}$ 
	is $\big[\begin{smallmatrix}
		0 & 1\\
		-1 & 0
	\end{smallmatrix}\big]$, i.e., $\Omega(X_{1},X_{2}):=\lambda_{1}\mu_{2}-\lambda_{2}\mu_{1}$, where $X_{1}=(\lambda_{1},\mu_{1})$ and 
	$X_{2}=(\lambda_{2},\mu_{2})$. 
	Recall also that 
	\begin{eqnarray}
		\textup{SL}(2,\mathbb{R}) := \Big\{\Big[\begin{matrix}
		a & b\\
		c & d
		\end{matrix}
		\Big]\in \textup{Mat}(2,\mathbb{R})\,\Big\vert\,\,ad-bc=1\Big\}
	\end{eqnarray}
	(here $\textup{Mat}(n,\mathbb{R})$ denote the space of $n\times n$ real matrices), and that we have the identification
	$\textup{SL}(2,\mathbb{R})=\textup{Sp}(2,\mathbb{R})$, where 
	\begin{eqnarray}
		\textup{Sp}(2,\mathbb{R}):=\big\{M\in \textup{Mat}(2,\mathbb{R})\,
		\big\vert\,{}^{t}\!M\Omega M=\Omega\big\},\,\,\,\,\,\,\textup{where}\,\,\,\,\
		\Omega=\big[\begin{smallmatrix}
		0 & 1\\
		-1 & 0
	\end{smallmatrix}\big]. 
	\end{eqnarray}

	Let $\textup{Aut}\big(\textup{Heis}(\mathbb{R})\big)$ denote the group of automorphisms of 
	$\textup{Heis}(\mathbb{R})$, that is, the group 
	of diffeomorphisms of $\textup{Heis}(\mathbb{R})$ that are also homeomorphisms. Consider the following map 
	\begin{eqnarray}
		\tau\,:\,\textup{SL}(2,\mathbb{R})\rightarrow \textup{Aut}\big(\textup{Heis}(\mathbb{R})\big),\,\,\,\,
		\tau(M)(X,\kappa):=(XM,\kappa),
	\end{eqnarray}
	where $M\in \textup{SL}(2,\mathbb{R})$, $(X,\kappa)\in \textup{Heis}(\mathbb{R})$, and where 
	$XM$ has to be understood has the multiplication of a row vector with a $2\times 2$ matrix. The fact that $\tau(M)$ is an 
	automorphism of $\textup{Heis}(\mathbb{R})$ 
	is a simple consequence of the identity $\textup{SL}(2,\mathbb{R})=\textup{Sp}(2,\mathbb{R})$, and clearly, 
	$\tau$ is an anti-homomorphism of groups, i.e., $\tau(M_{1}M_{2})=\tau(M_{2})\circ \tau(M_{2}).$ Therefore, one can form 
	the semi-direct product $\textup{SL}(2,\mathbb{R})\ltimes\textup{Heis}(\mathbb{R}).$
	By definition\footnote{Let $G,H$ be two groups and let 
	$\tau\,:\,H\rightarrow \textup{Aut}(G)$ be an anti-homomorphism of groups. By definition, the semi-direct product $H\ltimes G$ 
	is the Cartesian product $H\times G$ endowed with the multiplication $(h_{1},g_{1})\cdot (h_{2},g_{2}):=\big(h_{1}h_{2},
	(\tau(h_{2})g_{1})\cdot g_{2}\big)$. One can check that $H\ltimes G$ is a group and that 
	$(h,g)^{-1}=(h^{-1},\tau(h^{-1})g^{-1}).$}, it is the Cartesian product $\textup{SL}(2,\mathbb{R})\times \textup{Heis}(\mathbb{R})$ 
	endowed with the multiplication 
	\begin{eqnarray}
	(M_{1},X_{1},\kappa_{1})\cdot (M_{2},X_{2},\kappa_{2}):=\big(M_{1}M_{2},X_{1}M_{2}+X_{2},\kappa_{1}+\kappa_{2}+
	\Omega(X_{1}M_{2},X_{2})\big),
	\end{eqnarray}
	where $M_{1},M_{2}\in \textup{SL}(2,\mathbb{R})$ and $(X_{1},\kappa_{1}),(X_{2},\kappa_{2})
	\in \textup{Heis}(\mathbb{R}).$ 
	Following \cite{Eichler,Berndt98}, we call $\textup{SL}(2,\mathbb{R})\ltimes\textup{Heis}(\mathbb{R})$ the \textit{Jacobi group}, and 
	denote it by $G^{J}(\mathbb{R}),$ that is, 
	\begin{eqnarray}
		G^{J}(\mathbb{R}):=\textup{SL}(2,\mathbb{R})\ltimes\textup{Heis}(\mathbb{R}).
	\end{eqnarray}
	We shall also consider the \textit{affine symplectic group},
	\begin{eqnarray}
		\textup{ASp}(2,\mathbb{R}):=\textup{SL}(2,\mathbb{R})\ltimes \mathbb{R}^{2},
	\end{eqnarray}
	which is by definition the semi-direct product of $\textup{SL}(2,\mathbb{R})$ with the abelian group 
	$\mathbb{R}^{2}$ relative to the following anti-homomorphism of groups: 
	\begin{eqnarray}
		\tau\,:\,\textup{SL}(2,\mathbb{R})\rightarrow \textup{Aut}(\mathbb{R}^{2}),\,\,\,\,\tau(M)X:=XM\,\,\,\,\,\,
		(\textup{row vector}\,\,\times\,\,\textup{square matrix}),
	\end{eqnarray}
	where $M\in \textup{SL}(2,\mathbb{R})$ and $X\in \mathbb{R}^{2}.$ 
	By definition, the group multiplication on $\textup{ASp}(2,\mathbb{R})$ is $(M_{1},X_{1})\cdot (M_{2},X_{2})=
	(M_{1}M_{2},X_{1}M_{2}+X_{2})$, where $M_{1},M_{2}\in \textup{SL}(2,\mathbb{R})$ and $X_{1},X_{2}\in \mathbb{R}^{2}$. 
	Beware that $\textup{ASp}(2,\mathbb{R})$ is \textit{not} a subgroup of $G^{J}(\mathbb{R})$, 
	but the latter is a central extension of the former 
	for, there is a short exact sequence of Lie groups, 
	\begin{eqnarray}
		\{e\}\longrightarrow \mathbb{R}\overset{i}{\longrightarrow} G^{J}(\mathbb{R}) 
		\overset{\pi}{\longrightarrow} \textup{ASp}(2,\mathbb{R})
		\rightarrow \{e\},
	\end{eqnarray}
	where $i(\kappa):=\big(\big[\begin{smallmatrix}
		1 & 0\\
		0 & 1
	\end{smallmatrix}\big], 0,\kappa\big)$ and $\pi\big(\big[\begin{smallmatrix}
		a & b\\
		c & d
	\end{smallmatrix}\big], X,\kappa\big):=\big(\big[\begin{smallmatrix}
		a & b\\
		c & d
	\end{smallmatrix}\big], X\big),$ and where obviously the image of $i$ lies in the center of $G^{J}(\mathbb{R}).$

	Let $\mathbb{H}=\{\tau\in \mathbb{C}\,\vert\,\textup{Im}(\tau)>0\}$ denote the upper half-plane. 
	We define a left action of the Jacobi group $G^{J}(\mathbb{R})$ 
	on $\mathbb{H}\times \mathbb{C}$ as follows  
\begin{eqnarray}\label{ljwlerfjoerj}
		\Big(\Big[\begin{matrix}
		a & b\\
		c & d
		\end{matrix}
		\Big],(\lambda,\mu,\kappa)\Big) \cdot  
			(\tau,z)
		:=\Big(\dfrac{a\tau+b}{c\tau+d}, \dfrac{z+\lambda\tau+\mu}{c\tau+d}\Big),
	\end{eqnarray}
	where $(\tau,z)\in \mathbb{H}\times \mathbb{C}.$ It it not an effective action, but by 
	``forgetting" $\kappa$ in the above formula, one 
	obtains a left action of $\textup{ASp}(2,\mathbb{R})$ on $\mathbb{H}\times \mathbb{C}$ 
	which is effective. In particular, one can regard $\textup{ASp}(2,\mathbb{R})$ as a subgroup of the group
	$\textup{Diff}(\mathbb{H}\times \mathbb{C})$ of diffeomorphisms of $\mathbb{H}\times \mathbb{C}.$

\begin{definition}[K\"{a}hler-Berndt metric]\label{cspwpwpdkpw}
	Let $A,B>0$ be arbitrary. The \textit{K\"{a}hler-Berndt} metric is the metric $g_{A,B}$ 
	on $\mathbb{H}\times \mathbb{C}$ whose matrix representation in the coordinates 
	$(u,v,x,y)$ is  
	\begin{eqnarray}
		g_{A,B}(\tau,z):=\begin{bmatrix}
			\tfrac{Av+By^{2}}{v^{3}} & 0 & -\tfrac{By}{v^{2}}  & 0\\
			0 & \tfrac{Av+By^{2}}{v^{3}} & 0  &   -\tfrac{By}{v^{2}}  \\
			-\tfrac{By}{v^{2}} & 0 & \tfrac{B}{v} & 0 \\
			0 & -\tfrac{By}{v^{2}} & 0  &  \tfrac{B}{v} 
		\end{bmatrix},
	\end{eqnarray}
	where $\tau=u+iv\in \mathbb{H}$ and $z=x+iy\in \mathbb{C}.$
\end{definition}
\begin{remark}\label{snzjfekdjfnn} 
	The K\"{a}hler-Berndt metric is a K\"{a}hler metric with respect to the natural complex 
	structure of $\mathbb{H}\times \mathbb{C}$, invariant under the action of 
	the Jacobi group $G^{J}(\mathbb{R})$ (see for example \cite{Yang07,Yang10} and below). It 
	was introduced independently by K\"{a}hler and Berndt in the 80's for the following reasons. Berndt was apparently 
	looking for an invariant Riemannian metric on $\mathbb{H}\times \mathbb{C}$ whose Laplacian could be used 
	to impose good analytical conditions (like being an eigenfunction) on complex functions defined on 
	$\mathbb{H}\times \mathbb{C}$, the objective being to define ``Jacobi-like" functions \cite{Berndt84}; 
	this was just before Eichler and Zagier introduced and systematically 
	studied Jacobi forms in their classic book \cite{Eichler}. K\"{a}hler, on the other hand, was apparently motivated 
	by totally different reasons related to physics (see  \cite{Kahler86,Kahler}).
\end{remark}
\begin{remark}
	Berceanu showed that 
	the K\"{a}hler-Berndt metric can be understood within the group-theoretical 
	framework of Perelomov's coherent states \cite{Berceanu07,Berceanu08,Berceanu111,Berceanu11,Berceanu14}.
\end{remark}
	In a series of papers, Yang introduced the terminology ``Siegel-Jacobi space" (or ``Siegel-Jacobi disk") 
	for the complex space $\mathbb{H}\times \mathbb{C}$ together with a choice of one of the K\"{a}hler metrics $g_{A,B}$ above
	(see \cite{Yang00,Yang05,Yang07,Yang10,Yang13}). 
	In this paper, we shall adopt the following definition. 
\begin{definition}[Siegel-Jacobi space $\mathbb{S}^{J}$]\label{enekngkrgn}
	The \textit{Siegel-Jacobi space} is the K\"{a}hler manifold 
	\begin{eqnarray}
	\mathbb{S}^{J}:=\big(\mathbb{H}\times \mathbb{C},\tfrac{1}{2}g_{1,1}\big).
	\end{eqnarray}
\end{definition}
	In the sequel, we shall denote by $g_{KB}$ and $\omega_{KB}$ the metric and simplectic form of $\mathbb{S}^{J}$, that is, 
	$g_{KB}:=\tfrac{1}{2}g_{1,1}$. From now on, we shall refer to this metric 
	as the \textit{K\"{a}hler-Berndt metric}.

\subsection{K\"{a}hler structure}

	In this section, we return to the study of the K\"{a}hler structure $(g,J,\omega)$ of $T\mathcal{N}.$ 
	We start by recalling the following result (see \cite{Amari-Nagaoka}). 
\begin{proposition}\label{fdkfdgktprioep}
	\textbf{}
	\begin{description}
	\item[$(i)$]
		In the natural coordinates $\theta=(\theta_{1},\theta_{2}),$ the Fisher metric reads:
		\begin{eqnarray}\label{d,fkkd,k,ke}
		h_{F}(\theta)=\dfrac{1}{2(\theta_{2})^{2}}
			\begin{bmatrix}
				-\theta_{2}  &  \theta_{1} \\
				\theta_{1}   & \frac{\theta_{2}-(\theta_{1})^{2}}{\theta_{2}}
			\end{bmatrix},
		\end{eqnarray}
	\item[$(ii)$] in the coordinates $(\theta_{1},\theta_{2}),$ the Christoffel symbols $\Gamma_{ij}^{k}$ of $h_{F}$ are
		\begin{alignat}{5}
			\Gamma_{11}^{1}(\theta)\,\, =& \,\,\frac{\theta_{1}}{2\theta_{2}},  \quad & \quad  \Gamma_{12}^{1}(\theta) \,\,=& \,\,
			-\frac{(\theta_{1})^{2}
			+\theta_{2}}{2(\theta_{2})^{2}},  \quad & \quad   
			\Gamma_{22}^{1}(\theta) \,\,=& \,\,\frac{1}{2}\Big(\frac{\theta_{1}}{\theta_{2}}\Big)^{3},\\
		\Gamma_{11}^{2}(\theta) \,\,=& \,\,\frac{1}{2},   &   \Gamma_{12}^{2}(\theta) \,\,=& \,\,-\frac{\theta_{1}}{2\theta_{2}},  
			&    \Gamma_{22}^{2}(\theta) \,\,=& \,\,\frac{(\theta_{1})^{2}-2\theta_{2}}{2(\theta_{2})^{2}},
		\end{alignat}
	\item[$(iii)$] $(\mathcal{N},h_{F})$ is a complete Riemannian manifold with constant sectional curvature 
		$-\frac{1}{2}.$ 
	\end{description}
\end{proposition}
\begin{proposition}\label{dlfld,f}
	As a K\"{a}hler manifold, $T\mathcal{N}$ is the Siegel-Jacobi space $\mathbb{S}^{J}$ (see Definition \ref{enekngkrgn}), that is, 
	\begin{eqnarray}
		T\mathcal{N}\cong \mathbb{S}^{J}. 
	\end{eqnarray}
\end{proposition}
\begin{proof}
	According to Proposition \ref{proposition proprietes fam exp}, $(\theta_{1},\theta_{2})$ 
	are affine coordinates with respect to $\nabla^{(e)}$. Consequently, one can apply Proposition \ref{fzlfkcslkflef} and 
	conclude that in the coordinates $(\theta,\dot{\theta})=(\theta_{1},\theta_{2},\dot{\theta}_{1},\dot{\theta}_{2})$ 
	the matrix representations of $g,J,\omega$ are:
	\begin{eqnarray}\label{de,fkgjrkgj}
	g(\theta,\dot{\theta})=\begin{bmatrix}
			h_{F}(\theta)  &  0 \\
			0   & h_{F}(\theta)
		\end{bmatrix},\,\,\,\,\,
		J(\theta,\dot{\theta})=\begin{bmatrix}
			0  &  -I \\
			I   & 0
		\end{bmatrix},\,\,\,\,\,
		\omega(\theta,\dot{\theta})=
		\begin{bmatrix}
			0 & h_{F}(\theta) \\
			-h_{F}(\theta) & 0
		\end{bmatrix},
	\end{eqnarray}
	where $h_{F}(\theta)$ is given in \eqref{d,fkkd,k,ke}, and where $I$ is the $2\times 2$ identity matrix (recall 
	that $\dot{\theta}_{k}$ is just the differential of $\theta_{k}$, regarded as a function $T\mathcal{N}\rightarrow \mathbb{R}$). 
	From a complex point of view, we know that 
	$(z_{1},z_{2}):=(\theta_{1}+i\dot{\theta}_{1},\theta_{2}+i\dot{\theta}_{2})$ 
	are global holomorphic coordinates on the complex manifold $(T\mathcal{N},J)$ (see \eqref{ldmzdmzmzmm}). Consequently, 
	one has an identification of complex manifolds $T\mathcal{N}\cong \mathbb{C}\times i\mathbb{H}$ 
	(observe that $i\mathbb{H}=\{z\in\mathbb{C}\,\big\vert\,\textup{Real}(z)<0\}$). 
	Let $f$ be the map 
	\begin{eqnarray}\label{eljlejgler}
		T\mathcal{N}\cong \mathbb{C}\times i\mathbb{H}\rightarrow 
		\mathbb{H}\times \mathbb{C},\,\,\,(z_{1},z_{2})\mapsto (-iz_{2},iz_{1}).
	\end{eqnarray}
	Clearly, $f$ is biholomorphic, and in the coordinates $(\theta,\dot{\theta})$ on $T\mathcal{N}$ and 
	$(u,v,x,y)$ on $\mathbb{H}\times \mathbb{C}$ 
	(see Definition \ref{cspwpwpdkpw}), 
	it reads $f(\theta,\dot{\theta})=(\dot{\theta}_{2},-\theta_{2},-\dot{\theta}_{1},\theta_{1})$. 
	Now, using \eqref{de,fkgjrkgj} together with the explicit description of $g_{1,1}$ given in Definition \ref{cspwpwpdkpw}, 
	a straightforward computation shows 
	that $f^{*}g_{KB}=g$. The proposition follows.  
\end{proof}
\begin{proposition}\label{completnesssss}
	$(T\mathcal{N},g)$ is complete. 
\end{proposition}
\begin{proof}
	There are two ways to prove it. The first is to use Proposition \ref{dsmnkldsgmsa} and the fact that 
	$(\mathcal{N},h_{F})$ is complete (see Proposition \ref{fdkfdgktprioep}). 
	The second is to observe 
	that the Siegel-Jacobi space $\mathbb{S}^{J}$ is a homogeneous Riemannian manifold 
	(see Remark \ref{snzjfekdjfnn} and Proposition \ref{df,dkfjekrfjkefjkrt}).
\end{proof}
\begin{proposition}\label{wdklwkdlp}
	In the coordinates $(\theta,\dot{\theta}),$ the matrix representation
	of the Ricci tensor of $g$ is
	\begin{eqnarray}
		\textup{Ric}(\theta,\dot{\theta})= \begin{bmatrix}
			\beta(\theta)& 0\\
			0& \beta(\theta)
		\end{bmatrix}, \,\,\,\,\,\,\textup{where}\,\,\,\,\,\,\,\beta(\theta)=-\frac{3}{2}\begin{bmatrix}
			0& 0\\
			0& \frac{1}{(\theta_{2})^{2}}
		\end{bmatrix}.
	\end{eqnarray}
\end{proposition}
\begin{proof}
	Follows from Proposition \ref{gfsgag} and Proposition \ref{fdkfdgktprioep}.
\end{proof}
	From Proposition \ref{wdklwkdlp}, one easily deduces the following corollary. 
\begin{corollary}\label{elfkjfekdfje}
	\text{}
	\begin{description}
	\item[$(i)$] $Ric(X,X)\leq 0$ for all $X\in T(T\mathcal{N}).$
	\item[$(ii)$] $(T\mathcal{N},g)$ is not Einstein\footnote{A Riemannian manifold is \textit{Einstein} if its Ricci tensor is a scalar 
	multiple of the metric at each point. See \cite{Lee-Riemannian}.}. In particular, the holomorphic sectional 
	curvature\footnote{The \textit{holomorphic sectional curvature} of a K\"{a}hler manifold $(N,g,J,\omega)$ 
	is the function $TN\rightarrow \mathbb{R},\,\,u\mapsto \tfrac{g(R(u,Ju)Ju,u)}{g(u,u)^{2}},$ where $R$ is the curvature tensor. It is 
	well-known that if the holomorphic sectional curvature is constant, then $N$ is Einstein. See 
	for example \cite{Ballmann}.\label{ejfkejktre}} of 
	$T\mathcal{N}$ is not constant. 
	\item[$(iii)$] The scalar curvature of $(T\mathcal{N},g)$ is constant and equal to $-6.$
	\end{description}
\end{corollary}
\begin{remark}
	Since $T\mathcal{N}\cong \mathbb{S}^{J}$, one has the 
	analogues of Proposition \ref{completnesssss}, Proposition \ref{wdklwkdlp} and Corollary \ref{elfkjfekdfje} 
	for the Siegel-Jacobi space $\mathbb{S}^{J}.$ The analogue of Corollary \ref{elfkjfekdfje} for $\mathbb{S}^{J}$ 
	was established by Yang in \cite{Yang00}, and later on generalized by Berceanu \cite{Berceanu14} and 
	Yang \cite{Yang13} for the metric $g_{A,B}$. They showed, in particular, that 
	the scalar curvature of $g_{A,B}$ is constant and equal to $-\tfrac{3}{A}$. 
\end{remark}

\subsection{The group of holomorphic isometries}\label{dfiejfiejennc}

	Recall that the affine symplectic group $\textup{ASp}(2,\mathbb{R})$ acts effectively on the Siegel-Jacobi space 
	$\mathbb{S}^{J}\cong T\mathcal{N}$. Therefore, $\textup{ASp}(2,\mathbb{R})$ can be regarded as a subgroup of the group 
	$\textup{Diff}(T\mathcal{N})$ of diffeomorphisms of $T\mathcal{N}.$ 
	Recall also that the group of holomorphic isometries of $T\mathcal{N}$ is the subgroup of $\textup{Diff}(T\mathcal{N})$ whose 
	elements satisfy $\varphi^{*}g=g$ and $\varphi_{*}J=J\varphi_{*}.$ 	
\begin{theorem}\label{theoremnnn}
	The group of holomorphic isometries of $T\mathcal{N}$ is the affine symplectic group $\textup{ASp}(2,\mathbb{R})$. 
\end{theorem}

	As explained below, our proof relies on the resolution of the following system of partial differential equations
	\begin{eqnarray}\label{equationdelamort}
		\left \lbrace
		\begin{array}{cc}
			\Big(\dfrac{\partial u}{\partial x}\Big)^{2}+
		\Big(\dfrac{\partial u}{\partial y}\Big)^{2}=\Big(\dfrac{u}{x}\Big)^{2},\\[1em]
		\Delta u\equiv 0, 
		\end{array}
		\right.
	\end{eqnarray}
	where $u(x,y)$ is a smooth function defined on $U:=\big\{(x,y)\in \mathbb{R}^{2}\,\big\vert\,x<0\big\}$, and 
	where $\Delta=\frac{\partial^{2}}{\partial x^{2}}+\frac{\partial^{2}}{\partial y^{2}}$ is the Laplace 
	operator.
\begin{remark}
	If a solution $u$ of the first equation in \eqref{equationdelamort} satisfies $u(x,y)<0$ for all $(x,y)\in U,$ 
	then $v:=\ln (-u)$ is a solution of the 2-dimensional \textit{Eikonal equation}:
	\begin{eqnarray}
		\Big(\dfrac{\partial v}{\partial x}\Big)^{2}+
		\Big(\dfrac{\partial v}{\partial y}\Big)^{2}=\dfrac{1}{\big(f(x,y)\big)^{2}},
	\end{eqnarray}
	with $f(x,y)=x.$ In geometrical optics, the Eikonal equation describes the wave fronts of light in an inhomogeneous medium 
	with a variable index of refraction $\frac{1}{f^{2}}$ (see for example \cite{Courant,Yu}). Mathematically, 
	only a few explicit solutions are 
	known (see $\cite{Borovskikh,Moskalensky}$). 
\end{remark}
\begin{remark}\label{analyticity}
	Every solution of \eqref{equationdelamort} is real analytic (since it is harmonic). In particular, if $u,v$ are two solutions of 
	\eqref{equationdelamort} which coincide on an open subset of $U$, then they coincide on $U$ (see \cite{Axler}). 
\end{remark}
	Let us fix a smooth solution $u$ of \eqref{equationdelamort} satisfying $u(x,y)<0$ for all $(x,y)\in U$ (this last 
	condition will be justified below). Set 
	\begin{eqnarray}\label{definitionUUUUUuetu}
		U_{0}:=\Big\{(x,y)\in U\,\Big\vert\, \frac{\partial u}{\partial y}(x,y)=0\Big\}.
	\end{eqnarray}
\begin{lemma}\label{edoefkoe}
	If $U_{0}=U$ (i.e.\ $\frac{\partial u}{\partial y}\equiv 0$ on $U$), 
	then there exists $a\in \mathbb{R},$ $a\neq 0,$ such that for all $(x,y)\in U,$
	\begin{eqnarray}
		u(x,y)=a^{2}\,x.
	\end{eqnarray}
\end{lemma}
\begin{proof}
	By a direct calculation. 
\end{proof}
	Let us now assume $U_{0}\neq U$. This means that there exists $p=(p_{1},p_{2})\in U$ such 
	that $\frac{\partial u}{\partial y}(p)\neq 0$. Without loss of generality, we can assume 
	$\frac{\partial u}{\partial y}(p)>0$ (the case 
	$<0$ is completely analog). Fix $\varepsilon>0$ such that 
	\begin{eqnarray}
		\frac{\partial u}{\partial y}(q)>0\,\,\,\,\,\,\textup{for all}\,\,\,\,\,\,q\in \,
		]\,p_{1}-\varepsilon,p_{1}+\varepsilon\,[\times ]\,p_{2}-\varepsilon,p_{2}+\varepsilon\,[\,=:C.
	\end{eqnarray}
	On $C$, there exists a smooth function 
	$\alpha\,:\,C\rightarrow \mathbb{R}$ which satisfies (see the first equation in \eqref{equationdelamort}) 
	\begin{eqnarray}\label{kf,kdfnkdf,kd}
	\dfrac{x}{u}\dfrac{\partial u}{\partial x}=\cos(\alpha(x,y))\,\,\,\,\,\,\,\,\,\,\,\,\textup{and}\,\,\,\,\,\,\,\,\,\,\,\,\, 
	\dfrac{x}{u}\dfrac{\partial u}{\partial y}=\sin(\alpha(x,y))
	\end{eqnarray}
	for all $(x,y)\in C.$ 
	By specifying the image of $\alpha$, such a function is unique. We choose $0< \alpha < \pi$.
	
\begin{lemma}\label{kefjjgrgfvfklrkfl}
	We have: 
	\begin{eqnarray}\label{dlsockjeodghvjcs}
		\left \lbrace
		\begin{array}{ccc}
			\dfrac{\partial \alpha}{\partial x} &=& \dfrac{\sin(\alpha)}{x},\\[1em]
			\dfrac{\partial \alpha}{\partial y} &=& \dfrac{1-\cos(\alpha)}{x}. 
		\end{array}
		\right.
	\end{eqnarray}
\end{lemma}
\begin{proof}
	Observe that \eqref{kf,kdfnkdf,kd} can be rewritten 
	\begin{eqnarray}
		\dfrac{\partial}{\partial x}\big(\ln(-u)\big)=\dfrac{\cos(\alpha)}{x}\,\,\,\,\,\,\,\,\,\,\,\,
		\textup{and}\,\,\,\,\,\,\,\,\,\,\,\,\,\dfrac{\partial}{\partial y}\big(\ln(-u)\big)=\dfrac{\sin(\alpha)}{x}.
	\end{eqnarray}
	Taking the partial derivative with respect to $y$ of the first equation and the partial derivative with respect to 
	$x$ of the second equation immediately yields the equality 
	\begin{eqnarray}
		\dfrac{\partial}{\partial y}\Big(\dfrac{\cos(\alpha)}{x}\Big)=
			\dfrac{\partial}{\partial x}\Big(\dfrac{\sin(\alpha)}{x}\Big)
	\end{eqnarray} 
	which can be rewritten
	\begin{eqnarray}\label{kjflflefeg}
			\cos(\alpha)\dfrac{\partial \alpha}{\partial x}+ \sin(\alpha)\dfrac{\partial \alpha}{\partial y}=
			\dfrac{\sin(\alpha)}{x}.
	\end{eqnarray}
	On the other hand, the equation $\Delta u\equiv 0$ together with \eqref{kf,kdfnkdf,kd} yields 
	\begin{eqnarray}
		\dfrac{\partial}{\partial x}\Big(\dfrac{u}{x}\cos(\alpha)\Big)
		+\dfrac{\partial}{\partial y}\Big(\dfrac{u}{x}\sin(\alpha)\Big)=0,
	\end{eqnarray}
	which is equivalent to  
	\begin{eqnarray}\label{defdlgkldfk}
		-\sin(\alpha)\dfrac{\partial \alpha}{\partial x}+ \cos(\alpha)\dfrac{\partial \alpha}{\partial y}=
			\dfrac{\cos(\alpha)-1}{x}.
	\end{eqnarray}
	Multiplying $\eqref{kjflflefeg}$ by $\sin(\alpha)$ (resp.\ $\cos(\alpha)$) and \eqref{defdlgkldfk} by 
	$\cos(\alpha)$ (resp.\ $\sin(\alpha)$), then summing (resp.\ subtracting) exactly yields \eqref{dlsockjeodghvjcs}. 
	The lemma follows. 
\end{proof}
\begin{lemma}\label{qppeekk}
	There exists $b\in \mathbb{R}$ such that on $C,$
	\begin{eqnarray}
		\left \lbrace
		\begin{array}{ccc}
			\cos(\alpha(x,y)) &=& \dfrac{(y+b)^{2}-x^{2}}{(y+b)^{2}+x^{2}},\\[1em]
			\sin(\alpha(x,y)) &=&  -\,\dfrac{2x(y+b)}{(y+b)^{2}+x^{2}}. 
		\end{array}
		\right.
	\end{eqnarray}
\end{lemma}
\begin{proof}
	According to Lemma \ref{kefjjgrgfvfklrkfl}, we have
	\begin{alignat}{5}\label{fgf,dkgldgjfk}
		\dfrac{1}{\sin(\alpha)}\,\dfrac{\partial \alpha}{\partial x} &= \dfrac{1}{x}  &\quad
		&\Rightarrow  & \quad
		\ln(\tan(\alpha/2))&=\ln(-x)+g(y) &\quad 
		&\Rightarrow &\quad    \tan(\alpha/2)&=-x\,e^{g(y)},      \\
		\dfrac{1}{1-\cos(\alpha)}\dfrac{\partial \alpha}{\partial y} &= \dfrac{1}{x} &  &\Rightarrow  & 
		-\dfrac{1}{\tan(\alpha/2)} &= \dfrac{y}{x}+h(x) &   &\Rightarrow
		& \tan(\alpha/2)&=-\dfrac{1}{\frac{y}{x}+h(x)},
		\label{e,gkdfkdjfjkd}
	\end{alignat}
	where $g$ and $h$ are smooth functions of the variables $y$ and $x$ respectively. Thus,
	\begin{eqnarray}
		-x\,e^{g(y)}=-\dfrac{1}{\frac{y}{x}+h(x)}\,\,\,\,\,\,\Rightarrow\,\,\,\,\,\,\,xh(x)=e^{g(y)}-y,
	\end{eqnarray}
	from which we deduce the existence of a constant $E\in \mathbb{R}$ such that $xh(x)=E$ and $e^{g(y)}-y=E$ for all 
	$x\in\,]\,p_{1}-\varepsilon,p_{1}+\varepsilon\,[ $ 
	and all $y\in\,]\,p_{2}-\varepsilon,p_{2}+\varepsilon\,[.$ Thus, 
	\begin{eqnarray}
		g(y)=-\ln(E+y),\,\,\,\,\,\,\,\,h(x)=\dfrac{E}{x}.
	\end{eqnarray}
	Taking into account the last equation in \eqref{fgf,dkgldgjfk} (or \eqref{e,gkdfkdjfjkd}), we thus have 
	\begin{eqnarray}\label{dlskdlkfldfx}
		\alpha(x,y)=-2\arctan\Big(\dfrac{x}{y+E}\Big).
	\end{eqnarray}
	The lemma is now a simple consequence of \eqref{dlskdlkfldfx} together with 
	the following formulas: $\cos\big(2\arctan(r)\big)=\frac{1-r^{2}}{1+r^{2}}$ and 
	$\sin\big(2\arctan(r)\big)=\frac{2r}{1+r^{2}},$ $r\in \mathbb{R}.$
\end{proof}
\begin{lemma}\label{dkjfkdjgdhfjk}
	There exists $a\in \mathbb{R},$ $a\neq 0,$ and $b\in \mathbb{R}$ such that
	on $U,$
	\begin{eqnarray}\label{kfdkfjdfjdk}
		u(x,y)=\dfrac{a^{2}x}{(y+b)^{2}+x^{2}}.
	\end{eqnarray}
\end{lemma}
\begin{proof}
	Since $\frac{x}{u}\frac{\partial u}{\partial y}=\sin(\alpha),$ Lemma \ref{qppeekk} implies that on $C$,
	\begin{eqnarray}\label{vefnkdfkdghjfkdj²}
		\dfrac{\partial}{\partial y}\big(\ln(-u)\big)=-\dfrac{2(y+b)}{(y+b)^{2}+x^{2}}\,\,\,\,\,\,\,\,&\Leftrightarrow&\,\,\,\,\,\,\,
		\ln(-u)=-\ln\big((y+b)^{2}+x^{2}\big)+f(x)\label{fefefefee}\\
		&\Leftrightarrow&\,\,\,\,\,\,\,\, u=-\dfrac{e^{f(x)}}{(y+b)^{2}+x^{2}},
	\end{eqnarray}
	where $f$ is a smooth function depending on the variable 
	$x\in \,]\,p_{1}-\varepsilon,p_{1}+\varepsilon\,[.$ In order to find $f$, we differentiate the right hand side 
	of the equivalence in \eqref{fefefefee} and use $\frac{\partial}{\partial x}\big(\ln(-u)\big)=\frac{\cos(x)}{x}$. We obtain
	\begin{eqnarray}
		f'(x)-\dfrac{2x}{(y+b)^{2}+x^{2}}=\dfrac{1}{x}\dfrac{(y+b)^{2}-x^{2}}{(y+b)^{2}+x^{2}},
	\end{eqnarray}
	which leads to $xf'(x)=1,$ i.e., $f(x)=\ln(-x)$ (+ constant). Hence \eqref{kfdkfjdfjdk} holds on $C$. Using the fact that 
	$u$ is analytic (see Remark \ref{analyticity}), it also holds on $U.$ The lemma follows. 
\end{proof}
	Collecting our results, we deduce the following 
\begin{proposition}\label{solutionsystemmmm}
	Let $u$ be a solution of \eqref{equationdelamort} satisfying $u(x,y)<0$ for all $(x,y)\in U$. Then, $u$ has the following 
	form (two possibilities) : 
	\begin{description}
		\item[$(1)$] $u(x,y)=\dfrac{a^{2}x}{(y+b)^{2}+x^{2}},\,\,\,a,b\in \mathbb{R},\,\,a\neq 0,$
		\item[$(2)$] $u(x,y)=a^{2}x,\,\,a\in \mathbb{R},\,\,a\neq 0.$
	\end{description}
\end{proposition}
\begin{remark}\label{lddlkgfl}
	A variant of Proposition \ref{solutionsystemmmm} is as follows. Consider the system of partial differential equations 
	\begin{eqnarray}\label{dsfoefjos}
		\left \lbrace
		\begin{array}{cc}
			\Big(\dfrac{\partial u}{\partial x}\Big)^{2}+
			\Big(\dfrac{\partial u}{\partial y}\Big)^{2}=\lambda^{2},\\[1em]
		\Delta u\equiv 0, 
		\end{array}
		\right.
	\end{eqnarray}
	where $u(x,y)$ is a smooth function defined on $\mathbb{R}^{2}$, and where $\lambda\in \mathbb{R},$ $\lambda\neq 0.$ 
	If $u$ is a smooth solution of \eqref{dsfoefjos}, then there exist
	$a,b,c\in \mathbb{R}$ such that $a^{2}+b^{2}=\lambda^{2}$, and such that for all $(x,y)\in \mathbb{R}^{2}$,
	\begin{eqnarray}
		u(x,y)=ax+by+c.
	\end{eqnarray}
	This can be shown using arguments similar to the ones we already used. 
\end{remark}
	We now return to the group of holomorphic isometries of $T\mathcal{N}.$ 
	Let $\varphi\,:\,T\mathcal{N}\rightarrow T\mathcal{N}$ be a diffeomorphism. 
	In the coordinates $(\theta,\dot{\theta})$, $\varphi$ can be written 
	\begin{eqnarray}
		\varphi(\theta,\dot{\theta})=\big(\varphi^{1}(\theta,\dot{\theta}),\varphi^{2}(\theta,\dot{\theta}),
		\varphi^{3}(\theta,\dot{\theta}),\varphi^{4}(\theta,\dot{\theta})\big),
	\end{eqnarray}
	with $\varphi^{2}<0,$ and its derivative can be decomposed into blocks of $2\times 2$ real matrices: 
	\begin{eqnarray}\label{decompositionblocs}
		\varphi_{*_{(\theta,\dot{\theta})}}=
		\begin{bmatrix}
			A(\theta,\dot{\theta}) & B(\theta,\dot{\theta})\\
			C(\theta,\dot{\theta}) & D(\theta,\dot{\theta})
		\end{bmatrix}.
	\end{eqnarray}
	The entries of the matrices $A,B,C,D$ are denoted by 
	$a_{ij},b_{ij},c_{ij},d_{ij},$ respectively. Hence, $a_{11}=\frac{\partial \varphi^{1}}{\partial \theta_{1}},$ 
	$b_{22}=\frac{\partial \varphi^{4}}{\partial {\theta}_{2}},$ etc. 
	
	From a complex point of view, recall that 
	$(z_{1},z_{2})=(\theta_{1}+i\dot{\theta}_{1},\theta_{2}+i\dot{\theta}_{2})$ are global complex coordinates on $T\mathcal{N}$. Therefore, 
	$T\mathcal{N}\cong \mathbb{C}\times i\mathbb{H}$, and we have 
	\begin{eqnarray}
		\varphi\,\,\,\textup{is holomorphic}\,\,\,\,\,\,\,&\Leftrightarrow &\,\,\,\,\,\,\,\varphi^{1}+i\varphi^{3}\,\,\,
		\textup{and}\,\,\,\varphi^{2}+i\varphi^{4}\,\,\,\textup{are holomorphic functions}\nonumber\\
		&\Leftrightarrow& \text{}\,\,\,\,
			\small{\frac{\partial}{\partial \bar{z}_{k}}}(\varphi^{1}+i\varphi^{3})=
			\frac{\partial}{\partial \bar{z}_{k}}(\varphi^{2}+i\varphi^{4})=0,\,\,\,\,k=1,2,
	\end{eqnarray}
	where $\frac{\partial}{\partial \bar{z}_{k}}=\frac{1}{2}\big\{\frac{\partial}{\partial \theta_{k}}+
	i\frac{\partial}{\partial \dot{\theta}_{k}}\big\}.$ Equivalently, $\varphi$ is holomorphic if and only if $A=D$ and $B=-C$ 
	(Cauchy-Riemann equations). 
\begin{lemma}\label{fejlfkklwkew}
	Assume that $\varphi$ is holomorphic. In this situation, $\varphi$ is an isometry if and only if 
	$\varphi^{1}$ and $\varphi^{2}$ are solutions of the following system of partial differential equations: 
	\begin{alignat}{1}
		 h_{11}(\theta) &=  h_{11}(\varphi)\big[(a_{11})^{2}+(b_{11})^{2} \big]+
		 2\, h_{12}(\varphi)\big[a_{11}a_{21}+b_{11}b_{21}\big]+h_{22}(\varphi)\big[(a_{21})^{2}+(b_{21})^{2} \big],  \nonumber \\
		h_{12}(\theta)  &=  h_{11}(\varphi)\big[a_{11}a_{12}+b_{11}b_{12}\big] 
		+h_{12}(\varphi)\big[a_{11}a_{22}+a_{21}a_{12}+b_{11}b_{22}+b_{21}b_{12}\big]
		+h_{22}(\varphi)\big[a_{21}a_{22}+b_{21}b_{22}\big],\nonumber \\
	h_{22}(\theta) &=  h_{11}(\varphi)\big[(a_{12})^{2}+(b_{12})^{2} \big] +
		2\, h_{12}(\varphi)\big[a_{12}a_{22}+b_{12}b_{22}\big]+h_{22}(\varphi)\big[(a_{22})^{2}+(b_{22})^{2} \big], \nonumber\\
	  0  &= h_{11}(\varphi)\big[a_{11}b_{12}-a_{12}b_{11}\big]
	  +h_{12}(\varphi)\big[a_{11}b_{22}+a_{21}b_{12}-b_{11}a_{22}-b_{21}a_{12}\big]
	  +h_{22}(\varphi)\big[a_{21}b_{22}-b_{21}a_{22}\big], \nonumber
	 \end{alignat}
	 where $h_{ij}(\theta):=h_{F}(\theta)\big(\frac{\partial}{\partial \theta_{i}},\frac{\partial}{\partial \theta_{j}}\big).$
\end{lemma}
\begin{remark}
	Observe that $h_{ij}(\varphi)=h_{ij}\circ \varphi$ only depends on $\varphi^{1}$ and $\varphi^{2}$ (see item $(i)$ in Proposition 
	\ref{fdkfdgktprioep}).
\end{remark}
\begin{proof}[Proof of Lemma \ref{fejlfkklwkew}]
	By hypothesis, $\varphi$ is holomorphic, which means that $A=D$ and $B=-C.$ Consequently, the matrix 
	representation of the equation $\varphi^{*}g=g$ reads 
	\begin{gather}
		\transposee{\begin{bmatrix}
			A(\theta,\dot{\theta}) & B(\theta,\dot{\theta})\\
			-B(\theta,\dot{\theta}) & A(\theta,\dot{\theta})
		\end{bmatrix}}
		\begin{bmatrix}
			h_{F}(\varphi) & 0\\
			0 & h_{F}(\varphi)
		\end{bmatrix}
		\begin{bmatrix}
			A(\theta,\dot{\theta}) & B(\theta,\dot{\theta})\\
			-B(\theta,\dot{\theta}) & A(\theta,\dot{\theta})
		\end{bmatrix}
		=\begin{bmatrix}
			h_{F}(\theta) & 0\\
			0 & h_{F}(\theta)
		\end{bmatrix}
			\nonumber\\
	\Leftrightarrow \,\,\,\,\,\,\,\,\,\, 
		\left \lbrace
		\begin{array}{ccc}\label{lkelekrle}
			{}^{t}\!A(\theta,\dot{\theta})\,(h_{F}(\varphi))\, A(\theta,\dot{\theta})
			+{}^{t}\!B(\theta,\dot{\theta})\,(h_{F}(\varphi))\, B(\theta,\dot{\theta}) &=& h_{F}(\theta),\\
			{}^{t}\!A(\theta,\dot{\theta})\,(h_{F}\varphi))\, B(\theta,\dot{\theta})
			-{}^{t}\!B(\theta,\dot{\theta})\,(h_{F}(\varphi))\, A(\theta,\dot{\theta}) &=& 0.\\
		\end{array}
		\right.
	\end{gather}
	The first equation in \eqref{lkelekrle} is an equality of symmetric matrices, and thus produces three equations 
	which are after a direct calculation the first three equations of the lemma. 
	The second equation in \eqref{lkelekrle} is an equality of anti-symmetric matrices, thus it yields only one equation 
	which is the last equation of the lemma, as a simple calculation shows. The lemma follows. 
\end{proof}
	Instead of trying to solve directly the system of equations in Lemma \ref{fejlfkklwkew}, our strategy will be to use the fact 
	that the Ricci tensor is a Riemannian invariant, that is, $\varphi^{*}\textup{Ric}=\textup{Ric}$ for every isometry $\varphi.$
\begin{lemma}\label{isometric}
	If $\varphi$ is an isometry, then
	\begin{alignat}{2}
		\dfrac{\partial \varphi^{2}}{\partial \theta_{1}}=\dfrac{\partial \varphi^{2}}{\partial \dot{\theta}_{1}}&=0,  
			&\quad\quad \Big(\dfrac{\partial \varphi^{2}}{\partial \theta_{2}}\Big)^{2}
			+\Big(\dfrac{\partial \varphi^{4}}{\partial \theta_{2}}\Big)^{2}&=
			\Big(\dfrac{\varphi^{2}}{\theta_{2}}\Big)^{2},\label{dl,flsd,ljfedf}\\
		\dfrac{\partial \varphi^{4}}{\partial \theta_{1}}=\dfrac{\partial \varphi^{4}}{\partial \dot{\theta}_{1}}&=0, & 
			\Big(\dfrac{\partial \varphi^{2}}{\partial \dot{\theta}_{2}}\Big)^{2}+
			\Big(\dfrac{\partial \varphi^{4}}{\partial \dot{\theta}_{2}}\Big)^{2}&=
			\Big(\dfrac{\varphi^{2}}{{\theta}_{2}}\Big)^{2}.
	\end{alignat}
\end{lemma}
\begin{proof}
	In the coordinates $(\theta,\dot{\theta})$, we have (see Proposition \ref{wdklwkdlp}):
	\begin{eqnarray}\label{dzdnzkfnkned}
		\textup{Ric}(\theta,\dot{\theta})= \begin{bmatrix}
			\beta(\theta)& 0\\
			0& \beta(\theta)
		\end{bmatrix}, \,\,\,\,\,\,\textup{where}\,\,\,\,\,\,\,\beta(\theta)=-\frac{3}{2}\begin{bmatrix}
			0& 0\\
			0& \frac{1}{(\theta_{2})^{2}}
		\end{bmatrix}.
	\end{eqnarray}
	Using the bloc decomposition of $\varphi_{*}$ given in \eqref{decompositionblocs}, 
	the equation $\varphi^{*}\textup{Ric}=\textup{Ric}$ reads 
	\begin{gather}\label{ced,dggjfgkjfk}
		\transposee{\begin{bmatrix}
			A(\theta,\dot{\theta}) & B(\theta,\dot{\theta})\\
			C(\theta,\dot{\theta}) & D(\theta,\dot{\theta})
		\end{bmatrix}}
		\begin{bmatrix}
			-\beta(\varphi) & 0\\
			0 & -\beta(\varphi)
		\end{bmatrix}
		\begin{bmatrix}
			A(\theta,\dot{\theta}) & B(\theta,\dot{\theta})\\
			C(\theta,\dot{\theta}) & D(\theta,\dot{\theta})
		\end{bmatrix}
		=\begin{bmatrix}
			-\beta(\theta) & 0\\
			0 & -\beta(\theta)
		\end{bmatrix}
			\nonumber\\
	\Leftrightarrow \,\,\,\,\,\,\,\,\,\, 
		\left \lbrace
		\begin{array}{ccc}\label{cdcnzids}
			{}^{t}\!A(\theta,\dot{\theta})\,(\beta(\varphi))\, A(\theta,\dot{\theta})
			+{}^{t}\!C(\theta,\dot{\theta})\,(\beta(\varphi))\, C(\theta,\dot{\theta}) &=& \beta(\theta),\\
			{}^{t}\!A(\theta,\dot{\theta})\,(\beta(\varphi))\, B(\theta,\dot{\theta})
			+{}^{t}\!C(\theta,\dot{\theta})\,(\beta(\varphi))\, D(\theta,\dot{\theta}) &=& 0,\\
			{}^{t}\!B(\theta,\dot{\theta})\,(\beta(\varphi))\, B(\theta,\dot{\theta})
			+{}^{t}\!D(\theta,\dot{\theta})\,(\beta(\varphi))\, D(\theta,\dot{\theta}) &=& \beta(\theta).
		\end{array}
		\right.
	\end{gather}
	Taking into account the explicit form of $\beta$ in \eqref{dzdnzkfnkned}, the first equation in \eqref{cdcnzids} yields 
	\begin{eqnarray}
		\begin{bmatrix}
			(a_{21})^{2}+(c_{21})^{2} & a_{21}a_{22}+c_{21}c_{22}\\
			a_{22}a_{21}+c_{22}c_{21} & (a_{22})^{2}+(c_{22})^{2}
		\end{bmatrix}
		=\Big(\dfrac{\varphi^{2}}{\theta_{2}}\Big)
		\begin{bmatrix}
			0 & 0\\
			0 & 1
		\end{bmatrix}.
	\end{eqnarray}
	This implies $a_{21}=c_{21}=0$ and $(a_{22})^{2}+(c_{22})^{2}=\big(\frac{\varphi^{2}}{\theta_{2}}\big)^{2}$ which corresponds 
	exactly to the first two equations of the proposition (see \eqref{dl,flsd,ljfedf}). The other two equations are obtained 
	similarly using the third equation in \eqref{cdcnzids}. The lemma follows. 
\end{proof}
	Combining the Cauchy-Riemann equation 
	$\frac{\partial \varphi^{2}}{\partial \dot{\theta}_{2}} =-\frac{\partial \varphi^{4}}{\partial {\theta}_{2}}$ together 
	with the second equation in \eqref{dl,flsd,ljfedf} immediately yields the following lemma. 
\begin{lemma}\label{dfkodio}
	If $\varphi$ is a holomorphic isometry, then $\varphi^{2}$ is a solution of the system of partial differential 
	equations \eqref{equationdelamort}. In particular, it has to be of the form (two possibilities):
	\begin{description}
	\item[$(1)$] $\varphi^{2}(\theta,\dot{\theta})=\dfrac{a^{2}\theta_{2}}{(\dot{\theta}_{2}+b)^{2}+
		(\theta_{2})^{2}},$ \,\,\,$a,b\in \mathbb{R},$ $a\neq 0$, 
	\item[$(2)$] $\varphi^{2}(\theta,\dot{\theta})=a^{2}\theta_{2},$ \,\,\, $a\in \mathbb{R}$, $a\neq 0.$
	\end{description}
\end{lemma}
	From now on, we will assume that $\varphi$ is a holomorphic isometry (in particular $\varphi^{2}$ is given by lemma \ref{dfkodio}). 

	For convenience, let us rewrite explicitly the system of equations in Lemma \ref{fejlfkklwkew}, taking into account 
	Lemma \ref{isometric} and Lemma \ref{dfkodio}.
\begin{lemma}\label{pekodjgfjgk}
	We have: 
	\begin{alignat}{1}
	\Big(\dfrac{\partial \varphi^{1}}{\partial \theta_{1}}\Big)^{2}+
		\Big(\dfrac{\partial \varphi^{1}}{\partial \dot{\theta}_{1}}\Big)^{2} \,\,&=\,\,
		\quad\dfrac{\varphi^{2}}{\theta_{2}},\label{eq111}\\
	 \varphi^{1}\bigg[\dfrac{\partial \varphi^{1}}{\partial \theta_{1}}
		\dfrac{\partial \varphi^{2}}{\partial \theta_{2}}+\dfrac{\partial \varphi^{1}}{\partial \dot{\theta}_{1}}
		\dfrac{\partial \varphi^{2}}{\partial \dot{\theta}_{2}}\Bigg]-
		\varphi^{2}\bigg[\dfrac{\partial \varphi^{1}}{\partial \theta_{1}}
		\dfrac{\partial \varphi^{1}}{\partial \theta_{2}}+\dfrac{\partial \varphi^{1}}{\partial \dot{\theta}_{1}}
		\dfrac{\partial \varphi^{1}}{\partial \dot{\theta}_{2}}\Bigg] \,\, &=
		\,\,\dfrac{\theta_{1}}{(\theta_{2})^{2}}(\varphi^{2})^{2},\label{eq222}\\
	 2\varphi^{1}\bigg[\dfrac{\partial \varphi^{1}}{\partial \theta_{2}}
		\dfrac{\partial \varphi^{2}}{\partial \theta_{2}}+\dfrac{\partial \varphi^{1}}{\partial \dot{\theta}_{2}}
		\dfrac{\partial \varphi^{2}}{\partial \dot{\theta}_{2}}\Bigg]-
		\varphi^{2}\bigg[\Big(\dfrac{\partial \varphi^{1}}{\partial \theta_{2}}\Big)^{2}+
		\Big(\dfrac{\partial \varphi^{1}}{\partial \dot{\theta}_{2}}\Big)^{2}\bigg]+
		\dfrac{\varphi^{2}-(\varphi^{1})^{2}}{\varphi^{2}}\Big(\dfrac{\varphi^{2}}{\theta_{2}}\Big)^{2}\,\,&=\,\,
		\dfrac{\theta_{2}-(\theta_{1})^{2}}{(\theta_{2})^{3}}(\varphi^{2})^{2},\label{eq333}\\
	 \varphi^{1}\bigg[\dfrac{\partial \varphi^{1}}{\partial \theta_{1}}
		\dfrac{\partial \varphi^{2}}{\partial \dot{\theta}_{2}}-\dfrac{\partial \varphi^{1}}{\partial \dot{\theta}_{1}}
		\dfrac{\partial \varphi^{2}}{\partial {\theta}_{2}}\Bigg]+
		\varphi^{2}\bigg[\dfrac{\partial \varphi^{1}}{\partial \dot{\theta}_{1}}
		\dfrac{\partial \varphi^{1}}{\partial \theta_{2}}-\dfrac{\partial \varphi^{1}}{\partial {\theta}_{1}}
		\dfrac{\partial \varphi^{1}}{\partial \dot{\theta}_{2}}\Bigg]\,\, &=\,\,0.\label{eq444}
	\end{alignat}
\end{lemma}
	Since $\varphi^{2}$ doesn't depend on $\theta_{1}$ and $\dot{\theta}_{1}$, it follows from Remark \ref{lddlkgfl} together 
	with \eqref{eq111} that 
	\begin{eqnarray}\label{slkdgkdrkozkfd}
		\varphi^{1}(\theta,\dot{\theta})=r(\theta_{2},\dot{\theta}_{2})\theta_{1}+s(\theta_{2},\dot{\theta}_{2})\dot{\theta}_{1}+
		t(\theta_{2},\dot{\theta}_{2}),
	\end{eqnarray}
	where $r,s,t$ are smooth functions depending on $\theta_{2},\dot{\theta}_{2}$, and such that 
	$r(\theta_{2},\dot{\theta}_{2})^{2}+s(\theta_{2},\dot{\theta}_{2})^{2}=\frac{\varphi^{2}(\theta_{2},\dot{\theta}_{2})}{\theta_{2}}.$
\begin{lemma}\label{labelmiracle!}
	We have:
	\begin{eqnarray}\label{ddkgojfkoxs}
		\Big(\dfrac{\partial r}{\partial \theta_{2}}\Big)^{2}+
		\Big(\dfrac{\partial r}{\partial \dot{\theta}_{2}}\Big)^{2}+
		\Big(\dfrac{\partial s}{\partial \theta_{2}}\Big)^{2}+
		\Big(\dfrac{\partial s}{\partial \dot{\theta}_{2}}\Big)^{2}=
		\dfrac{1}{(\theta_{2})^{2}}\bigg[\dfrac{1}{\theta_{2}}\varphi^{2}-\dfrac{\partial \varphi^{2}}{\partial \theta_{2}}\bigg].
	\end{eqnarray}
	If $\varphi^{2}(\theta_{2},\dot{\theta}_{2})=a^{2}\theta_{2},$ then the right hand side of \eqref{ddkgojfkoxs} is zero. If 
	$\varphi^{2}(\theta_{2},\dot{\theta}_{2})=\frac{a^{2}\theta_{2}}{(\dot{\theta}_{2}+b)^{2}+(\theta_{2})^{2}}$, then the right hand side 
	is $\frac{2a^{2}}{[(\dot{\theta}_{2}+b)^{2}+(\theta_{2})^{2}]^{2}}$.
\end{lemma}
\begin{proof}
	First observe that $r$ and $s$ are harmonic. Indeed, if 
	$\Delta=\frac{\partial^{2}}{\partial \theta_{2}}+\frac{\partial^{2}}{\partial \dot{\theta}_{2}}$, then,
	\begin{eqnarray}
		0=\Delta \varphi^{1}=\theta_{1}\Delta r+\dot{\theta}_{1}\Delta s+
		\Delta t
	\end{eqnarray}
	for all $\theta_{1},\dot{\theta}_{1}\in \mathbb{R},$ which is only possible if $\Delta r=\Delta s=\Delta t=0.$ Now, taking the 
	Laplacian of both side of the equation $r^{2}+s^{2}=\frac{\varphi^{2}}{\theta_{2}}$ yields 
	\begin{eqnarray}
		2A+2r\Delta r+2s\Delta s=
		\Delta\Big(\dfrac{\varphi^{2}}{\theta_{2}}\Big),
	\end{eqnarray}
	where $A$ is the left hand side of \eqref{ddkgojfkoxs}. From this together with the harmonicity of $r,s$ and $\varphi^{2}$,
	one easily obtains \eqref{ddkgojfkoxs}.
\end{proof}
\begin{lemma}
	If $\varphi^{2}(\theta_{2},\dot{\theta}_{2})=a^{2}\theta_{2},$ $a\neq 0$, then there exist $b,c,d\in \mathbb{R}$ and 
	$\varepsilon \in \{+1,-1\}$ such that 
	\begin{eqnarray}\label{jkfjkdjgkjgk}
		\varphi(\theta,\dot{\theta})=\big(\epsilon a\,\theta_{1}+b\,\theta_{2},a^{2}\theta_{2},\epsilon 
		a\,\dot{\theta}_{1}+b\,\dot{\theta}_{2}+c,a^{2}\dot{\theta}_{2}+d\big).
	\end{eqnarray}
	Moreover, every transformation of this form is a holomorphic isometry of $T\mathcal{N}.$
\end{lemma}
\begin{proof}
	Lemma \ref{labelmiracle!} implies that 
	$\varphi^{1}(\theta,\dot{\theta})=r\theta_{1}+s \dot{\theta}_{1}+t(\theta_{2},\dot{\theta}_{2})$, 
	where $r,s\in \mathbb{R}$ are such that $r^{2}+s^{2}=a^{2}.$ Using \eqref{eq444}, one easily obtains $s=0$, $r=\pm a$ and 
	$\frac{\partial t}{\partial \dot{\theta}_{2}}=0.$ From \eqref{eq222}, one also get $t(\theta_{2})=b\,\theta_{2}$ for some constant 
	$b\in \mathbb{R}.$ Hence $\varphi^{1}(\theta,\dot{\theta})=(\pm a)\theta_{1}+b\,\theta_{2}.$ The other components of $\varphi$ 
	are obtained using the Cauchy-Riemann equations. The lemma follows. 
\end{proof}
\begin{remark}
	By changing the sign of $a$ if necessary, one may assume $\varepsilon a=a$ in the above lemma. 
\end{remark}
\begin{remark}
	Written in the complex coordinates $(z_{1},z_{2})\in \mathbb{C}\times i\mathbb{H}$, 
	the transformation in \eqref{jkfjkdjgkjgk} reads $\varphi(z_{1},z_{2})=
	\big((\epsilon a)z_{1}+bz_{2}+ic,\,(\epsilon a)^{2}z_{2}+id\big)$. 
\end{remark}
	Let us now consider the case $\varphi^{2}(\theta_{2},\dot{\theta}_{2})=
	\frac{a^{2}\theta_{2}}{(\dot{\theta}_{2}+b)^{2}+(\theta_{2})^{2}}.$
	In order to find $\varphi^{1}$, we will use the following facts: 
	\begin{description}
		\item[$(1)$] the map $-2\,\eta_{1}\,:\,T\mathcal{N}\rightarrow \mathbb{R},\,\,\,
			(\theta,\dot{\theta)}\mapsto \frac{\theta_{1}}{\theta_{2}}$ is a K\"{a}hler function (see Proposition \ref{fpelpçeglv}, 
			Proposition \ref{proposition proprietes fam exp} and \eqref{dsldgpppqqmdm,}), 
		\item[$(2)$] the composition of a K\"{a}hler function with a holomorphic isometry is a K\"{a}hler function (obvious). 
	\end{description}
	It follows from these two facts 
	that $\frac{\varphi^{1}}{\varphi^{2}}=\frac{r}{\varphi^{2}}\theta_{1}+\frac{s}{\varphi^{2}}\dot{\theta}_{1}+\frac{t}{\varphi^{2}}$ 
	is a K\"{a}hler function on $T\mathcal{N}.$
\begin{lemma}\label{delf,dfkl}
	A function on $T\mathcal{N}$ of the form 
	$R(\theta_{2},\dot{\theta}_{2})\theta_{1}+S(\theta_{2},\dot{\theta}_{2})\dot{\theta}_{1}+T(\theta_{2},\dot{\theta}_{2}),$ 
	where $R,S,T$ are smooth functions, is K\"{a}hler if and only if 
	there exist $C_{1},C_{2},C_{3}\in \mathbb{R}$ such that 
	\begin{eqnarray}
		R=\frac{C_{1}-C_{2}\dot{\theta}_{2}}{\theta_{2}},\,\,\,\,\,\,\,\,S=C_{2},\,\,\,\,\,\,\,\,T=C_{3}.
	\end{eqnarray}
\end{lemma}
\begin{proof}
	Taking into account Proposition \ref{c kdlfldfkdl} together with Proposition \ref{fdkfdgktprioep}, 
	one obtains after a direct calculation that: 
	$R(\theta_{2},\dot{\theta}_{2})\theta_{1}+S(\theta_{2},\dot{\theta}_{2})\dot{\theta}_{1}+
	T(\theta_{2},\dot{\theta}_{2})$ is K\"{a}hler if and only if
	\begin{eqnarray}
		\frac{\partial S}{\partial \theta_{2}}=\frac{\partial S}{\partial \dot{\theta}_{2}}=
			\frac{\partial T}{\partial \theta_{2}}=\frac{\partial T}{\partial \dot{\theta}_{2}}=
		\frac{R}{\theta_{2}}+\frac{\partial R}{\partial \theta_{2}}=
		\frac{S}{\theta_{2}}+\frac{\partial R}{\partial \dot{\theta}_{2}}=0.
	\end{eqnarray}
	Solving these equations exactly yields the lemma. 
\end{proof}
	From Lemma \ref{delf,dfkl}, it follows that there exist $C_{1},C_{2},C_{3}\in \mathbb{R}$ such that 
	\begin{eqnarray}\label{slljsgjkrfns}
		\frac{r}{\varphi^{2}}=\frac{C_{1}-C_{2}\dot{\theta}_{2}}{\theta_{2}},\,\,\,\,\,\,\,\,\frac{s}{\varphi^{2}}
		=C_{2},\,\,\,\,\,\,\,\,\frac{t}{\varphi^{2}}=C_{3}.
	\end{eqnarray}
	Now, rewriting the equation $r^{2}+s^{2}=\frac{\varphi^{2}}{\theta_{2}}$ using \eqref{slljsgjkrfns}
	leads to an equality of two polynomials in $\theta_{2}$ and $\dot{\theta}_{2}$:
	\begin{eqnarray}
		(C_{1})^{2}-2C_{1}C_{2}\dot{\theta}_{2}+(C_{2})^{2}(\dot{\theta}_{2})^{2}+(C_{2})^{2}(\theta_{2})^{2}
		=\frac{b^{2}+2b\dot{\theta}_{2}+(\dot{\theta}_{2})^{2}+(\theta_{2})^{2}}{a^{2}},
	\end{eqnarray}
	from which we get a system of equations which is equivalent to $C_{1}=-b\,C_{2}$ and $(C_{2})^{2}=\frac{1}{a^{2}}$.
	Since there is no constraints on the sign of $a$, we can assume $C_{2}=\frac{1}{a}$ and $C_{1}=-\frac{b}{a}.$ Returning 
	to \eqref{slkdgkdrkozkfd}, and setting $c:=aC_{3}$ for convenience, a direct calculation gives :
	\begin{eqnarray}
		\varphi^{1}(\theta,\dot{\theta})=a\,\dfrac{-(\dot{\theta}_{2}+b)\theta_{1}+
		(\dot{\theta}_{1}+c)\theta_{2}}{(\dot{\theta}_{2}+b)^{2}+(\theta_{2})^{2}},\,\,\,\,\,\,\,\,\,
		\varphi^{2}(\theta,\dot{\theta})=
	\frac{a^{2}\theta_{2}}{(\dot{\theta}_{2}+b)^{2}+(\theta_{2})^{2}},
	\end{eqnarray}
	where $a,b,c\in \mathbb{R},$ $a\neq 0.$ Finally, solving the Cauchy-Riemann equations corresponding to the holomorphic functions 
	$\varphi^{1}+i\varphi^{3}$ and $\varphi^{2}+i\varphi^{4}$ gives 
	\begin{eqnarray}
		\varphi^{3}(\theta,\dot{\theta})=-a\frac{(\dot{\theta}_{1}+c)(\dot{\theta}_{2}+b)+
		\theta_{1}\theta_{2}}{(\dot{\theta}_{2}+b)^{2}+(\theta_{2})^{2}}+d,\,\,\,\,\,\,\,\,\,
		\varphi^{4}(\theta,\dot{\theta})=-\frac{a^{2}(\dot{\theta}_{2}+b)}{(\dot{\theta}_{2}+b)^{2}+(\theta_{2})^{2}}+e,
	\end{eqnarray}
	where $d,e\in \mathbb{R}.$ In terms of the complex variables $z_{k}=\theta_{k}+i\dot{\theta}_{k}$, this can be rewritten 
	\begin{eqnarray}
		(\varphi^{1}+i\varphi^{3})(z_{1},z_{2})=-ia\,\frac{z_{1}+ic}{z_{2}+ib}+id,\,\,\,\,\,\,\,\,
		(\varphi^{2}+i\varphi^{4})(z_{1},z_{2})=\dfrac{a^{2}}{z_{2}+ib}+ie.
	\end{eqnarray}
	Collecting our results, we obtain the following lemma.

\begin{lemma}\label{theoremrrr}
	Let $\varphi$ be a diffeomorphism of 
	$T\mathcal{N}\cong \mathbb{C}\times i\mathbb{H}$. Then $\varphi$ 
	is a holomorphic isometry if and only if it has the following form (two possibilities):
	\begin{alignat}{4}
		\varphi_{1}(z_{1},z_{2}) \,\,=&\,\, \Big(-ia\,\frac{z_{1}+ic}{z_{2}+ib}+id,\,\dfrac{a^{2}}{z_{2}+ib}+ie\Big), 
		\quad& \quad a,b,c,d,e\in \mathbb{R},\,\,a\neq 0,\label{dlf,lsdklzd}\\
		\varphi_{2}(z_{1},z_{2})\,\, =&\,\, \big(az_{1}+bz_{2}+ic,\,a^{2}z_{2}+id\big),
		\quad& \quad a,b,c,d\in \mathbb{R},\,\,a\neq 0.\label{dsld,lsgkogvj}
	\end{alignat}
\end{lemma}
	To conclude the proof of Theorem \ref{theoremnnn}, recall that the map $f\,:\,\mathbb{C}\times i\mathbb{H}\rightarrow 
	\mathbb{H}\times \mathbb{C},\,\,\,
	(z_{1},z_{2})\mapsto (-iz_{2},iz_{1})$ is a biholomorphic isometry (see \eqref{eljlejgler}) 
	and that the action of $\textup{ASp}(2,\mathbb{R})$ 
	on $\mathbb{H}\times \mathbb{C}$ is given by 
	\begin{eqnarray}\label{lefdlgjdlgjd}
	\big(\Big[\begin{matrix}
	a & b \\
	c & d 
	\end{matrix}\Big],(\lambda,\mu)\big)\cdot (\tau,z)=\Big(\frac{a\tau+b}{c\tau+d},\frac{z+\lambda \tau+\mu}{c\tau+d}\Big).
	\end{eqnarray}
	Having this in mind, we observe after a direct calculation that for $(\tau,z)\in \mathbb{H}\times \mathbb{C}$,
	\begin{alignat}{5}
		(f\circ \varphi_{1}\circ f^{-1})(\tau,z)\,\,=&\,\,\Big(-\tfrac{1}{a}\Big[\begin{matrix} \label{dewmdwkpppoo}
			e & eb{-}a^{2}\\
			1 & b
			\end{matrix}\Big],
			\big(\tfrac{d}{a},-c+\tfrac{bd}{a},0\big)
			\Big)\cdot (\tau,z), \\
		(f\circ \varphi_{2}\circ f^{-1})(\tau,z)\,\,=& \,\,\Big(\tfrac{1}{a}\Big[\begin{matrix}\label{ajmdlfennn}
			a^{2} & d\\
			0 & 1
			\end{matrix}\Big],
			\big(-\tfrac{b}{a},-\tfrac{c}{a},0\big)\Big)\cdot (\tau,z),
		\end{alignat}
	where $\varphi_{1}$ and $\varphi_{2}$ are defined in \eqref{dlf,lsdklzd} and \eqref{dsld,lsgkogvj} respectively. From this it 
	follows that $f\circ \varphi\circ f^{-1}\in \textup{ASp}(2,\mathbb{R})$ for all holomorphic isometries $\varphi$ of 
	$T\mathcal{N}$, which shows 
	that the group of holomorphic isometries of $T\mathcal{N}$ is included in 
	$\textup{ASp}(2,\mathbb{R})$. The converse inclusion being obviously true (by inspection of \eqref{dewmdwkpppoo} and \eqref{ajmdlfennn}), 
	the equality holds. \\

	Let us now derive a few consequences. Consider the following subgroup of $\textup{SL}(2,\mathbb{R})$ :
	\begin{eqnarray}
		K\, :=\, \Big\{ 
		\Big[\begin{matrix}
			a & b\\
			0 & \tfrac{1}{a}
		\end{matrix}\Big]\in \textup{Mat}(2,\mathbb{R})
		\,\,\,\Big\vert\,\,a,b\in \mathbb{R},\,\,a\neq0
		\Big\}.
	\end{eqnarray}
	Clearly, $K$ is a $2$-dimensional Lie group having two connected components (according to the sign of $a$). 
	We denote by $K_{0}$ the connected component of $K$ containing the identity. Since $K_{0}$ is a subgroup of 
	$\textup{SL}(2,\mathbb{R})$, one can form the semi-direct product $K_{0}\ltimes
	\mathbb{R}^{2};$ it is naturally a subgroup of $\textup{SL}(2,\mathbb{R})\ltimes \mathbb{R}^{2}=\textup{ASp}(2,\mathbb{R})$.
\begin{proposition}\label{df,dkfjekrfjkefjkrt}
	In this situation,
	\begin{description}
		\item[$(i)$] The actions of $G^{J}(\mathbb{R})$, $\textup{ASp}(2,\mathbb{R})$ 
		and $K_{0}\ltimes  \mathbb{R}^{2}$ 
		on $T\mathcal{N}$ are transitive,
		\item[$(ii)$] The isotropy subgroups of $o:=(i,0)\in \mathbb{H}\times \mathbb{C}$ 
			relative to the actions of 
			$G^{J}(\mathbb{R})$, $\textup{ASp}(2,\mathbb{R})$ and $K_{0}\ltimes \mathbb{R}^{2}$
			are isomorphic to $\textup{SO}(2)\times \mathbb{R}$, $\textup{SO}(2)$ and $\{0\}$, respectively. 
	\end{description}
	Therefore, $T\mathcal{N}$ is a homogeneous K\"{a}hler manifold and we have the identifications: 
	\begin{eqnarray}
			T\mathcal{N}\cong G^{J}(\mathbb{R})/\textup{SO}(2)\times \mathbb{R}\cong \textup{ASp}(2,\mathbb{R})/\textup{SO}(2)
			\cong K_{0}\ltimes \mathbb{R}^{2}.
	\end{eqnarray}
\end{proposition}
\begin{proof}
	By a direct calculation. 
\end{proof}
\begin{corollary}
	$T\mathcal{N}$ itself is a Lie group (isomorphic to $K_{0}\ltimes \mathbb{R}^{2}$) whose K\"{a}hler structure is left-invariant.
\end{corollary}
	Let us now discuss the whole group of isometries of $T\mathcal{N}.$ To this end, we introduce the following 
	group 
	\begin{eqnarray}
		\textup{SL}^{\pm}(2,\mathbb{R}):=\Big\{\Big[\begin{matrix}
		a & b\\
		c & d
		\end{matrix}
		\Big]\in \textup{Mat}(2,\mathbb{R})\,\Big\vert\,\,ad-bc=\pm1 \Big\}. 
	\end{eqnarray}
	Since $\textup{SL}^{\pm}(2,\mathbb{R})$ acts linearly on the right on $\mathbb{R}^{2}$, one has the semi-direct product 
	$\textup{SL}^{\pm}(2,\mathbb{R})\ltimes \mathbb{R}^{2}$, with multiplication $(M_{1},X_{1})\cdot (M_{2},X_{1})
	=(M_{1}M_{2},X_{2}+X_{1}\cdot M_{2})$. We define an action of $\textup{SL}^{\pm}(2,\mathbb{R})$ on $\mathbb{H}\times \mathbb{C}$ 
	as follows:
	\begin{eqnarray}
		\Big(\Big[\begin{matrix}
			a & b\\
			c & d
			\end{matrix}\Big],
			\big(\lambda,\mu\big)\Big)\cdot (\tau,z):=
			\left \lbrace
			\begin{array}{ccc}
				\Big(\displaystyle\frac{a\tau+b}{c\tau+d},\frac{z+\lambda \tau+\mu}{c\tau+d}\Big)  &  \textup{if} & 
				ad-bc=\,\,1,\\[1em]
				\Big(\displaystyle\frac{a\bar{\tau}+b}{c\bar{\tau}+d},\frac{\bar{z}+\lambda \bar{\tau}+\mu}{c\bar{\tau}+d}\Big)
				 & \textup{if} &\textup\,\, ad-bc=-1,
			\end{array}
			\right.\label{quedirfewsswsd}
	\end{eqnarray}
	where $\bar{z}$ denotes the complex conjugate of $z\in\mathbb{C}.$ 

	Since this action is effective, one can regard $\textup{SL}^{\pm}(2,\mathbb{R})\ltimes \mathbb{R}^{2}$ as a subgroup of 
	$\textup{Diff}(\mathbb{H}\times \mathbb{C})\cong \textup{Diff}(T\mathcal{N}).$
\begin{theorem}\label{ef,kdgj,kdgf,k}
	The group of isometries of $T\mathcal{N}$ (not necessarily holomorphic) is the semi-direct product 
	$\textup{SL}^{\pm}(2,\mathbb{R})\ltimes \mathbb{R}^{2}.$
\end{theorem}
	The proof of Theorem \ref{ef,kdgj,kdgf,k} is based on the following result which is due to Kulkarni (see \cite{Kulkarni}).
\begin{proposition}\label{ksnfksgnkt}
	Let $N_{1}$ and $N_{2}$ be two connected K\"{a}hler manifolds with corresponding holomorphic 
	sectional curvature functions\footnote{As defined in footnote \ref{ejfkejktre}.} 
	$H_{1}$ and $H_{2}.$ Suppose that the real dimension of $N_{1}$ is greater than 4 and that
	there exists a diffeomorphism $f\,:\,N_{1}\rightarrow N_{2}$ such that $f^{*}H_{2}=H_{1}.$ Then either $H_{1}=H_{2}
	=const.$ or $f$ is a holomorphic or anti-holomorphic isometry. 
\end{proposition}
\begin{corollary}[of Proposition \ref{ksnfksgnkt}]\label{elf,ef,,r}
	Let $N$ be a connected K\"{a}hler manifold whose holomorphic sectional curvature is not constant, and whose real dimension 
	is greater than 4. Then every isometry of $N$ is either holomorphic or anti-holomorphic.  
\end{corollary}
\begin{proof}[Proof of Theorem \ref{ef,kdgj,kdgf,k}]
	In terms of the variables $(z_{1},z_{2})\in\mathbb{C}\times i\mathbb{H}$, it not difficult to see that the map 
	$T\mathcal{N}\rightarrow T\mathcal{N},\,\,(z_{1},z_{2})\mapsto (\bar{z}_{1},\bar{z}_{2})$ is an 
	anti-holomorphic isometry of $T\mathcal{N}$ (this is actually a general feature of Dombrowski's construction). In terms 
	of the variables $(\tau,z)=(-iz_{2},iz_{1})\in \mathbb{H}\times \mathbb{C}$, this means 
	that the map $(\tau,z)\mapsto (-\bar{\tau},-\bar{z})$ 
	is an anti-holomorphic isometry of $\mathbb{H}\times \mathbb{C}$. 
	Therefore, there is a 1-to-1 correspondence between the set of holomorphic isometries 
	and the set of anti-holomorphic isometries of $T\mathcal{N}$ which is given by 
	$\varphi(\tau,z)\mapsto \varphi(-\bar{\tau},-\bar{z}).$ From this, it is easy to see that \eqref{quedirfewsswsd} 
	exhausts all the possible holomorphic and anti-holomorphic isometries of $T\mathcal{N}$ (and nothing else). 
	But according to Corollary \ref{elf,ef,,r}, this is already the whole isometry group of $T\mathcal{N}.$ The proposition follows. 
\end{proof}
	Let us conclude this section with a discussion on the Lie group structure of the group of isometries of $T\mathcal{N}.$ To this end, 
	we recall the following result which is due to Myers and Steenrod \cite{Myers} (see also 
	\cite{Kobayashi-transformation} or \cite{Kobayashi-Nomizu} for a modern proof). 
\begin{proposition}\label{dkfjkejfkegjkr}
	Let $M$ be a connected Riemannian manifold. Then 
	the group $\textup{Isom}(M)$ of isometries of $M$ is a Lie group with respect to the compact-open 
	topology\footnote{Let $X,Y$ be two metric spaces, and let 
	$C^{0}(X,Y)$ be the space of continuous maps between $X$ and $Y$. Then, the \textit{compact-open topology} is the topology on 
	$C^{0}(X,Y)$ whose subbases is given by all the subsets of the form $W(K,U):=\{f\in C^{0}(X,Y)\,\big\vert\,f(K)\subseteq U\}$, 
	where $K$ is a compact subset of $X$ and $U$ is an open subset of $Y$.} in $M$. Moreover, the natural action of 
	$\textup{Isom}(M)$ on $M$ is smooth.
\end{proposition}
	Let $M$ be a manifold acted upon by a Lie group $G$ with Lie algebra $\mathfrak{g}$. Given $\xi\in \mathfrak{g}$, the 
	\textit{fundamental vector field} $\xi_{M}$ is the vector field on $M$ which 
	is defined, for $p\in M,$ by 
	\begin{eqnarray}
		(\xi_{M})_{p}:=\dfrac{d}{dt}\,\bigg\vert_{0}\,\textup{exp}(t\xi)\cdot p,
	\end{eqnarray} 
	where $\textup{exp}\,:\,\mathfrak{g}\rightarrow G$ is the standard exponential map. Observe that fundamental vector fields only 
	depend on the action of $G_{0}$ on $M$, where $G_{0}$ is the connected component of $G$ containing the identity. 
	If $G$ acts via isometries on a Riemannian manifold $M$, then every fundamental vector field $\xi_{M}$ 
	is a Killing vector field. We denote by $\mathfrak{i}(M)$ the space of Killing vector 
	fields of a Riemannian manifold $M;$ it is a Lie algebra for the Lie bracket 
	of vector fields. 
\begin{proposition}[Complement of Proposition \ref{dkfjkejfkegjkr}]\label{flkflekglrtkg}
	Let $M$ be a connected Riemannian manifold with isometry group $\textup{Isom}(M)$ and Lie algebra $\mathfrak{g}$. 
	If $M$ is complete, then the map 
	$\phi\,:\,\mathfrak{g}\rightarrow \mathfrak{i}(M),\,\,\,\xi\mapsto \xi_{M}$
	is an anti-isomorphism of Lie algebras, that is, it is an isomorphism of vector spaces satisfying 
	\begin{eqnarray}
		\phi([\xi,\eta])=-[\phi(\xi),\phi(\eta)]
	\end{eqnarray}
	for all $\xi,\eta\in \mathfrak{g}.$
\end{proposition}
	If a Lie group $G$ acts effectively on a manifold $M$, then there are a priori two topologies on $G$ : 
	the intrinsic topology of $G$, and the compact-open topology coming 
	from the injection $G\rightarrow \textup{Diff}(M)$. If the image of $G$ coincides with $\textup{Isom}(M)$ in $\textup{Diff}(M)$, 
	like in Theorem \ref{ef,kdgj,kdgf,k}, then we have the following result. 
\begin{lemma}\label{felfkelfkelr}
	Let $\Phi\,:\,G\times M\rightarrow M$ be an action of a Lie group $G$ on a connected and complete Riemannian manifold $M$. 
	Suppose that this action is smooth, effective and 
	that $\textup{Isom}(M)=\{\Phi_{g}\,\vert\,g\in G\}$, 
	where $\Phi_{g}\,:\,M\rightarrow M,\,\,\,p\mapsto \Phi(g,p)$. Then the map $G\rightarrow \textup{Isom}(M),\,\,g\mapsto \Phi_{g}$ 
	is an isomorphism of Lie groups (here $\textup{Isom}(M)$ is endowed with 
	the Lie group structure described in Proposition \ref{dkfjkejfkegjkr}). 
\end{lemma}
\begin{proof}
	It is based on the following result: if $(\varphi_{n})_{n\in \mathbb{N}}$ is a sequence of 
	isometries of $M$ such that $\varphi_{n}(p)$ converges to $\varphi(p)$ for all $p\in M$, where $\varphi$ is a fixed isometry, then 
	$\varphi_{n}$ converges to $\varphi$ for the compact-open topology 
	(see \cite{Kobayashi-Nomizu}, Lemma 5, Chapter 1 and Theorem 3.10, Chapter 4). 
	From this together with the continuity of $\Phi\,:\,G\times M\rightarrow M$, one sees that 
	$G\rightarrow \textup{Isom}(M),\,\,g\mapsto \Phi_{g}$ is a continuous and bijective homomorphism of topological groups. Since 
	continuous homomorphisms of Lie groups are automatically smooth, the map $g\mapsto \Phi_{g}$ is smooth. By the inverse 
	function theorem, its inverse is also smooth. The lemma follows. 
 \end{proof}
	Combining Theorem \ref{ef,kdgj,kdgf,k}, Proposition \ref{flkflekglrtkg}, Lemma \ref{felfkelfkelr} and the fact that 
	$(\textup{SL}^{\pm}(2,\mathbb{R})\ltimes \mathbb{R}^{2})_{0}=\textup{ASp}(2,\mathbb{R})$, 
	we obtain the following result. 
\begin{proposition}\label{kgkn,fkgrk}
	Let $\mathfrak{asp}(2,\mathbb{R})$ be the Lie algebra of $\textup{ASp}(2,\mathbb{R})$. Then 
	the map $\mathfrak{asp}(2,\mathbb{R})\rightarrow \mathfrak{i}(T\mathcal{N}),\,\,\,\xi\mapsto \xi_{T\mathcal{N}}$, is an 
	anti-isomorphism of Lie algebras.
\end{proposition}

\subsection{K\"{a}hler functions and momentum map}\label{que direndknkfgnkfg}
	Let $\mathfrak{g}^{J}$, $\mathfrak{sl}(2,\mathbb{R})$ and $\mathfrak{h}$ denote respectively
	the Lie algebras of $G^{J}(\mathbb{R}),$ $\textup{SL}(2,\mathbb{R})$ and $\textup{Heis}(\mathbb{R}).$ We recall that 
	$\mathfrak{sl}(2,\mathbb{R})$ is the space of $2\times 2$ real matrices of trace 0,
	\begin{eqnarray}
		\mathfrak{sl}(2,\mathbb{R})= \Big\{\Big[\begin{matrix}
		\alpha & \beta\\
		\gamma & \delta
		\end{matrix}
		\Big]\in \textup{Mat}(2,\mathbb{R})\,\Big\vert\,\,\alpha+\delta=0\Big\},
	\end{eqnarray}
	and that $\mathfrak{h}$ can be identified with $\mathbb{R}^{2}\times \mathbb{R}$ endowed with the Lie bracket
	\begin{eqnarray}
		\big[(\xi,r),(\eta,s)\big]=
		\big(0,2\,\Omega(\xi,\eta)\big), 
	\end{eqnarray}
	where $\xi,\eta\in \mathbb{R}^{2}$, $r,s\in \mathbb{R}$ and where $\Omega(\xi,\eta)=\xi_{1}\eta_{2}-\xi_{2}\eta_{1}.$ 
	In the sequel, we shall use 
	the following basis for $\mathfrak{sl}(2,\mathbb{R})$, 
	\begin{alignat}{5}
		F:=&\begin{bmatrix}
		0 & 1\\
		0 & 0
		\end{bmatrix},
		\quad & \quad 
		G:=&\begin{bmatrix}
		0 & 0\\
		1 & 0
		\end{bmatrix},
		\quad& \quad 
		H:=&\begin{bmatrix}
		1 & 0\\
		0 & -1
		\end{bmatrix},
	\end{alignat}
	and denote by $\{P,Q,R\}$ the canonical basis  of $ \mathfrak{h}\cong \mathbb{R}^{2}\times \mathbb{R}\cong \mathbb{R}^{3}$,
	\begin{alignat}{5}
		P:=&(1,0,0), 
		\quad & \quad 
		Q:=&(0,1,0),
		\quad & \quad
		R:=& (0,0,1).
	\end{alignat}
	The Lie algebra $\mathfrak{g}^{J}$ of the Jacobi group $G^{J}(\mathbb{R})$ 
	is the semi-direct product $\mathfrak{g}^{J}=\mathfrak{sl}(2,\mathbb{R})\ltimes \mathfrak{h},$ that is, it is the 
	Cartesian product $\mathfrak{sl}(2,\mathbb{R})\times \mathfrak{h}$ endowed with the Lie bracket 
	\begin{eqnarray}\label{efkldfjljftj}
		\big[(A,\xi,r),(B,\eta,s)\big]=\big([A,B],\xi B -\eta A,2\,\Omega(\xi,\eta)\big), 
	\end{eqnarray}
	where $A,B\in \mathfrak{sl}(2,\mathbb{R})$, $\xi,\eta\in \mathbb{R}^{2}$, $r,s\in \mathbb{R}$, and 
	where $[A,B]=AB-BA$ is the usual commutator of matrices. 
	By construction, $\mathfrak{sl}(2,\mathbb{R})$ and $\mathfrak{h}$ are Lie subalgebras of 
	$\mathfrak{g}^{J}$, therefore $\{F,G,H,P,Q,R,\}$ can be regarded as a basis for $\mathfrak{g}^{J}$. 
	A direct calculation using \eqref{efkldfjljftj} gives the following commutation relations (see also \cite{Berndt98}):
\begin{alignat}{5}
	[F,G]\,=&\, H,  \quad & \quad[F,Q]\,=&\,0, \quad&\quad [G,Q]\,=&\,-P,    \quad&\quad [P,Q]\,=&\, 2R, \label{kdlsdklwsd}   \\
	[F,H]\,=&\, -2F, \quad& \quad[G,H]\,=&\,2G,  \quad& \quad[H,P]\,=&\,-P,  \quad&\quad [R,\,.\,]\,=&\,0.    \\
	[F,P]\,=&\, -Q,  \quad&\quad [G,P]\,=&\,0, \quad&\quad [H,Q]\,=&\,Q,     \quad&\quad \textbf{} \label{kdlsdklwsd2}
\end{alignat}

	Let us now recall a few basic definitions related to Lie group actions (see \cite{Marsden-Ratiu}). 
	Let $(M,\omega)$ be a symplectic manifold acted upon by a Lie 
	group $G$ with Lie algebra $\mathfrak{g}$. Let $\mathfrak{g}^{*}$ be the dual of the Lie algebra 
	$\mathfrak{g}$. A \textit{momentum map} is a smooth map 
	$\textup{\textbf{J}}\,:\,M\rightarrow \mathfrak{g}^{*}$
	satisfying $\xi_{M}=X_{\textup{\textbf{J}}^{\xi}}$ for all $\xi\in \mathfrak{g}$, where $\xi_{M}$ is the fundamental vector 
	field of $\xi$ and where $\textup{\textbf{J}}^{\xi}$ is 
	the function $M\rightarrow \mathbb{R}$ defined by $\textup{\textbf{J}}^{\xi}(p):=\textup{\textbf{J}}(p)(\xi)$ (here 
	$X_{\textup{\textbf{J}}^{\xi}}$ denotes the Hamiltonian vector field associated to $\textup{\textbf{J}}^{\xi})$. Let us 
	denote explicitly the action of $G$ on $M$ 
	by $\Phi\,:\,G\times M\rightarrow M.$ Given $g\in G$, we also denote by $\Phi_{g}$ the diffeomorphism 
	$M\rightarrow M,\,\,p\mapsto \Phi(g,p).$ In this situation, a momentum map is said to be \textit{equivariant} if it satisfies 
	\begin{eqnarray}
		\textup{Ad}^{*}(g)\circ \textup{\textbf{J}}=\textup{\textbf{J}}\circ \Phi_{g}
	\end{eqnarray}
	for all $g\in G$, where $\textup{Ad}^{*}$ is the coadjoint representation\footnote{
	If $\langle\,,\,\rangle$ is the natural pairing between 
	$\mathfrak{g}$ and $\mathfrak{g}^{*}$, then the coadjoint representation is defined via the formula 
	$\langle \textup{Ad}^{*}(g)\alpha,\xi\rangle=\langle \alpha, \textup{Ad}(g^{-1})\xi\rangle$, 
	where $\xi\in \mathfrak{g}$, $\alpha\in \mathfrak{g}^{*}$, and where 
	$\textup{Ad}$ is the usual adjoint representation of $G$.} of $G.$ Equivalently,  $\textup{\textbf{J}}$ is equivariant 
	if $\textup{\textbf{J}}^{\xi}\circ \Phi_{g}=\textup{\textbf{J}}^{\textup{Ad}(g^{-1})\xi}$ for all $g\in G$ and all 
	$\xi\in \mathfrak{g}$.  

	Having this in mind, let $C^{\infty}(\mathbb{S}^{J})$ denote the space of smooth functions on the Siegel-Jacobi space 
	$\mathbb{S}^{J}.$ Using the symplectic coordinates $(\eta,\dot{\theta})$ on $\mathbb{S}^{J}\cong T\mathcal{N}$ 
	(see Proposition \ref{fzlfkcslkflef} 
	and Proposition \ref{proposition proprietes fam exp}), 
	we define a linear map $\psi\,:\,\mathfrak{g}^{J}\rightarrow C^{\infty}(\mathbb{S}^{J})$ as follows:
\begin{alignat}{5}\label{flfelfjlefjoe}
	F\,\,\mapsto&\,\, -\eta_{2},  \,\,\quad&\quad\,\,
	P\,\,\mapsto&\,\, \tfrac{1}{2}\dot{\theta}_{1}+\eta_{1}\dot{\theta}_{2},\\
	G\,\,\mapsto&\,\, \tfrac{1}{4}(\dot{\theta}_{1})^{2}+\eta_{2}(\dot{\theta}_{2})^{2}+
	\eta_{1}\dot{\theta}_{1}\dot{\theta}_{2}-\tfrac{1}{4((\eta_{1})^{2}-\eta_{2})}, \,\,\quad&\quad\,\,
	Q\,\,\mapsto& \,\,\eta_{1},\\
	H\,\,\mapsto& \,\,-\eta_{1}\dot{\theta}_{1}-2\,\eta_{2}\dot{\theta}_{2},\,\,\quad&\quad\,\,
	R\,\,\mapsto&\,\, -\tfrac{1}{4}.\label{lelgjrgjrl}
\end{alignat}
\begin{remark}
	Observe that the last term of $\psi(G)$ can be rewritten $\tfrac{1}{4((\eta_{1})^{2}-\eta_{2})}=\tfrac{\theta_{2}}{2}.$
\end{remark}
\begin{proposition}\label{eknfkd,lsdnszz}
	For every $L\in \mathfrak{g}^{J}$, the Hamiltonian vector field of $\psi(L)$ coincide with the fundamental vector field 
	generated by $L$, that is, $X_{\psi(L)}=L_{\mathbb{S}^{J}}$. Therefore, 
	the map $\textup{\textbf{J}}\,:\,\mathbb{S}^{J}\rightarrow (\mathfrak{g}^{J})^{*}$ defined by 
	\begin{eqnarray}
		\textup{\textbf{J}}(p)(L):=\psi(L)(p),\,\,\,\,\,\,\,(p\in \mathbb{S}^{J},\,\,\,L\in \mathfrak{g}^{J})
	\end{eqnarray}
	is a momentum map. 
\end{proposition}
\begin{proof}
	Using the relations $\eta_{1}=-\tfrac{\theta_{1}}{2\theta_{2}}$ and 
	$\eta_{2}=\tfrac{(\theta_{1})^{2}-2\theta_{2}}{4(\theta_{2})^{2}}$, one can rewrite 
	the functions $\psi(L)$ in terms of the coordinates $(\theta,\dot{\theta})$, and 
	compute their Hamiltonian vector fields $X_{\psi(L)}$ 
	via the formula $(X_{f})_{(\theta,\dot{\theta})}=
	h^{ij}\tfrac{\partial f}{\partial \dot{\theta}_{i}}\tfrac{\partial}{\partial \theta_{j}}-
	h^{ij}\tfrac{\partial f}{\partial {\theta}_{i}}\tfrac{\partial}{\partial \dot{\theta}_{j}}$. One obtains: 
	\begin{alignat}{5}
	(X_{\psi(F)})_{(\theta,\dot{\theta})}\,\,=&\,\,(0,0,0,1),   \,\,\quad&\quad\,\,
	(X_{\psi(P)})_{(\theta,\dot{\theta})}\,\,=&\,\,(-\theta_{2},0,-\dot{\theta}_{2},0),\nonumber\\
	(X_{\psi(G)})_{(\theta,\dot{\theta})}\,\,=&\,\, (-\dot{\theta}_{1}\theta_{2}-\theta_{1}\dot{\theta}_{2},-2\theta_{2}\dot{\theta}_{2},
	-\dot{\theta}_{1}\dot{\theta}_{2}+\theta_{1}\theta_{2},(\theta_{2})^{2}-(\dot{\theta}_{2})^{2}),	 \,\,\quad&\quad\,\,
	(X_{\psi(Q)})_{(\theta,\dot{\theta})}\,\,=&\,\,(0,0,-1,0),\nonumber \\
	(X_{\psi(H)})_{(\theta,\dot{\theta})}\,\,=&\,\,(\theta_{1},2\theta_{2},\dot{\theta}_{1},2\dot{\theta}_{2}), \,\,\quad&\quad\,\,
	(X_{\psi(R)})_{(\theta,\dot{\theta})}\,\,=&\,\,(0,0,0,0).\nonumber
	\end{alignat}
	On the other hand, the fundamental 
	vector fields associated to $F,G,H,P,Q,R$ can be computed in the $(\theta,\dot{\theta})$-coordinates using \eqref{lefdlgjdlgjd} 
	and the relation $(\tau,z)=(-iz_{2},iz_{1})=(-i(\theta_{2}+i\dot{\theta}_{2}),i(\theta_{1}+i\dot{\theta}_{1}))=
	(\dot{\theta}_{2}-i\theta_{2},-\dot{\theta}_{1}+i\theta_{1}).$ By comparing the results, 
	one sees that $X_{\psi(L)}=L_{\mathbb{S}^{J}}$ for all $L\in \mathfrak{g}^{J}$. The proposition follows.
\end{proof}
	Since $G^{J}(\mathbb{R})$ acts via isometries on $\mathbb{S}^{J}$, it follows from the relation $X_{\psi(L)}=L_{\mathbb{S}^{J}}$ 
	that $X_{\psi(L)}$ is a Killing vector field for all $L\in \mathfrak{g}^{J}$, which means that 
	$\psi(L)$ is a K\"{a}hler function for all $L\in \mathfrak{g}^{J}$. One can thus regard $\psi$ as a map 
	$\psi\,:\,\mathfrak{g}^{J}\rightarrow \mathscr{K}(\mathbb{S}^{J})$, where $\mathscr{K}(\mathbb{S}^{J})$ is the Lie algebra 
	of K\"{a}hler functions on $\mathbb{S}^{J}$ (see Section \ref{Kalhlerpouettepouette}). 
\begin{proposition} 
	The map $\psi\,:\,\mathfrak{g}^{J}\rightarrow \mathscr{K}(\mathbb{S}^{J})$ is a Lie algebra isomorphism.
\end{proposition}
\begin{proof}
	The fact that $\psi\,:\,\mathfrak{g}^{J}\rightarrow \mathscr{K}(\mathbb{S}^{J})$ is an injective homomorphism 
	of Lie algebras follows from a direct calculation. For dimensional reasons, it is also surjective. 
	Indeed, one has $\textup{dim}(\mathfrak{i}(\mathbb{S}^{J}))=5$ (see Proposition \ref{kgkn,fkgrk}), and the kernel of the linear map 
	$\phi\,:\,\mathscr{K}(\mathbb{S}^{J})\rightarrow \mathfrak{i}(\mathbb{S}^{J}),\,\,\,f\mapsto X_{f}$ 
	is isomorphic to $\mathbb{R}$ (by connectedness of $\mathbb{S}^{J}$). Thus, 
	\begin{eqnarray}
	\textup{dim}(\mathscr{K}(\mathbb{S}^{J}))-1=
	\textup{dim}(\phi(\mathscr{K}(\mathbb{S}^{J})))\leq \textup{dim}(\mathfrak{i}(\mathbb{S}^{J}))=5. 
	\end{eqnarray}
	Therefore $\textup{dim}(\mathscr{K}(\mathbb{S}^{J}))\leq 6$. Since $\psi(\mathfrak{g}^{J})$ 
	is a 6-dimensional subspace of $\mathscr{K}(\mathbb{S}^{J})$, this implies 
	$\psi(\mathfrak{g}^{J})=\mathscr{K}(\mathbb{S}^{J})$. The proposition 
	follows. 
\end{proof}
\begin{corollary}\label{cofejdgnksdkzpefj}
	A smooth function $f\,:\,\mathbb{S}^{J}\rightarrow \mathbb{R}$ is a K\"{a}hler function 
	if and only if there exists $L\in \mathfrak{g}^{J}$ such that $f=\textup{\textbf{J}}^{L}.$
\end{corollary}
\begin{corollary}\label{kgkfg,rkg,}
	The momentum map $\textup{\textbf{J}}\,:\,\mathbb{S}^{J}\rightarrow (\mathfrak{g}^{J})^{*}$ is equivariant. \end{corollary}
\begin{proof}
	It is a consequence of the connectedness of $G^{J}(\mathbb{R})$ and the fact that 
	$\psi\,:\,\mathfrak{g}^{J}\rightarrow \mathscr{K}(\mathbb{S}^{J})$ is a Lie algebra homomorphism 
	(see \cite{Marsden-Ratiu}, Chapter 12).
\end{proof}
\begin{remark}
	If we denote explicitly the action of $G^{J}(\mathbb{R})$ on $\mathbb{S}^{J}$ by $\Phi$, then the equivariance 
	of $\textbf{\textup{J}}$ can be reformulated in terms of the map 
	$\psi\,:\,\mathfrak{g}^{J}\rightarrow \mathscr{K}(\mathbb{S}^{J})$ as follows: 
	\begin{eqnarray}
		\psi(\textup{Ad}(g^{-1})L)=\psi(L)\circ \Phi_{g},
	\end{eqnarray}
	where $g\in G^{J}(\mathbb{R})$ and $L \in \mathfrak{g}^{J}.$
\end{remark}

	One of the \textit{raison d'\^{e}tre} of the momentum map is the classification of all 
	homogeneous symplectic manifolds in terms of coadjoint orbits (up to coverings); this is 
	\textit{Kostant's Coadjoint Orbit Covering Theorem} (stated below). 
	For the convenience of the reader, we recall the main ingredients of this classification (see \cite{Marsden-Ratiu}).

	Let $M$ be a manifold acted upon by a Lie group $G$ with Lie algebra $\mathfrak{g}$. Given $\mu\in \mathfrak{g}^{*}$, the 
	coadjoint orbit of $G$ through $\mu$ is the subset 
	\begin{eqnarray}
		\textup{Orb}(\mu):=\big\{\textup{Ad}^{*}(g)(\mu)\in \mathfrak{g}^{*}\,\big\vert\,g\in G\big\},
	\end{eqnarray}
	where $\textup{Ad}^{*}\,:\,G\times \mathfrak{g}^{*}\rightarrow \mathfrak{g}^{*}$ is the coadjoint representation of $G.$ Being 
	an orbit, $\textup{Orb}(\mu)$ is automatically an immersed submanifold of $\mathfrak{g}^{*}$, and its 
	tangent bundle at $\alpha\in \textup{Orb}(\mu)$ can be identified with 
	$\{\textup{ad}^{*}(\xi)(\alpha)\in \mathfrak{g}^{*}\,\vert\,\xi\in \mathfrak{g}\}$, where 
	$\textup{ad}^{*}:\mathfrak{g}\times \mathfrak{g}^{*}\rightarrow \mathfrak{g}^{*}$ is defined by 
	$\langle \textup{ad}^{*}(\xi)(\alpha),\eta \rangle=\langle \alpha, [\xi,\eta]\rangle,$ $\xi,\eta\in \mathfrak{g}.$ Using this 
	identification, one defines a symplectic form on $\mathcal{O}:=\textup{Orb}(\mu)$ as follows: 
\begin{eqnarray}
	(\omega_{\mathcal{O}})_{\alpha}\big(\textup{ad}^{*}(\xi)(\alpha),\textup{ad}^{*}(\eta)(\alpha)\big):=
	\langle \alpha,[\xi,\eta]\rangle, 
\end{eqnarray}
	where $\alpha\in \mathcal{O}$ and $\xi,\eta\in \mathfrak{g}$. The symplectic form $\omega_{\mathcal{O}}$ is known as the 
	\textit{Kirillov-Kostant-Souriau symplectic form}.
\begin{theorem}[Kostant's Coadjoint Orbit Covering Theorem \cite{Kostant-orbits}]\label{lekklgjjgke}
	Let $(M,\omega)$ be a symplectic manifold and let $\Phi\,:\,G\times M\rightarrow M$ be a left and transitive action having 
	an equivariant momentum map $\textup{\textbf{J}}\,:\,M\rightarrow \mathfrak{g}^{*}$. Then $\textup{\textbf{J}}$ is a 
	local diffeomorphism onto a coadjoint orbit $\mathcal{O}$, and it satisfies $\textup{\textbf{J}}^{*}\omega_{\mathcal{O}}=\omega.$
\end{theorem}
	Returning to the Siegel-Jacobi space $\mathbb{S}^{J}$, we have the following result which is a complement of 
	Corollary \ref{kgkfg,rkg,}.
\begin{proposition}\label{que direnkdnkgdgn}
	The momentum map $\textup{\textbf{J}}\,:\,\mathbb{S}^{J}\rightarrow (\mathfrak{g}^{J})^{*}$ is a diffeomorphism 
	onto a coadjoint orbit $\mathcal{O}$, and it satisfies $\textup{\textbf{J}}^{*}\omega_{\mathcal{O}}=\omega_{KB},$ 
	where $\omega_{KB}$ is the K\"{a}hler-Berndt symplectic form. In other words, the Siegel-Jacobi space 
	$\mathbb{S}^{J}$ (regarded as a symplectic manifold) 
	is a coadjoint orbit of the Jacobi group $G^{J}(\mathbb{R}).$ 
\end{proposition}
\begin{proof}
	By application of Theorem \ref{lekklgjjgke}, it suffices to show that 
	$\textup{\textbf{J}}\,:\,\mathbb{S}^{J}\rightarrow (\mathfrak{g}^{J})^{*}$ is injective, 
	or equivalently, to show that given two points $p,q\in \mathbb{S}^{J}$,
	\begin{eqnarray}\label{dl,flgv,g,rkg,rkg,rk}
		f(p)=f(q)\,\,\,\,\,\textup{for all}\,\,\,\,\,f\in \mathscr{K}(\mathbb{S}^{J})\,\,\,\,\,\,\Rightarrow\,\,\,\,\,\, p=q. 
	\end{eqnarray}
	This can be seen using \eqref{flfelfjlefjoe}-\eqref{lelgjrgjrl}. 
\end{proof}
\begin{remark}
	In \cite{Molitor-exponential}, we defined the \textit{K\"{a}hlerification} of an exponential family $\mathcal{E}$ as the quotient 
	$\mathcal{E}^{\mathbb{C}}:=T\mathcal{E}/\Gamma(\mathcal{E})$, 
	where $\Gamma(\mathcal{E})$ is the subgroup of $\textup{Diff}(T\mathcal{E})$ defined by 
	\begin{eqnarray}
		\Gamma(\mathcal{E}):=\big\{\phi\in \textup{Diff}(T\mathcal{E}) \,\big\vert\, \phi^{*}g=g,\,\,\phi_{*}J=J\phi_{*}\,\,\,
		\textup{and}\,\,\,\,f\circ \phi=f\,\,\,\textup{for all}\,\,\,f\in \mathscr{K}(T\mathcal{E})\,\big\},
	\end{eqnarray}
	where $(g,J)$ is the natural K\"{a}hler structure of $T\mathcal{E}$, as described in Section \ref{section expo families}. 
	If $\Gamma(\mathcal{E})$ is discrete and if its natural action on $T\mathcal{E}$ is free and proper, then 
	$\mathcal{E}^{\mathbb{C}}$ is a K\"{a}hler manifold in a natural way. In the case $\mathcal{E}=\mathcal{N}$, it follows from 
	\eqref{dl,flgv,g,rkg,rkg,rk} that $\Gamma(\mathcal{N})$ is trivial. 
	Therefore, the K\"{a}hlerification of $\mathcal{N}$ is the Siegel-Jacobi space $\mathbb{S}^{J}$, that is, 
	$\mathcal{N}^{\mathbb{C}}\cong \mathbb{S}^{J}$. 
\end{remark}
	We now discuss the spectral theory of the K\"{a}hler functions of $\mathbb{S}^{J}$ (in a sense to be discussed below).
	Let $\mathfrak{a}$ be the abelian Lie subalgebra of $\mathfrak{g}^{J}$ generated by $F,Q,R$, i.e.,
	\begin{eqnarray}
		\mathfrak{a}:=\textup{Vect}_{\mathbb{R}}\{F,Q,R\}.
	\end{eqnarray}
	In what follows, we shall identify $\mathfrak{a}$ with the space $\mathscr{P}_{2}(\mathbb{R})$ of 
	polynomials in one variable of degree $\leq2$ with real coefficients, via the isomorphism
	\begin{eqnarray}\label{lksfldnn}
		F\mapsto -x^{2},\,\,\,\,\,Q\mapsto x,\,\,\,\,\,\,R\mapsto -\tfrac{1}{4}. 
	\end{eqnarray}
	Thus, an arbitrary element of $\mathfrak{a}\cong \mathscr{P}_{2}(\mathbb{R})$ can be written as $k(x)=\alpha x^{2}+\beta x+\gamma$, 
	where $\alpha, \beta, \gamma\in \mathbb{R}.$ We also introduce the following subgroup of $G^{J}(\mathbb{R})$:
	\begin{eqnarray}
		B:=\Big\{\Big(\Big[\begin{matrix}
		a & b\\
		0 & a^{-1}
		\end{matrix}
		\Big],(\lambda,\mu,\kappa)\Big)\,\Big\vert\,a,b,\lambda, \mu,\kappa\in \mathbb{R},\,\,a\neq 0\Big\}. 
	\end{eqnarray}
	The group $B$ is a maximal closed, connected and solvable subgroup of $G^{J}(\mathbb{R})$, i.e., 
	it is a Borel subgroup of $G^{J}(\mathbb{R})$ (see \cite{Berndt98}). For $b=\big(\big[\begin{smallmatrix}
		a & b\\
		0 & a^{-1}
		\end{smallmatrix}
		\big],(\lambda,\mu,\kappa)\big)\in B$ and $x\in \mathbb{R}$, the formula 
	\begin{eqnarray}\label{cefklkpr3pop}
		x\cdot \Big(\Big[\begin{matrix}
		a & b\\
		0 & a^{-1}
		\end{matrix}
		\Big],(\lambda,\mu,\kappa)\Big):=ax-\frac{\lambda}{2}
	\end{eqnarray}
	defines a right action of $B$ on $\mathbb{R}$. Therefore, $B$ also acts on the left on $\mathscr{P}_{2}(\mathbb{R})$ 
	via the formula $b\cdot k(x):=k(x\cdot b)$, where $b\in B.$ 
\begin{lemma}\label{sdçkgvfrikpeopw}
	\begin{description}
	\item[$(i)$] Let $\textup{Ad}\,:\,G^{J}(\mathbb{R})\times \mathfrak{g}^{J}\rightarrow \mathfrak{g}^{J}$ 
	be the adjoint representation of $G^{J}(\mathbb{R}).$ Then, 
	\begin{eqnarray}
		\textup{Ad}(M,X,\kappa)\cdot (A,\xi,r)=\big(MAM^{-1},XAM^{-1}+\xi M^{-1},r-2\Omega(\xi,X)-\Omega(XA,X)\big),
	\end{eqnarray}
	where $M\in \textup{SL}(2,\mathbb{R})$, $A\in \mathfrak{sl}(2,\mathbb{R})$, $X,\xi\in \mathbb{R}^{2}$ and $\kappa,r\in \mathbb{R}$. 

	\item[$(ii)$] For $k(x)\in \mathfrak{a}$ and $g\in G^{J}(\mathbb{R}),$ we have: 
	\begin{eqnarray}
		\textup{Ad}(g)\,k(x)\in \mathfrak{a}\,\,\,\,\,\,\Leftrightarrow\,\,\,\,\,\,g\in B 
		\,\,\,\,\,\,\textup{or}\,\,\,\,\,\,k(x)\,\,\,\textup{is a constant polynomial}.
	\end{eqnarray}
	In particular, $\textup{Ad}(b)\mathfrak{a}\subseteq \mathfrak{a}$ for all $b\in B.$ Moreover, if $k(x)$ is a constant polynomial, 
	then $\textup{Ad}(g)k(x)=k(x)$ for all $g\in G^{J}(\mathbb{R})$.
	\item[$(iii)$] For $b\in B$ and $k(x)\in \mathfrak{a}$, we have:
	\begin{eqnarray}
		\textup{Ad}(b)k(x)=k(x\cdot b)
	\end{eqnarray}
	(here $\textup{Ad}$ is the adjoint representation of $G^{J}(\mathbb{R})$). 
	\end{description}
\end{lemma}
\begin{proof}
	The first item follows from a direct calculation while $(ii)$ and $(iii)$ are easily obtained from the matrix representation 
	of the restriction of $\textup{Ad}(M,X,\kappa)$ to $\mathfrak{a}$ relative to the basis 
	$\{F,Q,R\}$ and $\{F,G,H,P,Q,R\}$. As a simple calculation shows, this matrix is:
	\begin{eqnarray}\label{ld,frjgrkgjrk}
	\left[\begin{smallmatrix}
		a^{2}      & 0        &  0\\
		-c^{2}     & 0        &  0\\
		-ac        & 0        &  0\\
		-c\lambda  & -c       &  0\\
		a\lambda   & a        &  0\\
		\lambda^{2}& 2\lambda &  1
	\end{smallmatrix}\right],\,\,\,\,\,\,\,
	\textup{where}\,\,\,\,M=\Big[\begin{matrix}
		a & b\\
		c & d
		\end{matrix}\Big] \in \textup{SL}(2,\mathbb{R}),\,\,\,\,\textup{and}\,\,\,X=(\lambda,\mu)\in \mathbb{R}^{2}.  
	\end{eqnarray}
	From this, one easily concludes the proof. 
\end{proof}
\begin{lemma}\label{rdkgjdkfg}
	For $g_{1},g_{2}\in G^{J}(\mathbb{R})$ and $k_{1}(x),k_{2}(x)\in \mathfrak{a}$, we have:
	\begin{eqnarray}
		\textup{Ad}(g_{1})k_{1}(x)=\textup{Ad}(g_{2})k_{2}(x)\,\,\,\,\,\,\Rightarrow\,\,\,\,\,\,\,
		\textup{Im}(k_{1})=\textup{Im}(k_{2}),
	\end{eqnarray}
	where $\textup{Im}(k_{i})$ is the image of the polynomial $k_{i}(x)$ 
	(regarded as a function $k_{i}\,:\,\mathbb{R}\rightarrow \mathbb{R}$).
\end{lemma}
\begin{proof}
	If $\textup{Ad}(g_{1})k_{1}(x)=\textup{Ad}(g_{2})k_{2}(x)$, then $\textup{Ad}((g_{2})^{-1}g_{1})k_{1}(x)=k_{2}(x)$ and according to 
	Lemma \ref{sdçkgvfrikpeopw}, $(g_{2})^{-1}g_{1}\in B$ or $k_{1}(x)=\textup{constant}$. If $(g_{2})^{-1}g_{1}\in B$, then there exists 
	$b\in B$ such that $g_{1}=g_{2}b$, and we have, taking into account Lemma \ref{sdçkgvfrikpeopw},
	\begin{eqnarray}
		\textup{Ad}(g_{1})k_{1}(x)=\textup{Ad}(g_{2})k_{2}(x)\,\,\,\,\,\,&\Rightarrow&\,\,\,\,\,\,\,
		\textup{Ad}(g_{2})\textup{Ad}(b)k_{1}(x)=\textup{Ad}(g_{2})k_{2}(x),\nonumber\\
		&\Rightarrow& \,\,\,\,\,\,\textup{Ad}(b)k_{1}(x)=k_{2}(x),\nonumber\\
		&\Rightarrow& \,\,\,\,\,\,k_{1}(x\cdot b)=k_{2}(x),\nonumber\\
		&\Rightarrow& \,\,\,\,\,\,\textup{Im}(k_{1})=\textup{Im}(k_{2}).
	\end{eqnarray}
	In the case $k_{1}(x)=\textup{constant}$, Lemma \ref{sdçkgvfrikpeopw} implies that $k_{1}(x)=\textup{Ad}(g)k_{1}(x)$ for all 
	$g\in G^{J}(\mathbb{R})$. Consequently, $k_{1}(x)=
	\textup{Ad}((g_{2})^{-1}g_{1})k_{1}(x)=k_{2}(x)$, that is, $k_{1}(x)=k_{2}(x).$ 
	The lemma follows. 
\end{proof}
\begin{definition}[Spectrum of a K\"{a}hler function]\label{denfnk,dskndjefjefne}
	The spectrum of a K\"{a}hler function $f\in \mathscr{K}(\mathbb{S}^{J})$ 
	of the form $f=\textbf{\textup{J}}^{\textup{Ad}(g)k(x)}$, 
	where $g\in G^{J}(\mathbb{R})$ and $k(x)\in \mathfrak{a}$, is the following subset of $\mathbb{R}:$
	\begin{eqnarray}
		\textup{Spec}(f):=\textup{Im}(k),
	\end{eqnarray}
	where $\textup{Im}(k)$ is the image of the polynomial $k(x)$ (regarded as a function $k\,:\,\mathbb{R}\rightarrow \mathbb{R}$).
\end{definition}
\begin{remark}
	Not every K\"{a}hler function $f\in \mathscr{K}(\mathbb{S}^{J})$ can be 
	written as $f=\textbf{\textup{J}}^{\textup{Ad}(g)k(x)}$ (consider $\textbf{\textup{J}}^{H}$ for example). 
	Therefore, not every K\"{a}hler function $f=\textbf{\textup{J}}^{L}$ possesses a spectrum. But if it does, 
	Lemma \ref{rdkgjdkfg} guaranties that its spectrum is independent of the decomposition $L=\textup{Ad}(g)k(x)$ 
	(such decomposition is not unique in general). 
\end{remark}
\begin{remark}\label{ldjogjdks,xlsd}
	 Due to the equivariance of the momentum map 
	 $\textup{\textbf{J}}\,:\,\mathbb{S}^{J}\rightarrow (\mathfrak{g}^{J})^{*}$, one easily sees that 
	 $\textup{Spec}(f\circ \Phi_{g})=\textup{Spec}(f)$ for all $g\in G^{J}(\mathbb{R})$ 
	 (provided that $f\in \mathscr{K}(\mathbb{S}^{J})$ possesses a spectrum). 
\end{remark}
	In order to give a statistical meaning to the spectrum of a K\"{a}hler function $f\in \mathscr{K}(\mathbb{S}^{J})$, 
	let us recall the following facts:
	\begin{description}
	\item[$\bullet$] We have an identification of K\"{a}hler manifolds $\mathbb{S}^{J}\cong T\mathcal{N}$ 
		(see Proposition \ref{dlfld,f}), and consequently, the canonical projection 
		$T\mathcal{N}\rightarrow \mathcal{N}$ gives a projection 
		$\mathbb{S}^{J}\rightarrow \mathcal{N}$ that we shall also denote by $\pi.$ Thus, 
		for every $p\in \mathbb{S}^{J}$, $\pi(p)$ is 
		a Gaussian distribution function over $\mathbb{R}$. If $dx$ denotes the Lebesgue measure, 
		then $\pi(p)(x)dx$ is the associated probability measure (here we denote by $x$ the variable living in the 
		measure space $(\mathbb{R},dx)$).  
	\item[$\bullet$] The expectation parameters $\eta_{1},\eta_{2}\,:\,\mathcal{N}\rightarrow \mathbb{R}$ 
		are by definition the expectations (in the probabilistic sense) of the random variables $x$ and $x^{2}$ over $\mathbb{R}$ with 
		respect to the probability measures $p(x)dx$ ($p\in \mathcal{N}$), that is, 
		$\eta_{1}(p):=\int_{-\infty}^{\infty}\,xp(x)dx$ and $\eta_{2}(p)=
		\int_{-\infty}^{\infty}\,x^{2}p(x)dx$ (see \eqref{fedjfdkgjrkgtr} and 
		\eqref{equation reecriture normal}). 
	\item[$\bullet$] We have identified the vectors $F,Q,R\in \mathfrak{g}^{J}$ 
		with the polynomials $-x^{2},x$ and $-\tfrac{1}{4}$, respectively (see \eqref{lksfldnn}), and we have 
		$\textup{\textbf{J}}^{F}=-\eta_{2}\circ \pi$, 
		$\textup{\textbf{J}}^{Q}=\eta_{1}\circ \pi$ and $\textup{\textbf{J}}^{R}=-\tfrac{1}{4}$ (see \eqref{flfelfjlefjoe}). 
	\end{description}
	Let us denote by $\Phi$ the action of $G^{J}(\mathbb{R})$ on $\mathbb{S}^{J}$, and let $f$ be a K\"{a}hler function of the 
	form $f=\textbf{\textup{J}}^{\textup{Ad}(g)k(x)}$, where $k(x)=\alpha x^{2}+\beta x+\gamma\in 
	\mathfrak{a}$ and $g\in G^{J}(\mathbb{R})$. Using the equivariance of 
	$\textup{\textbf{J}}\,:\,\mathbb{S}^{J}\rightarrow (\mathfrak{g}^{J})^{*}$, one sees that 
	\begin{eqnarray}
		f(p)&=&\textup{\textbf{J}}^{\textup{Ad}(g)k(x)}(p)=(\textup{\textbf{J}}^{k(x)}\circ \Phi_{g^{-1}})(p)=
		(\textup{\textbf{J}}^{-\alpha F+\beta Q-4\gamma R}\circ \Phi_{g^{-1}})(p)\nonumber\\
		&=&\big[(\alpha \eta_{2}+\beta \eta_{1}+\gamma)\circ \pi\circ \Phi_{g^{-1}}\big](p)
		=\int_{-\infty}^{\infty}\,(\alpha x^{2}+\beta x+\gamma) \big[(\pi\circ \Phi_{g^{-1}})(p)\big](x)dx,
	\end{eqnarray}
	where $p\in \mathbb{S}^{J}.$ We thus have proved the following ``spectral decomposition" result. 
\begin{proposition}\label{ffkelgkrltikelrkflkr}
	Let $f\in \mathscr{K}(\mathbb{S}^{J})$ be a K\"{a}hler function of the form $f=\textbf{\textup{J}}^{\textup{Ad}(g)k(x)}$, 
	where $g\in G^{J}(\mathbb{R})$ and $k(x)=\alpha x^{2}+\beta x+\gamma\in \mathfrak{a}$. Then,
	\begin{eqnarray}
		f(p)=\int_{-\infty}^{\infty}\,(\alpha x^{2}+\beta x+\gamma) \big[(\pi\circ \Phi_{g^{-1}})(p)\big](x)dx
	\end{eqnarray}
	 for all $p\in \mathbb{S}^{J}$. 
\end{proposition}
	Therefore, a K\"{a}hler function of the form 
	$\textbf{\textup{J}}^{\textup{Ad}(g)k(x)}$ is simply the expectation of the polynomial $k(x)=\alpha x^{2}+\beta x+\gamma$ 
	with respect to the probability measure $\big[(\pi\circ \Phi_{g^{-1}})(p)\big](x)dx,$ and its spectrum 
	is the set of all possible expectations. 
\begin{example}\label{Que dire? ,f,grkg,tkg,}
	Using the matrix representation of $\textup{Ad}(g)$ given in \eqref{ld,frjgrkgjrk} together with the invariance 
	property of $\textup{Spec}$ (see Remark \ref{ldjogjdks,xlsd}),
	it is not difficult to see that 
	\begin{alignat}{4}
	\textup{Spec}(\textbf{\textup{J}}^{F})\,\,=&\,\,(-\infty,0\,], \quad & \quad 
	\textup{Spec}(\textbf{\textup{J}}^{G})\,\,=&\,\,[\,0,\infty), \quad&\quad 
	\textup{Spec}(\textbf{\textup{J}}^{P})\,\,=&\,\,(-\infty,\infty), \\
	\textup{Spec}(\textbf{\textup{J}}^{Q})\,\,=&\,\,(-\infty,\infty), \quad & \quad 
	\textup{Spec}(\textbf{\textup{J}}^{R})\,\,=&\,\,\{-\tfrac{1}{4}\}. \quad&\quad 
	\end{alignat}
	As we already mentioned, $\textbf{\textup{J}}^{H}$ doesn't have a spectrum in the sense of Definition \ref{denfnk,dskndjefjefne}. 
\end{example}
	Following \cite{Molitor-exponential}, we want to associate to a K\"{a}hler function $f=\textbf{\textup{J}}^{\textup{Ad}(g)k(x)}$ 
	and a point $p\in \mathbb{S}^{J}$, a probability measure $P_{f,p}$ on $\textup{Spec}(f)$. 
	To this end, recall that the subgroup $B$ acts on 
	the right on $\mathbb{R}$ as follows (see \eqref{cefklkpr3pop}) : $\Psi_{g}(x):=x\cdot g=ax-\tfrac{\lambda}{2}$, where 
	$g=\big(\big[\begin{smallmatrix}
		a & b\\
		0 & a^{-1}
		\end{smallmatrix}
		\big],(\lambda,\mu,\kappa)\big)\in B$ and $x\in \mathbb{R}$. With this notation, we have the following lemma.
\begin{lemma}\label{fkjdsjksgjsd}
	Let $p\in \mathbb{S}^{J}$ be such that $\pi(p)$ is the Gaussian distribution function of mean $\mu$ and deviation $\sigma$, 
	that is, $\pi(p)(x)=\tfrac{1}{(2\pi)^{1/2}\sigma}\textup{exp}
		\big\{-\tfrac{(x-\mu)^{2}}{2\,\sigma^{2}}\big\}$, $x\in \mathbb{R}$. Let $g=\big(\big[\begin{smallmatrix}
		a & b\\
		0 & a^{-1}
		\end{smallmatrix}
		\big],(\lambda,\mu,\kappa)\big)\in B$ be arbitrary. Then,
	\begin{description}
		\item[$(i)$] $(\pi\circ \Phi_{g})(p)$ is the Gaussian distribution function of mean 
			$\mu'=(\tfrac{\lambda}{2}+\mu)/a$ and deviation $\sigma'=\tfrac{\sigma}{\vert a\vert}$.
		\item[$(ii)$] If $dx$ is regarded as the Riemannian volume form of the Euclidean metric on $\mathbb{R}$, then, 
			\begin{eqnarray}
				\Psi_{g}^{*}(\pi(p)dx)=\varepsilon(g)\cdot (\pi\circ \Phi_{g})(p)dx,
			\end{eqnarray}
		where $\Psi_{g}^{*}$ is the pull-back operator on differential forms, and where $\varepsilon(g)=1$ if 
		$\Psi_{g}$ is orientation preserving and $-1$ otherwise. 
	\end{description}
\end{lemma}
\begin{proof}
	The first item can be easily obtained by remembering the various identifications and changes of variables me made: 
	\begin{description}
		\item[$\bullet$] $\theta_{1}=\tfrac{\mu}{\sigma^{2}},$\,\,\,\,\,$\theta_{2}=-\tfrac{1}{2\sigma^{2}}$ 
			(see \eqref{equation reecriture normal}),
		\item[$\bullet$] $T\mathcal{N}\cong \mathbb{C}\times i\mathbb{H}$ by means of the complex coordinates 
			$z_{1}=\theta_{1}+i\dot{\theta}_{1}$ and $z_{2}=\theta_{2}+i\dot{\theta}_{2}$, 
		\item[$\bullet$] $\mathbb{S}^{J}=\mathbb{H}\times \mathbb{C}$, and we have the identification 
			$\mathbb{C}\times i\mathbb{H}\cong \mathbb{H}\times \mathbb{C}$ via 
			the map $(z_{1},z_{2})\mapsto (-iz_{2},iz_{1})$,
		\item[$\bullet$] the action of $B$ on $\mathbb{H}\times \mathbb{C}$ is explicitely given by $\big(\big[\begin{smallmatrix}
		a & b\\
		0 & a^{-1}
		\end{smallmatrix}
		\big],(\lambda,\mu,\kappa)\big)\cdot (\tau,z)=\big(a(a\tau+b), a(z+\lambda \tau+\mu)\big).$
	\end{description}
	The second item is an easy consequence of $(i)$ together with the fact that 
	$\Psi_{g}^{*}(\pi(p)dx)=(\pi(p)\circ \Psi_{g})\Psi^{*}dx=(\pi(p)\circ \Psi_{g})(adx)$. The lemma follows. 
\end{proof}
	A direct consequence of Lemma \ref{sdçkgvfrikpeopw} and Lemma \ref{fkjdsjksgjsd} is that if 
	$\textbf{\textup{J}}^{\textup{Ad}(g_{1})k_{1}(x)}=\textbf{\textup{J}}^{\textup{Ad}(g_{2})k_{2}(x)}$, 
	where $g_{1},g_{2}\in G^{J}(\mathbb{R})$ and $k_{1}(x),k_{2}(x)\in \mathfrak{a}\cong \mathscr{P}_{2}(\mathbb{R})$, then 
	the probability distribution functions of $k_{1}(x)$ and $k_{2}(x)$ with respect to 
	$\big[(\pi\circ \Phi_{g_{1}^{-1}})(p)\big](x)dx$
	and $\big[(\pi\circ \Phi_{g_{2}^{-1}})(p)\big](x)dx$ are equal. 	
\begin{definition}[Spectral measure]\label{eg,,kd,zld,ek}
	Let $f\in \mathscr{K}(\mathbb{S}^{J})$ be a K\"{a}hler function of the form 
	$f(p)=\int_{-\infty}^{\infty}\,k(x)\big[(\pi\circ \Phi_{g^{-1}})(p)\big](x)dx$, where $k(x)\in \mathscr{P}_{2}(\mathbb{R})$ and 
	$g\in G^{J}(\mathbb{R})$. For $p\in \mathbb{S}^{J}$, the spectral measure $P_{f,p}$ is the 
	probability distribution functions of $k(x)$ with respect to 
	$\big[(\pi\circ \Phi_{g_{1}^{-1}})(p)\big](x)dx,$
	that is, 
	\begin{eqnarray}
		P_{f,p}(A):=\int_{k^{-1}(A)}\, [(\pi\circ \Phi_{g^{-1}})(p)](x)dx,
	\end{eqnarray}
	where $A\subseteq \textup{Spec}(f)$ is a measurable subset.
\end{definition}
	From a quantum mechanical point of view, 
	the quantity $P_{f,p}(A)$ is interpreted as 
	the probability that the observable $f\in \mathscr{K}(\mathbb{S}^{J})$ yields upon measurement an eigenvalue 
	$\lambda\in A\subseteq \textup{Spec}(f)$ while the system is in the state $p\in \mathbb{S}^{J}$.

\section{Gaussian distributions: extrinsic geometry}\label{section extrinsic}

	Let $\mathcal{H}:=L^{2}(\mathbb{R})$ be the Hilbert space of square integrable functions $f\,:\,\mathbb{R}\rightarrow \mathbb{C}$ 
	endowed with the Hermitian product $\langle f,g \rangle:=\int_{\mathbb{R}}\,\bar{f}gdx,$ where $dx$ is the Lebesgue measure. 
	Associated to it is the complex projective space $\mathbb{P}(\mathcal{H}):=(\mathcal{H}-\{0\})/\sim$, where 
	the equivalence relation is defined by 
	\begin{eqnarray}
		f\sim g\,\,\,\,\,\,\Leftrightarrow \,\,\,\,\,\, \exists\,\lambda\in \mathbb{C}-\{0\}\,\,:\,\,f=\lambda g. 
	\end{eqnarray}
	We denote by $[f]$ the equivalence class of $f\in \mathcal{H}-\{0\}$, that is, $[f]=\mathbb{C}{\cdot} f.$ 
	In this section, we shall regard the Siegel-Jacobi space $\mathbb{S}^{J}$ as a subspace of $\mathbb{P}(\mathcal{H})$ 
	via the injection 
	\begin{eqnarray}
		T\,:\,\mathbb{S}^{J}\hookrightarrow \mathbb{P}(\mathcal{H}),\,\,\,\,T(\tau,z):=
		\Big[e^{\tfrac{i}{2}\displaystyle(\tau x^{2}-z x)}\Big],
	\end{eqnarray}
	where $(\tau,z)\in \mathbb{H}\times \mathbb{C}\cong \mathbb{S}^{J}$, and where $x\in \mathbb{R}$. 
\subsection{Symplectic immersion}\label{zf,kf,kef,k}
	Let us recall a few facts related to the K\"{a}hler structure of $\mathbb{P}(\mathcal{H})$. 
	Given $f\in \mathcal{H}$ such that 
	$\|f\|^{2}=\langle f,f\rangle=1$, we can define a chart $(U_{f},\phi_{f})$ of $\mathbb{P}(\mathcal{H})$ by letting
	\begin{equation}
	\left \lbrace
	\begin{alignedat}{5}\label{equation definition carte projective}  
		U_{f}\,:&=\,\Big\{[g]\in \mathbb{P}(\mathcal{H})\,\big\vert\, 
		[f]\cap [g]=\{0\}\Big\}\,,\\[0.3em]
		\phi_{f}\,:&\,\,\,\,U_{f}\rightarrow [f]^{\perp}\subseteq \mathcal{H}\,,\,\,\,
		[g]\mapsto\dfrac{1}{\langle f,g\rangle}\cdot g-f\,,
	\end{alignedat}
	\right.
	\end{equation}
	where $[f]^{\perp}:=\big\{g\in\mathcal{H}\,\big\vert\,\langle f,g\rangle=0\big\}$. 
	If $f$ varies among all the unit vectors in $\mathcal{H}\,,$ then the corresponding 
	charts $(U_{f},\phi_{f})$ form an atlas for $\mathbb{P}(\mathcal{H})$ which becomes an infinite dimensional manifold. 

	The Fubini-Study metric $g_{FS}$ and the Fubini-Study symplectic form $\omega_{FS}$ 
	are now characterized as follows. Let $B:=\big\{f\in \mathcal{H}\,\big\vert\,
	\langle f,f\rangle=1\big\}$ be the unit ball with inclusion map $j\,:\,B\hookrightarrow \mathcal{H}$. 
	We denote by $\pi\,:\,B\rightarrow \mathbb{P}(\mathcal{H})$ the projection induced by the action of the circle 
	$S^{1}:=\{e^{i\theta}\,\vert\theta\in \mathbb{R}\}$ on $B$ (the action 
	being $e^{i\theta}\cdot f:=f\,e^{i\theta}$).
	Regarded as a real vector space, it is known that $\mathcal{H}$ is a K\"{a}hler manifold whose symplectic form (resp. metric)
	is the imaginary part (resp. real part) of the Hermitian inner product $\langle\,,\,\rangle$, and we have 
	(see \cite{Chernoff-Marsden}) : 
	\begin{eqnarray}\label{lf,lsdslfngjb}
		\pi^{*}\omega_{FS}=j^{*}\,\textup{Im}(\langle\,,\,\rangle),\,\,\,\,\,\,
		\pi^{*}g_{FS}=j^{*}\,\textup{Real}(\langle\,,\,\rangle).
	\end{eqnarray}
	Since $\pi$ is a submersion, these formulas characterize the Fubini-Study symplectic form and the Fubini-Study 
	metric\footnote{Depending on the convention, the Fubini-Study metric and symplectic form may appear in the 
	literature multiplied by a positive 
	constant.}.

	Having this in mind, let us return to the properties of the map $T(\tau,z)=\Big[e^{\tfrac{i}{2}\displaystyle(\tau x^{2}-z x)}\Big].$ 
\begin{proposition}\label{fejfkegjkegfe4jk}
	The map $T\,:\,\mathbb{S}^{J}\hookrightarrow \mathbb{P}(\mathcal{H})$ is a smooth immersion satisfying 
	\begin{eqnarray}\label{dslfkfkdgvfn}
		T^{*}\omega_{FS}=\tfrac{1}{4}\omega_{KB}\,\,\,\,\,\,\,\,\,\,\,\,\,\,\,\,\textup{and}
		\,\,\,\,\,\,\,\,\,\,\,\,\,\,\,\,T^{*}g_{FS}=\tfrac{1}{4}g_{KB}+\tfrac{1}{4}S,
	\end{eqnarray}
	where $S$ is the tensor field of symetric bilinear forms on $\mathbb{S}^{J}$ 
	whose matrix representation in the coordinates $(\theta,\dot{\theta})$ is
	\begin{eqnarray}
	S(\theta,\dot{\theta}):=
	\begin{bmatrix}
			0 & 0\\
			0 & \eta_{i}\eta_{j}
		\end{bmatrix}
	\end{eqnarray}
	(here $\eta_{i}$, $i=1,2$, are the expectation parameters of $\mathcal{N}$). 
\end{proposition}
\begin{remark}
	It follows from \eqref{dslfkfkdgvfn} that $T$ is a symplectic map\footnote{Let $(M_{1},\omega_{1})$ and $(M_{2},\omega_{2})$ be two 
	symplectic manifolds. A smooth map $f\,:\,M_{1}\rightarrow M_{2}$ is \textit{symplectic} if $f^{*}\omega_{2}=\omega_{1}$. 
	If $f$ is a symplectic map, then its derivative $f_{*_{p}}\,:\,T_{p}M_{1}\rightarrow T_{f(p)}M_{2}$ is injective for every 
	$p\in M_{1}$ (including if $M_{2}$ is infinite dimensional).\label{,k,kdf,dkf,kf,ee}}, 
	but it not isometric nor holomorphic. 
\end{remark}
	Before showing Proposition \ref{fejfkegjkegfe4jk}, let us make a few remarks. The map $T$ has been defined above in terms of 
	the variables $(\tau,z)\in \mathbb{H}\times \mathbb{C}$, but in terms of the variables 
	$(z_{1},z_{2})=(-iz,i\tau)\in \mathbb{C}\times i\mathbb{H}$, it reads 
	\begin{eqnarray}\label{csf,kg,kgf,kee}
		T(z_{1},z_{2})=\Big[e^{\tfrac{1}{2}\displaystyle(z_{1}x+z_{2}x^{2})}\Big]=
		\Big[e^{\tfrac{1}{2}\displaystyle(\theta_{1}x+\theta_{2}x^{2})+
		\tfrac{i}{2}\displaystyle(\dot{\theta}_{1}x+\dot{\theta}_{2}x^{2})}\Big],
	\end{eqnarray}
	where $\theta_{k}$ are the natural parameters of $\mathcal{N}$ (in particular $z_{k}=\theta_{k}+i\dot{\theta}_{k}$, 
	see \eqref{eljlejgler} and Definition \ref{cspwpwpdkpw}). In order to use the unit ball in $\mathcal{H}=L^{2}(\mathbb{R})$, 
	we want to normalize the function withing bracket in \eqref{csf,kg,kgf,kee}. To this end, we introduce 
	the following map 
	\begin{eqnarray}\label{dfndkfndkfnd}
		\Psi\,:\,\mathbb{S}^{J}\rightarrow \mathcal{H},\,\,\,\,\,\Psi(z_{1},z_{2})(x)&:=&
		e^{\tfrac{1}{2}\displaystyle(\theta_{1}x+\theta_{2}x^{2}-\psi(\theta))+
		\tfrac{i}{2}\displaystyle(\dot{\theta}_{1}x+\dot{\theta}_{2}x^{2})}\nonumber\\
			&=& e^{\tfrac{1}{2}\displaystyle(z_{1}x+z_{2}x^{2}-\psi(\theta))}, 
	\end{eqnarray}
	where $\psi(\theta)=-\tfrac{(\theta_{1})^{2}}{4\theta_{2}}+\tfrac{1}{2}\textup{ln}\,
	\big(-\tfrac{\pi}{\theta_{2}}\big)$. By comparing \eqref{dfndkfndkfnd} with 
	the exponential-family form of the Gaussian distribution in \eqref{equation reecriture normal}, 
	one sees that $\Psi(z_{1},z_{2})$ is normalized, that is, $\langle \Psi(z_{1},z_{2}),\Psi(z_{1},z_{2})\rangle=1$ 
	for all $(z_{1},z_{2})\in \mathbb{C}\times i\mathbb{H}.$ Therefore, $\Psi$ can be regarded as a smooth
	map $\mathbb{S}^{J}\rightarrow B\subseteq \mathcal{H},$ and we have $T(z_{1},z_{2})=\big[\Psi(z_{1},z_{2})\big]$. 
\begin{proof}[Proof of Proposition \ref{fejfkegjkegfe4jk}]
	Taking into account Footnote \ref{,k,kdf,dkf,kf,ee} together with the 
	characterization of the Fubini-Study metric and symplectic form given above 
	(in terms of the unit ball $B\in \mathcal{H}$, see \eqref{lf,lsdslfngjb}), it suffices to show that 
	\begin{eqnarray}\label{fjdkfnkfnkfnkr}
		\big\langle \Psi_{*_{p}}A,\Psi_{*_{p}}B\big\rangle =4\big\{ g_{KB}(A,B)+i\omega_{KB}(A,B)+ S(A,B)\big\}
	\end{eqnarray}
	for all $p\in \mathbb{S}^{J}$ and all $A,B\in T_{p}\mathbb{S}^{J}$ (in the above formula it is understood that 
	$T_{\Psi(p)}\mathcal{H}\cong \mathcal{H}$). We work in the coordinates $(\theta,\dot{\theta}).$ Take 
	$p=(\theta_{1},\theta_{2},\dot{\theta}_{1},\dot{\theta}_{2})\in \mathbb{S}^{J}$ and choose 
	$A=(A_{1},A_{2},A_{3},A_{4})$ and $B=(B_{1},B_{2},B_{3},B_{4})$ in  $T_{p}\mathbb{S}^{J}$. Using the notation 
	\begin{alignat}{4}
		X_{1}\,:=&\,A_{1}+iA_{3},  \quad & \quad X_{2}\,:=&\,A_{2}+iA_{4},\quad & \quad 
		Y_{1}\,:=&\,B_{1}+iB_{3},  \quad & \quad Y_{2}\,:=&\,B_{2}+iB_{4}, 
	\end{alignat}
	we see that 
	\begin{eqnarray}
		\Psi_{*_{p}}A&=&\dfrac{d}{dt}\bigg\vert_{0}\,\Psi\big(\theta_{1}+tA_{1},\theta_{2}+tA_{2},\dot{\theta}_{1}+tA_{3},
		\dot{\theta}_{2}+tA_{4}\big)\nonumber\\
		&=&\dfrac{d}{dt}\bigg\vert_{0}\, 
		e^{\tfrac{1}{2}\displaystyle\big[(z_{1}+tX_{1})x+(z_{2}+tX_{2})x^{2}-\psi(\theta+tA)\big]}\nonumber\\
		&=& \frac{1}{2}\Big(X_{1}x+X_{2}x^{2}-\dfrac{\partial \psi}{\partial \theta_{1}}A_{1}-
		\dfrac{\partial \psi}{\partial \theta_{2}}A_{2}\Big)\cdot \Psi.
	\end{eqnarray}
	As a direct calculation shows, $\tfrac{\partial \psi}{\partial \theta_{1}}=\eta_{1}$ and 
	$\tfrac{\partial \psi}{\partial \theta_{2}}=\eta_{2}$ (see \eqref{equation reecriture normal} and \eqref{dsldgpppqqmdm,}), 
	and thus,  
	\begin{eqnarray}\label{ldlqldzbbcbc}
		\Psi_{*_{p}}A=\tfrac{1}{2}\big(X_{1}x+X_{2}x^{2}-\eta_{1}A_{1}-
		\eta_{2}A_{2}\big)\cdot \Psi,
	\end{eqnarray}
	from which it follows that 
	\begin{eqnarray}\label{ef,kdf,kfn,rkf}
		&&\big\langle\Psi_{*_{p}}A,\Psi_{*_{p}}B \big\rangle\nonumber\\
		&=&\frac{1}{4}\Big\langle \big(X_{1}x+X_{2}x^{2}-\eta_{1}A_{1}-
		\eta_{2}A_{2}\big)\cdot \Psi, \big(Y_{1}x+Y_{2}x^{2}-\eta_{1}B_{1}-
		\eta_{2}B_{2}\big)\cdot \Psi
		\Big\rangle \nonumber\\
		&=& \frac{1}{4}\int_{-\infty}^{\infty}\,(\overline{X}_{1}x+\overline{X}_{2}x^{2}-\eta_{1}A_{1}-
		\eta_{2}A_{2})(Y_{1}x+Y_{2}x^{2}-\eta_{1}B_{1}-
		\eta_{2}B_{2})\,p(x;\theta)dx\nonumber\\
		&=&\frac{1}{4}\int_{-\infty}^{\infty}\,\bigg[
		\overline{X}_{1}Y_{1}x^{2}+\overline{X}_{1}Y_{2}x^{3}-\overline{X}_{1}B_{1}x\eta_{1}
		-\overline{X}_{1}B_{2}x\eta_{2}+\overline{X}_{2}Y_{1}x^{3}+\overline{X}_{2}Y_{2}x^{4}-\overline{X}_{2}B_{1}x^{2}\eta_{1}
		\nonumber\\
		&&\,\,\,-\overline{X}_{2}B_{2}x^{2}\eta_{2}-A_{1}Y_{1}x\eta_{1}-A_{1}Y_{2}x^{2}\eta_{1}+A_{1}B_{1}(\eta_{1})^{2}+
		A_{1}B_{2}\eta_{1}\eta_{2}-A_{2}Y_{1}x\eta_{2}-A_{2}Y_{2}x^{2}\eta_{2}\nonumber\\
		&&\,\,\,+A_{2}B_{1}\eta_{1}\eta_{2}+A_{2}B_{2}(\eta_{2})^{2}\bigg]\,p(x;\theta)dx,
	\end{eqnarray}
	where $p(x;\theta):=e^{\displaystyle x\theta_{1}+x^{2}\theta_{2}-\psi(\theta)}$. To compute the above integral, 
	we use the following well-known result (see \cite{Amari-Nagaoka}) : 
	if $\mathcal{E}$ is an exponential family whose elements can be written 
	$p(x;\theta)=\textup{exp}\big\{C(x)+\sum_{i=1}^{n}\theta_{i}F_{i}(x)-\psi(\theta)\big\}$ 
	(as in Definition \ref{definition exp}), 
	then the components of the Fisher metric are 
	$(h_{F})_{ij}(\theta)=\mathbb{E}((F_{i}-\eta_{i})(F_{j}-\eta_{i}))$, where $\eta_{i}$ are the 
	expectation parameters, and where the expectation is taking with respect to the probability determined by $p(x;\theta)$. 
	In our case, $F_{1}(x)=x$ and $F_{2}(x)=x^{2}$, and thus, we easily see that for $i,j\in \{1,2\}$, 
	\begin{eqnarray}\label{zf,,fkjfkgjr}
		\int_{-\infty}^{\infty}\,x^{i+j}p(x;\theta)dx=(h_{F})_{ij}+\eta_{i}\eta_{j}. 
	\end{eqnarray}
	By separating the real and imaginary parts in \eqref{ef,kdf,kfn,rkf}, and taking into account \eqref{de,fkgjrkgj}, 
	\eqref{zf,,fkjfkgjr} together with the fact that $\eta_{1}(\theta)=\int_{-\infty}^{\infty}\,xp(x;\theta)dx$, 
	one exactly finds \eqref{fjdkfnkfnkfnkr}. The proposition follows. 
	\end{proof}

\subsection{Schr\"{o}dinger-Weil representation and quantum observables}
	
	Let $\textup{End}\big(C^{\infty}(\mathbb{R},\mathbb{C})\big)$ denotes the space of $\mathbb{C}$-linear endomorphisms 
	of $C^{\infty}(\mathbb{R},\mathbb{C})$, and let $\mathbf{Q}\,:\,\mathfrak{g}^{J}\rightarrow 
	\textup{End}\big(C^{\infty}(\mathbb{R},\mathbb{C})\big)$ 
	be the linear map 
\begin{alignat}{5}	
	F\,\,\mapsto&\,\, -x^{2},  \,\,\quad&\quad\,\,\label{lf,lddf,dlf,lef,dl}
	P\,\,\mapsto&\,\, -i\frac{\partial}{\partial x},\\
	G\,\,\mapsto&\,\,-\frac{\partial^{2}}{\partial x^{2}} , \,\,\quad&\quad\,\,
	Q\,\,\mapsto& \,\,x,\\
	H\,\,\mapsto& \,\,2i\Big(x\frac{\partial}{\partial x}+\frac{1}{2}I\Big),\,\,\quad&\quad\,\,
	R\,\,\mapsto&\,\, -\frac{1}{4}I\label{ddld,zledz}
\end{alignat}
	($I$ denotes the identity operator). In the above formulas, it is understood that $-x^{2}$ and $x$  
	act by multiplication. 
	Regarded as unbounded operators acting on $L^{2}(\mathbb{R})$ with appropriate domains, these operators are Hermitian. 
\begin{remark}\label{dkefkdgjrkgjrk}
	From a physical point of view, the operators 
	\begin{eqnarray}\label{ednfkmdksdmlkws}
	-\frac{\partial^{2}}{\partial x^{2}}=\mathbf{Q}(G),\,\,\,\,\,
	-\frac{\partial^{2}}{\partial x^{2}}+ax^{2}=\mathbf{Q}(G-aF),\,\,\,\,\,\,
	-\frac{\partial^{2}}{\partial x^{2}}+ax^{2}-bx=\mathbf{Q}(G-aF-bQ), 
	\end{eqnarray}
	where $a>0$ and $b\in \mathbb{R}$, are respectively the Hamiltonians of the free quantum particle, 
	the quantum harmonic oscillator and the (time-independent) quantum forced oscillator. The operators 
	$\mathbf{Q}(Q)=x$ and $\mathbf{Q}(P)=-i\frac{\partial}{\partial x}$ are the usual position and momentum operators. 
\end{remark}
\begin{proposition}\label{dz,dkf,ekf,e}
	We have 
	\begin{eqnarray}	
		[\mathbf{Q}(A),\mathbf{Q}(B)]:=2i\mathbf{Q}([A,B])
	\end{eqnarray}
	for all $A,B\in \mathfrak{g}^{J}.$ In particular, $-\tfrac{i}{2}\mathbf{Q}$ is a unitary representation of 
	the Lie algebra $\mathfrak{g}^{J}$.
\end{proposition}
\begin{proof}
	By a direct calculation using the commutation relations \eqref{kdlsdklwsd}-\eqref{kdlsdklwsd2}. 
\end{proof}
	In the literature, the representation $-\tfrac{i}{2}\mathbf{Q}$ is essentially known as the 
	\textit{infinitesimal Schr\"{o}dinger-Weil representation} (see \cite{Berndt98,Berceanu08}).
\begin{proposition}\label{ed;fndjfkdfejfn}
	For every $L\in \mathfrak{g}^{J}$ and every $p\in \mathbb{S}^{J}$, we have 
	\begin{eqnarray}
		\big\langle \Psi(p),\mathbf{Q}(L)\Psi(p)\big\rangle = \textup{\textbf{J}}^{L}(p),
	\end{eqnarray}
	where $\Psi\,:\,\mathbb{S}^{J}\rightarrow L^{2}(\mathbb{R})$ is the map introduced in \eqref{dfndkfndkfnd}. 
\end{proposition}
\begin{proof}
	By a direct verification using \eqref{flfelfjlefjoe} and \eqref{lf,lddf,dlf,lef,dl}--\eqref{ddld,zledz}. 
\end{proof}
\begin{remark}\label{f,kfekfrkgrkgrkg}
	Given an arbitrary Hilbert space $\mathcal{H}$ and a bounded\footnote{To 
	some extent, this is also true for unbounded self-adjoint operators 
	(see \cite{Cirelli-Hamiltonian}).} self-adjoint operator $H$, it is known that the function 
	$f_{H}([\psi]):= \tfrac{\langle \psi,H\psi\rangle }{\langle \psi,\psi\rangle}$ is a K\"{a}hler function 
	on the complex projective space $\mathbb{P}(\mathcal{H})$ (see \cite{Ashtekar,Cirelli-Quantum}). 
	Therefore, one can reformulate Proposition \ref{ed;fndjfkdfejfn} heuristically as follows: 
	every K\"{a}hler function on $\mathbb{S}^{J}$ extends as a K\"{a}hler function on $\mathbb{P}(\mathcal{H})$ via the map $T=[\Psi]$.  
\end{remark}
\begin{remark}
	Given $L\in \mathfrak{g}^{J}$, it would be interesting to compare the spectrum of the operator 
	$\mathbf{Q}(L)$ with that of $\textup{\textbf{J}}^{L}$ 
	(in the sense of Definition \ref{denfnk,dskndjefjefne}). In this paper we don't 
	treat this question, but the reader can easily see that $\textup{Spec}\big(\mathbf{Q}(L)\big)=
	\textup{Spec}\big(\textup{\textbf{J}}^{L}\big)$ for all $L\in \{P,Q,R,F,G\}$ (see Example \ref{Que dire? ,f,grkg,tkg,}). 
	It is also interesting to note, 
	in relation to the quantum harmonic oscillator, that the spectrum of the operator $\mathbf{Q}(G-aF)$ (see \eqref{ednfkmdksdmlkws}) is 
	discrete\footnote{It can be shown that the spectrum of $\mathbf{Q}(G-aF)=-\tfrac{\partial^{2}}{\partial x^{2}}+ax^{2}$ is 
	the set $\{(2n+1)\sqrt{a}\in \mathbb{R}\,\vert\,n=0,1,...\}$ (see \cite{Davies}).} and that $\textup{\textbf{J}}^{G-aF}$ 
	does not have a spectrum in the sense of Definition \ref{denfnk,dskndjefjefne}. 
\end{remark}

\subsection{Dynamics and the Schr\"{o}dinger equation}\label{final sectionnnn}
	Given $L\in \mathfrak{g}^{J}$, we denote by $X_{\textup{\textbf{J}}^{L}}$ the Hamiltonian vector field of the K\"{a}hler 
	function $\textbf{J}^{L}\,:\,\mathbb{S}^{J}\rightarrow \mathbb{R}$ with respect to the K\"{a}hler-Berndt symplectic 
	form $\omega_{KB}$.
\begin{proposition}\label{,fk,kf,ekf,ek}
	There exists a smooth map $\kappa\,:\,\mathbb{S}^{J}\times \mathfrak{g}^{J}\rightarrow \mathbb{C}$, linear in 
	the second entry, with the following property: 
	if $\alpha\,:\,I\rightarrow \mathbb{S}^{J}$ is an integral curve of the Hamiltonian vector field 
	$X_{\textup{\textbf{J}}^{L}}$, then $\psi(t):=\Psi\big(\alpha(t)\big)$ satisfies 
	\begin{eqnarray}\label{ded,kf,ekf,ekf,rk}
		i\dfrac{d\psi}{dt}=\frac{1}{2}\mathbf{Q}(L)\psi+\frac{1}{2}\kappa_{L}(t)\psi, 
	\end{eqnarray}
	where $\kappa_{L}(t):=\kappa(\alpha(t),L)$ and where 
	$\Psi\,:\,\mathbb{S}^{J}\rightarrow L^{2}(\mathbb{R})$ is the map introduced in \eqref{dfndkfndkfnd}.
\end{proposition}
\begin{proof}
	Given $p=(\theta,\dot{\theta})=(\eta,\dot{\theta})\in \mathbb{S}^{J}$, we define a linear 
	map $\mathfrak{g}^{J}\rightarrow \mathbb{C}$ as follows: 
	\begin{alignat}{5}	
	F\,\mapsto&\, 0,  \,  &\quad\,
	P\,\mapsto&\, i\big(\tfrac{\theta_{1}+i\dot{\theta}_{1}}{2}+\theta_{2}\eta_{1}\big),\\
	G\,\mapsto&\,i\eta_{1}(\dot{\theta}_{1}\theta_{2}+\theta_{1}\dot{\theta}_{2})+2i\eta_{2}\theta_{2}\dot{\theta}_{2}
	-\tfrac{1}{4}(\dot{\theta_{1}})^{2}+\theta_{2}+i\dot{\theta_{2}}+\tfrac{1}{4}(\theta_{1})^{2}
	+\tfrac{1}{2}i\theta_{1}\dot{\theta_{1}}, \, &\quad\,
	Q\,\mapsto& \,0,\\
	H\,\mapsto& \,-i(\eta_{1}\theta_{1}+2\eta_{2}\theta_{2}+1),\,&\quad\,
	R\,\mapsto&\, \frac{1}{4}
\end{alignat}
	In this way, one obtains a map $\kappa\,:\,\mathbb{S}^{J}\times \mathfrak{g}^{J}\rightarrow \mathbb{C}$ which is linear in 
	the second entry. Now, by a direct calculation using the proof of Proposition \ref{eknfkd,lsdnszz}, 
	\eqref{ldlqldzbbcbc} and the definition of $\mathbf{Q}$, 
	one sees that \eqref{ded,kf,ekf,ekf,rk} holds. The proposition follows.  
\end{proof}
\begin{corollary}\label{sfndkfnkdfnkd}
	Let $\alpha\,:\,I\rightarrow \mathbb{S}^{J}$ be an integral curve of the Hamiltonian vector field 
	$X_{\textup{\textbf{J}}^{L}}$, and let $F(t)$ be a primitive of $\kappa(\alpha(t),L)$ on $I$. Then 
	$\psi(t):=e^{\frac{i}{2}F(t)}\Psi\big(\alpha(t)\big)$ satisfies the Schr\"{o}dinger equation
	\begin{eqnarray}\label{fksfksnf,kdf,ek}
		i\dfrac{d\psi}{dt}=H\psi,
	\end{eqnarray}
	where $H:=\frac{1}{2}\mathbf{Q}(L).$
\end{corollary}
\begin{proof}
	Again by a direct verification using Proposition \ref{,fk,kf,ekf,ek}. 
\end{proof}

	\textbf{}\\
\textbf{Acknowledgements.} It is a pleasure to thank all my colleagues and friends from the Federal University of Bahia in Salvador for their 
	invaluable help during my postdoctoral stay. I would like in particular to thank Ana Lucia Pinheiro Lima for her 
	availability, professionalism and kindness.  

	This work was done with the financial support of the CNPq and CAPES.

\bibliography{bibala}
\end{document}